\documentclass[a4paper, 10pt]{amsart}

\usepackage[utf8]{inputenc}
\usepackage{amsmath, amssymb, amsfonts, mathtools}
\usepackage{mathrsfs}

\usepackage[dvips,%
    includehead,%
    includefoot,%
    nomarginpar,%
    lmargin=1.3in,%
    rmargin=1.3in,%
    tmargin=1.4in,%
    bmargin=1.4in,%
  ]{geometry}
\usepackage{xcolor}
\definecolor{lightblue}{rgb}{.1,0.35,.8}
\usepackage[colorlinks,
    linkcolor={lightblue},
    citecolor={lightblue},
    urlcolor={black}]{hyperref}

\usepackage{enumerate,enumitem}
\usepackage{tikz}
\usetikzlibrary{calc, matrix, arrows, cd, decorations.markings, knots}
\usetikzlibrary{decorations.pathmorphing}
\usetikzlibrary{decorations.pathreplacing}

\theoremstyle{plain}
\newtheorem{proposition}{Proposition}[section]
\newtheorem{theorem}{Theorem}[section]
\newtheorem{corollary}{Corollary}[section]
\newtheorem{lemma}{Lemma}[section]
\newtheorem{theoremalpha}{Theorem}
\newtheorem{corollaryalpha}{Corollary}
\newtheorem*{theoremA}{Theorem A}
\newtheorem*{theoremB}{Theorem B}

\newtheorem*{corollaryF}{Corollary F}

\theoremstyle{definition}
\newtheorem{definition}{Definition}[section]
\newtheorem{notation}{Notation}[section]
\newtheorem{question}{Question}[section]

\newtheorem{example}{Example}[section]

\theoremstyle{remark}
\newtheorem{remark}{Remark}[section]
\newtheorem{construction}{Construction}[section]

\makeatletter
\let\c@corollary=\c@theorem
\let\c@proposition=\c@theorem
\let\c@lemma=\c@theorem
\let\c@remark=\c@theorem
\let\c@definition=\c@theorem
\let\c@notation=\c@theorem
\let\c@construction=\c@theorem
\let\c@example=\c@theorem
\let\c@equation\c@theorem
\let\c@question\c@theorem
\makeatother

\def\makeautorefname#1#2{\expandafter\def\csname#1autorefname\endcsname{#2}}
\makeautorefname{theoremalpha}{Theorem}%
\makeautorefname{corollaryalpha}{Corollary}%
\makeautorefname{theoremalphaprime}{Theorem}%

\newcommand{\wt}{\widetilde}

\newcommand{\ol}{\overline}
\newcommand{\sm}{\setminus}

\newcommand{\Q}{\mathbb{Q}}
\newcommand{\Z}{\mathbb{Z}}

\newcommand{\C}{\mathbb{C}}
\newcommand{\R}{\mathbb{R}}

\newcommand{\CP}{\operatorname{\C P}}
\newcommand{\RP}{\operatorname{\R P}}

\newcommand{\im}{\operatorname{im}}

\newcommand{\Id}{\operatorname{Id}}
\newcommand{\GL}{\operatorname{GL}}

\newcommand{\Aut}{\operatorname{Aut}}
\newcommand{\Tor}{\operatorname{Tor}}
\newcommand{\Th}{\operatorname{Th}}

\newcommand{\Hom}{\operatorname{Hom}}

\newcommand{\CATSpin}{\operatorname{CATSpin}}
\newcommand{\SCAT}{\operatorname{SCAT}}
\newcommand{\BCATSpin}{\operatorname{BCATSpin}}
\newcommand{\BSCAT}{\operatorname{BSCAT}}

\newcommand{\BTOPSpin}{\operatorname{BTopSpin}}
\newcommand{\BSpin}{\operatorname{BSpin}}
\newcommand{\Spin}{\operatorname{Spin}}
\newcommand{\TOPSpin}{\operatorname{TopSpin}}
\newcommand{\STOP}{\operatorname{STop}}
\newcommand{\BSTOP}{\operatorname{BSTop}}
\newcommand{\BSTOPdot}{\operatorname{BSTop.}}
\newcommand{\BTOP}{\operatorname{BTop}}
\newcommand{\SO}{\operatorname{SO}}
\newcommand{\Orth}{\operatorname{O}}
\newcommand{\BO}{\operatorname{BO}}

\newcommand{\BSO}{\operatorname{BSO}}
\newcommand{\BSOdot}{\operatorname{BSO.}}
\newcommand{\CAT}{\operatorname{CAT}}
\newcommand{\Diff}{\operatorname{Diff}}
\newcommand{\TOP}{\operatorname{Top}}

\newcommand{\FV}{\mathcal{FV}}
\newcommand{\Homeo}{\operatorname{Homeo}}
\newcommand{\Diffeo}{\operatorname{Diffeo}}
\newcommand{\Emb}{\operatorname{Emb}}
\newcommand{\ev}{\operatorname{ev}}
\newcommand{\fix}{\operatorname{fix}}
\DeclareMathOperator\ad{ad}
\DeclareMathOperator\pt{pt}
\DeclareMathOperator\pr{pr}

\DeclareMathOperator\Sq{Sq}
\DeclareMathOperator\PD{PD}

\DeclareMathOperator\colim{colim}

\renewcommand{\epsilon}{\varepsilon}
\renewcommand{\phi}{\varphi}

\newcommand{\bsm}{\left(\begin{smallmatrix}}
\newcommand{\esm}{\end{smallmatrix}\right)}

\begin{document}
\title{Mapping class groups of simply connected 4-manifolds with boundary}

\author[Patrick Orson]{Patrick Orson}
\address{Department of Mathematics, California Polytechnic State University, USA}
\email{porson@calpoly.edu}

\author[Mark Powell]{Mark Powell}
\address{School of Mathematics and Statistics, University of Glasgow, UK}
\email{mark.powell@glasgow.ac.uk}

\def\subjclassname{\textup{2020} Mathematics Subject Classification}
\expandafter\let\csname subjclassname@1991\endcsname=\subjclassname

\begin{abstract}
We compute the topological mapping class group of every compact, simply connected, topological 4-manifold. This was previously only known in the closed case.  If the 4-manifold is smooth, we deduce an analogous description of the stable smooth mapping class group, extending work of Saeki.
\end{abstract}

\maketitle

%\setcounter{tocdepth}{1}
%\tableofcontents

\section{Introduction}

Given an oriented, topological manifold $X$, with (possibly empty) boundary, $\Homeo^+(X,\partial X)$ denotes the topological group of orientation preserving self-homeomorphisms that restrict to the identity on the boundary $\partial X$, with the compact-open topology. The set of connected components $\pi_0 \Homeo^+(X,\partial X)$ is the \emph{topological mapping class group} of $X$, the group of isotopy classes of orientation preserving self-homeomorphisms that fix the boundary pointwise.
We study topological mapping class groups for compact, oriented, simply connected $4$-manifolds.

Let $\lambda_X  \colon H_2(X) \times H_2(X) \to \Z$ be the intersection pairing of $X$. When $\partial X = \emptyset$, it was shown by  Perron~\cite{MR862426} and  Quinn~\cite[Theorem~1.1]{Quinn:isotopy} (cf.\ Kreck~\cite[Theorem~1]{MR561244}), that if two orientation preserving self-homeomorphisms of $X$ induce the same isometry of the intersection form then they are isotopic. Freedman~\cite[Theorem~1.5,~Addendum]{F} showed that every automorphism of the intersection form is induced by a homeomorphism. Therefore the results of Perron, Quinn, and Freedman combine to compute the mapping class group of every closed, simply connected 4-manifold, in the sense of reducing the problem to algebra:~$\pi_0\Homeo^+(X)\xrightarrow{\cong}\Aut(H_2(X),\lambda_X);\; F \mapsto F_*$.

When $X$ has nonempty boundary, we need to consider a refinement of $\Aut(H_2(X),\lambda_X)$ to capture the algebraic data of a homeomorphism. A map $F\in\Homeo^+(X,\partial X)$ determines a homomorphism $\Delta_F\colon H_2(X,\partial X)\to H_2(X)$ called a \emph{variation}~\cite{MR390278, MR385872, MR375337}, defined by $[x] \mapsto [x-F(x)]$.   Using that $X$ has Poincar\'{e}-Lefschetz duality, Saeki~\cite{MR2197449} showed that $\Delta_F$ satisfies an additional condition, making it what we call a \emph{Poincar\'{e} variation} (Definition~\ref{def:algebraicvariational}). There is a binary operation on the set of Poincar\'{e} variations, together with which they form a group $\mathcal{V}(H_2(X),\lambda_X)$. The map $F \mapsto F_*$ factors through this group via homomorphisms: \[\pi_0\Homeo^+(X, \partial X) \xrightarrow{F \mapsto \Delta_F} \mathcal{V}(H_2(X),\lambda_X) \xrightarrow{\Delta \mapsto \Id - \Delta \circ j} \Aut(H_2(X),\lambda_X),\]
where $j \colon H_2(X) \to H_2(X,\partial X)$ is the quotient map. If $X$ is closed, the second map is an isomorphism, but
in general $\Delta_F$ contains more information than $F_*$.
Saeki~\cite{MR2197449} used $\mathcal{V}(H_2(X),\lambda_X)$ to describe the smooth stable mapping class group for simply connected 4-manifolds with nonempty, connected boundary.

When $\partial X$ has more than one connected component and ~$X$ admits a spin structure, there is a further invariant that does not appear in the closed case nor when the boundary is connected. For $F\in\Homeo^+(X,\partial X)$ we may compare a (topological) spin structure $\mathfrak{s}$ on $X$ with the induced spin structure $F^*\mathfrak{s}$. The two agree on $\partial X$ because $F$ fixes the boundary pointwise. There is a free, transitive action of $H^1(X,\partial X;\Z/2)$ on the set of isomorphism classes of spin structures on $X$ that agree on $\partial X$, and we denote by $\Theta(F) \in H^1(X,\partial X;\Z/2)$ the class representing the difference between $\mathfrak{s}$ and $F^*\mathfrak{s}$ (see Definition~\ref{def:spinobstruction} for a more detailed discussion).

Our main result shows that these invariants describe the entire topological mapping class group.

\begin{theoremalpha}\label{theoremA}
Let $X$ be a compact, simply connected, oriented, topological $4$-manifold with (possibly empty) boundary.
\begin{enumerate}[leftmargin=*]\setlength\itemsep{0em}
\item \label{item:spincase-intro}
When $X$ is spin, the map $F\mapsto  (\Theta(F),\Delta_F)$ induces a group isomorphism
\[
\pi_0\Homeo^+(X,\partial X)\xrightarrow{\cong}H^1(X,\partial X;\Z/2)\times\mathcal{V}(H_2(X),\lambda_X).
\]
\item \label{item:nonspincase-intro}
When $X$ is not spin, the map $F\mapsto  \Delta_F$ induces a group isomorphism
\[
\pi_0\Homeo^+(X,\partial X)\xrightarrow{\cong}\mathcal{V}(H_2(X),\lambda_X).
\]
\end{enumerate}
\end{theoremalpha}

Our key contribution is injectivity of the maps in Theorem~\ref{theoremA}. Let us outline the proof strategy.  First recall that a \emph{topological pseudo-isotopy} is a homeomorphism $F \colon X \times I \to X \times I$ such that $F|_{\partial X \times I} = \Id_{\partial X \times I}$. The restrictions $F_0 = F|_{X \times \{0\}}$ and $F_1 := F|_{X \times \{1\}}$ are said to be \emph{topologically pseudo-isotopic}. In this article we will classify homeomorphisms of simply connected 4-manifolds with boundary, up to topological pseudo-isotopy.
The strategy builds on that of~\cite[Proposition~2]{MR561244}, and is explained in Section~\ref{subsection:proof-thm-main}. In broad strokes, if we can find a 6-manifold with boundary the (capped off) mapping torus of $F$, such that the 6-manifold is a rel.\ boundary $h$-cobordism from $X \times [0,1]$ to itself, then it follows that $F$ is pseudo-isotopic to the identity. Our proof consists of an analysis of the obstructions to finding such an $h$-cobordism, and uses modified surgery theory~\cite{MR1709301} as the main tool in its construction.
Kreck's theory requires the construction of certain surgery control maps into a control space $B$ that depends on the manifold $X$. In our application, $B$ is the \emph{normal $2$-type} of $X$ (Definition~\ref{def:normal2type}). The main challenge in our application of modified surgery is the construction of a \emph{normal $2$-smoothing} $X\times[0,1]\to B$ (Definition~\ref{def:normal2type}), that extends a particular map $\partial(X\times [0,1])\to B$, dependent on $F$. A core technical result of the paper is Proposition~\ref{prop:extensionfromboundary}, where we identify the primary obstruction to this extension as $\Theta(F)$ and the secondary obstruction as $\Delta_F$.

With the pseudo-isotopy classification in hand, the proof that the maps in Theorem~\ref{theoremA} are injective concludes by appealing to Perron and Quinn's result~\cite{MR862426},~\cite[Theorem~1.4]{Quinn:isotopy} (see also~\cite{GGHKP}) that topological pseudo-isotopy implies topological isotopy for homeomorphisms of simply connected, compact 4-manifolds.
We remark that Perron only stated that topological pseudo-isotopy implies topological isotopy for closed 4-manifolds, and for smooth 4-manifolds with boundary with no 1-handles.  But shortly afterwards his proof for the closed case~\cite[Section~7]{MR862426} could be adapted using Boyer's work~\cite{MR857447,zbMATH00218286} to recover the result for all compact 1-connected 4-manifolds (see \cite[Section 5]{GGHKP}).

Of course, the injectivity in Theorem~\ref{theoremA} can be applied to diffeomorphisms of smooth 4-manifolds, yielding a topological isotopy.
This is a important step in the hunt for exotic diffeomorphisms, which is currently a topic of considerable interest. For example Theorem~\ref{theoremA} was applied in this way by Iida-Konno-Mukherjee-Taniguchi \cite{Iida-Konno-Mukherjee-Taniguchi} and Konno-Mallick-Taniguchi~\cite{KMT23}.

When $X$ has nonempty, connected boundary, surjectivity of the map $\pi_0\Homeo^+(X,\partial X)\to \mathcal{V}(H_2(X),\lambda_X)$ was already known, and is a consequence of Boyer's classification of simply connected compact 4-manifolds with connected boundary, and a subsequent result of Saeki~\cite{MR857447,zbMATH00218286, MR2197449}.
The proofs of Boyer and Saeki admit routine extensions to the case of disconnected boundary, though, carefully written, these extensions are quite lengthy. In this article we use a quicker method instead, suggested by Teichner, where we reduce the disconnected boundary case to the case of connected boundary; see the proof of Theorem~\ref{thm:realisevariational}.
To show moreover that the map in Theorem~\ref{theoremA}~(\ref{item:spincase-intro}) is surjective, in particular to realise the $\Theta$ invariants topologically, requires a novel geometric construction, again in combination with Boyer and Saeki's results~\cite{MR857447,zbMATH00218286, MR2197449}.

If $X$ is moreover a smooth manifold with (possibly empty) boundary, then a \emph{smooth pseudo-isotopy} is defined similarly to the topological case, replacing homeomorphism with diffeomorphism. In this article we will also classify diffeomorphisms of such manifolds up to smooth pseudo-isotopy. Quinn proved~\cite[Theorem~1.4]{Quinn:isotopy},~\cite{GGHKP} that smooth pseudo-isotopy implies \emph{smooth stable isotopy}. We note that there is now an independent proof of this due to Gabai~\cite[Theorem~2.5]{Gabai-schoenflies}. Here, we say that two self-diffeomorphisms $F$ and $F'$ of a smooth 4-manifold $X$ are \emph{stably isotopic} if, after isotoping each of them to fix the same $D^4 \subseteq X$, and using this $D^4$ to perform connected sums, for some $g\geq 0$ they induce smoothly isotopic diffeomorphisms
\[
F \# \Id,\,\, F' \# \Id \colon X \#^g S^2 \times S^2 \xrightarrow{\cong} X \#^g S^2 \times S^2.
\]
By~\cite[Corollary~5.4]{MR3355110}, this operation (for $g=1$) determines a homomorphism~
\[
\pi_0\Diffeo^+(X,\partial X) \to \pi_0\Diffeo^+(X \# S^2 \times S^2,\partial X)
\]
for every simply connected $X$.  When $\partial X \neq \emptyset$, the appeal to \cite{MR3355110} is unnecessary -- a diffeomorphism that fixes the boundary pointwise can be isotoped to fix a boundary collar pointwise, in a canonical way, and we can then take the connected sum in the collar.  We consider the corresponding colimit
\[
\underset{g\to\infty}{\colim}\, \pi_0\Diffeo^+(X \#^g S^2 \times S^2,\partial X),
\]
 the \emph{stable mapping class group of $X$}, which was studied by Wall~\cite{Wall-diffeo}, Kreck~\cite{MR561244}, Quinn~\cite{Quinn:isotopy}, and Saeki~\cite{MR2197449}.
For $4$-manifolds, the unstable smooth mapping class group is only known for $\R^4$. We do not even know $\pi_0 \Diffeo^+(D^4,\partial D^4)$, and so we do not know the complete mapping class group for any compact 4-manifold. But, by passing to the stable group, we have been able to give a complete answer in the simply-connected case.  This contrasts fruitfully with exotic phenomena that occur in unstable smooth mapping class groups; see e.g.~\cite{Ruberman-98,Ruberman-99,Baraglia-Konno,Iida-Konno-Mukherjee-Taniguchi,KMT23}.

On the algebraic side we consider the effect of stabilisation on Poincar\'{e} variations (see Section~\ref{sec:pseudoisotopy} for definitions). The stabilisation $F\mapsto F\#\Id_{S^2\times S^2}$ of a self-diffeomorphism rel.~ boundary $F\colon X\to X$ induces a homomorphism $\mathcal{V}(H_2(X),\lambda_X) \to \mathcal{V}\big((H_2(X),\lambda_X) \oplus \mathcal{H}\big)$, where the variation $\Delta_F$ gets extended by the identity on the hyperbolic form $\mathcal{H} := (\Z \oplus \Z,\bsm 0 & 1 \\ 1 & 0 \esm)$. We may again consider the corresponding colimit. Also note that we may canonically identify $H^1(X,\partial X;\Z/2)$ and $H^1(X\#^g S^2 \times S^2, \partial (X \#^g S^2 \times S^2);\Z/2)$. Our smooth pseudo-isotopy classification  (Theorem~\ref{thm:main}) combines with \cite[Theorem~1.4]{Quinn:isotopy}, \cite{GGHKP}, and \cite[Theorem~2.5]{Gabai-schoenflies} to prove injectivity in the following theorem, which computes stable mapping class groups. Surjectivity follows from our Theorem~\ref{theoremA} and the Freedman-Quinn sum-stable smoothing theorem~\cite[Theorem~8.6]{FQ}. The proof is given in Section~\ref{sec:ABproofs}.

\begin{theoremalpha}\label{theoremB}
Let $X$ be a compact, simply connected, oriented, smooth 4-manifold.
\begin{enumerate}[leftmargin=*]\setlength\itemsep{0em}
\item \label{item:spincase-intro-thmB}
When $X$ is spin, the map $F \mapsto  (\Theta(F),\Delta_F)$ induces a group isomorphism
\[
\underset{g\to\infty}{\colim}\, \pi_0\Diffeo^+(X \#^g S^2 \times S^2,\partial X) \xrightarrow{\cong} H^1(X,\partial X;\Z/2)\times \underset{g\to\infty}{\colim}\, \mathcal{V} \big((H_2(X),\lambda_X)\oplus \mathcal{H}^{\oplus g}\big).
\]
\item \label{item:nonspincase-intro-thmB}
When $X$ is not spin, the map $F \mapsto  \Delta_F$ induces a group isomorphism
\[
\underset{g\to\infty}{\colim}\, \pi_0\Diffeo^+(X \#^g S^2 \times S^2,\partial X) \xrightarrow{\cong} \underset{g\to\infty}{\colim}\, \mathcal{V} \big((H_2(X),\lambda_X)\oplus \mathcal{H}^{\oplus g}\big).
\]
\end{enumerate}
\end{theoremalpha}

In particular, if $\Delta_F = \Delta_{F'}$, and additionally in the spin case if  $\Theta(F) =\Theta(F')$, then $F$ and~$F'$ are smoothly stably isotopic rel.\ boundary.
In the special case that the boundary of $X$ is nonempty and connected, Theorem~\ref{theoremB} exactly recovers \cite[Theorem~3.7]{MR2197449}.
when $X$ is closed, this recovers the classical computation of Wall, Kreck, and Quinn~\cite{Wall-diffeo,MR561244,Quinn:isotopy}, which as above also uses one of \cite{Gabai-schoenflies,GGHKP}.

In Section~\ref{sec:isometries}, we study the relationship of Poincar\'{e} variations to isometries of the intersection form, which allows us in some cases to phrase Theorems~\ref{theoremA} and~\ref{theoremB} in the more familiar language of isometries.
To state the result, let  $\Aut^{\fix}_{\partial}(H_2(X),\lambda_X)\subseteq \Aut(H_2(X),\lambda_X)$ be the subgroup of the isometries of the intersection form $\lambda_X$ that induce the identity map on the homology of the boundary and induce the trivial permutation of the set of spin structures on the boundary (see~Definitions~\ref{def:algebraicidentity} and~\ref{def:algebraicfix} for the precise statement). When the rank of $H_1(\partial X;\Q)$ is at most~1, there is an isomorphism $\mathcal{V}(H_2(X),\lambda_X)\cong\Aut_\partial^{\fix}(H_2(X),\lambda)$; see Theorem~\ref{thm:Saekialgebraic}. When~$X$ is spin all isometries of $\lambda_X$ induce the trivial permutation on the boundary spin structures, and so in the spin case $\Aut^{\fix}_{\partial}(H_2(X),\lambda_X)= \Aut_{\partial}(H_2(X),\lambda_X)$; see Lemma~\ref{lem:spinkillsfix}.

\setcounter{corollaryalpha}{2}
\begin{corollaryalpha}\label{theoremA'}
Let $X$ be a compact, simply connected, oriented, topological $4$-manifold such that $\partial X$ has the rational homology of $S^3$ or of $S^1\times S^2$.
\begin{enumerate}[leftmargin=*]\setlength\itemsep{0em}
\item\label{item:QHS31} The map $F\mapsto F_*$ induces a group isomorphism
\[
\pi_0\Homeo^+(X,\partial X)\xrightarrow{\cong}\Aut^{\fix}_\partial(H_2(X),\lambda_X).
\]
\item\label{item:QHS32} When $X$ is moreover smooth, the map $F\mapsto F_*$ induces a group isomorphism
\[
\underset{g\to\infty}{\colim}\, \pi_0\Diffeo^+(X \#^g S^2 \times S^2,\partial X) \xrightarrow{\cong} \underset{g\to\infty}{\colim}\, \Aut_\partial^{\fix} \big((H_2(X),\lambda_X)\oplus \mathcal{H}^{\oplus g}\big).
\]
\end{enumerate}
\end{corollaryalpha}

Item (\ref{item:QHS31}) is a new result but, as mentioned above, (\ref{item:QHS32}) also follows from \cite[Theorem~3.7]{MR2197449}; we include its statement here for completeness. The proof of Corollary~\ref{theoremA'} is given at the end of Section~\ref{sec:isometries}.

Finally, we phrase a consequence of our results in terms of the \emph{Torelli group} of $(X,\partial X)$. The \emph{topological Torelli group} $\Tor^0(X,\partial X)\subseteq \pi_0\Homeo^+(X,\partial X)$ (resp.~\emph{smooth Torelli group} $\Tor^\infty(X,\partial X)\subseteq \pi_0\Diffeo^+(X,\partial X)$) is defined as the subgroup of mapping classes that induce the identity map on the homology of $X$. For an element  $F$ in either type of Torelli group, it is observed in \cite[\textsection4]{MR2197449} that the Poincar\'{e} variation $\Delta_F\colon H_2(X,\partial X)\to H_2(X)$ determines a map $H_1(\partial X)\to H_2(\partial X)\cong H_1(\partial X)^{\ast}$, and that the Poincar\'{e} property of $\Delta_F$ implies that the induced pairing $H_1(\partial X)\times H_1(\partial X)\to \Z$ is skew-symmetric. We write $\kappa_F\in\wedge^2{H_1(\partial X)^{\ast}}$ for this skew-symmetric pairing. Combining the work in Section~\ref{sec:isometries} with Theorems~\ref{theoremA} and~\ref{theoremB}, we deduce the following result, whose proof is given at the end of Section~\ref{sec:isometries}.

\begin{corollaryalpha}\label{theoremTorelli}
Let $X$ be a compact, simply connected, oriented, topological $4$-manifold with (possibly empty) boundary.
\begin{enumerate}[leftmargin=*]\setlength\itemsep{0em}
\item\label{item:cor-torelli-1} When $X$ is spin, there is an isomorphism
\[
\Tor^0(X,\partial X)\xrightarrow{\cong} \wedge^2{H_1(\partial X)^{\ast}}\times H^1(X,\partial X;\Z/2);\quad F\mapsto (\kappa_F,\Theta(F)).
\]
\item\label{item:cor-torelli-2} When $X$ is not spin, there is an isomorphism
\[
\Tor^0(X,\partial X)\xrightarrow{\cong} \wedge^2{H_1(\partial X)^{\ast}};\quad F\mapsto \kappa_F.
\]
\end{enumerate}
Now suppose moreover that $X$ is smooth.
\begin{enumerate}[leftmargin=*]\setlength\itemsep{0em}
\setcounter{enumi}{2}
\item\label{item:cor-torelli-3} When $X$ is spin, there is an isomorphism
\[
\underset{g\to\infty}{\colim}\, \Tor^\infty(X \#^g S^2 \times S^2,\partial X) \xrightarrow{\cong}\wedge^2{H_1(\partial X)^{\ast}}\times H^1(X,\partial X;\Z/2);\,\,\,F\mapsto (\kappa_F,\Theta(F)).
\]
\item\label{item:cor-torelli-4} When $X$ is not spin, there is an isomorphism
\[
\underset{g\to\infty}{\colim}\, \Tor^\infty(X \#^g S^2 \times S^2,\partial X) \xrightarrow{\cong} \wedge^2{H_1(\partial X)^{\ast}};\quad F\mapsto \kappa_F.
\]
\end{enumerate}
\end{corollaryalpha}

For the remainder of the introduction we describe consequences and applications of our results.

\subsection{Dehn twists}\label{subsection-intro-Ehn-twists}

An important type of self-homeomorphism of 4-manifolds is the \emph{Dehn twist}, which arises as follows. Let $\phi_t \in \pi_1(\SO(4))$ be a generator based at the identity matrix, represented by a smooth map $S^1 \to \SO(4)$ that is constant near the basepoint. This induces a loop of self-diffeomorphisms of $S^3$, which generates $\pi_1(\Diffeo^+(S^3))\cong\Z/2$~\cite{Hatcher-Smale-conj}, and thence a self-diffeomorphism
\[
  \Phi \colon S^3 \times I  \xrightarrow{\cong} S^3 \times I;\qquad (x,t) \mapsto (\phi_t(x),t).
\]
Given an embedding of $S^3 \times I$ into a 4-manifold, one can extend the map $\Phi$ by the identity to obtain a self-homeomorphism of the entire 4-manifold, and we call any self-homeomorphism obtained this way a \emph{Dehn twist}. If $X$ is smooth to begin with, and $S^3 \times I$ is smoothly embedded, then the Dehn twist is a self-diffeomorphism.

Now let $X$ be a closed, simply connected 4-manifold and decompose $X \sm \mathring{D}^4$ as the union $N \cup_{S^3 \times \{1\}} S^3 \times I$ of a collar neighbourhood of $\partial (X \sm \mathring{D}^4)$  and the closure of its complement.
The diffeomorphism~$\Phi$ induces a Dehn twist homeomorphism
\[
  t_X  \colon X \sm \mathring{D}^4 \to X \sm \mathring{D}^4;\qquad
  y  \mapsto \begin{cases}
    \Phi(x,t) & y = (x,t) \in S^3 \times I, \\
     y & y \in N.
  \end{cases}
\]

\setcounter{theoremalpha}{4}
\begin{theoremalpha}\label{theoremC}
For every closed, simply connected, topological manifold $X$, the Dehn twist $t_X $ is topologically isotopic rel.\ boundary to $\Id_{X \sm \mathring{D}^4}$.  If $X \sm \mathring{D}^4$ admits a smooth structure, then $t_X $ is smoothly pseudo-isotopic to $\Id_{X \sm \mathring{D}^4}$.
\end{theoremalpha}

\begin{proof} Note that $t_X$ lies in the Torelli group. The topological Torelli groups are computed in Corollary~\ref{theoremTorelli}. As $\partial (X\sm\mathring{D}^4)\cong S^3$ is connected and has $H_1(S^3)=0$ the Torelli group is trivial.
\end{proof}

For the specific case of the $K3$ surface, it was already known that $t_{K3}$ was topologically isotopic to the identity (see~\cite[Remark 6]{KK}, which uses results of Baraglia-Konno~\cite{BK}). This result was used recently by Lin-Mukherjee \cite{Lin-Muker}. For general smooth $X$, it was not previously known that $t_X$ is smoothly pseudo-isotopic to the identity.
Our Theorem~\ref{theoremB}, applied in a special case, recovers
Krannich-Kupers'~\cite[Theorem B]{KK}, which was proven by a totally different method, and states that $t_{X}$ is stably smoothly isotopic to $\Id_{X \sm \mathring{D}^4}$. (Krannich-Kupers'~\cite[Theorem B]{KK} also follows directly from Saeki's~\cite[Theorem~3.7]{MR2197449}.)

An explicit geometric argument of Giansiracusa shows that $t_{\CP^2}$ is smoothly isotopic to the identity~\cite{MR2400996}. This result can be extended to show that $t_X$ is smoothly isotopic to the identity for any non-spin, smooth, simply connected, closed $4$-manifold $X$; this argument was communicated to us by Dave Auckly, and we include the proof in Appendix~\ref{sec:Auckly}. On the other hand, it was shown independently by Baraglia-Konno~\cite{BK} and Kronheimer-Mrowka~\cite{MR4216604} that $t_{K3}$ is not smoothly isotopic to the identity. In light of Theorem~\ref{theoremC}, this prompts the obvious question.

\begin{question}\label{q:dehntwist}
For which closed, spin, simply connected, smooth manifolds $X$ is $t_X$ \emph{smoothly} isotopic rel.\ boundary to the identity?
\end{question}

Arguably, the particular interest in Dehn twists comes from the fact that they lie in the Torelli group. We now consider extensions of the idea of a Dehn twist, with Corollary~\ref{theoremTorelli} in mind.

\subsection{Generalisations of Dehn twists}\label{sec:generaliseddehn}

Let $Y \neq S^3$ be a closed, oriented 3-manifold and suppose there is a nontrivial loop $\phi_t \in \pi_1(\Diffeo^+(Y))$, based at the identity. Represent it by a continuous map $S^1 \to \Diffeo^+(Y)$, and write $\Phi \colon Y \times [0,1] \to Y \times [0,1]$ for the induced self-diffeomorphism. Not all $3$-manifolds have such a loop, for instance diffeomorphism groups of hyperbolic 3-manifolds are homotopy equivalent to their isometry group~\cite{Gabai-Smale-conj}, and are thus discrete. On the other hand every Seifert fibred $3$-manifold admits such a nontrivial loop of self-diffeomorphisms, as we show in Proposition~\ref{prop:niftycalc}. We now formulate a generalised version of Question~\ref{q:dehntwist}, using these types of diffeomorphism. In the topological category, the answer to the following question is ``no'' by Corollary~\ref{theoremTorelli}.

\begin{question}\label{q:exotictwists}
Is there a simply connected $4$-manifold $X$ with connected boundary $Y \neq S^3$, and a nontrivial loop in $\phi_t \in \pi_1(\Diffeo^+(Y))$ such that applying the corresponding $\Phi$ in a collar neighbourhood $Y \times [0,1]$ of the boundary gives a self-diffeomorphism of $F$ of $X$, such that $\kappa_F=0\in \wedge^2{H_1(\partial X)^{\ast}}$, but such  that $F$ is not smoothly isotopic to the identity?\footnote{Since the initial version of this article, Question~\ref{q:exotictwists} was answered affirmatively by Konno-Mallick-Taniguchi~\cite{KMT23}. For many interesting examples of $4$-manifolds with connected, Seifert fibred boundary, the authors prove the boundary twist described above is not smoothly isotopic to the identity rel.~boundary. These diffeomorphisms are  shown to be topologically isotopic to the identity, rel.~boundary, using Corollary~\ref{theoremA'}.}
\end{question}

When $\partial X$ is connected, Corollary~\ref{theoremTorelli} shows that Dehn twists along a boundary collar, and their generalisations considered above, are always topologically isotopic to the identity, provided the $\kappa$-invariant of the twist vanishes. This is not the case when $\partial X$ has more than one connected component. Indeed, Dehn twists are the prototypical example of self-homeomorphisms with nontrivial invariant $\Theta$.

\begin{example}\label{example:dehn-twist-intro}
  The Dehn twist on $S^3 \times [0,1]$ gives a homeomorphism that is not isotopic rel.\ boundary to the identity, despite necessarily acting trivially on $H_2(S^3 \times [0,1]) =0$. This is detected by the nontrivial action of the Dehn twist on the set of relative spin structures, which has cardinality $2 = |H^1(S^3 \times [0,1],S^3 \times \{0,1\};\Z/2)|$. See Example~\ref{ex:S3} for details.
\end{example}

To generalise this, suppose there is a loop in $\pi_1(\Diffeo^+(Y))$, based at the identity, and a spin structure on $Y \times [0,1]$, such that the induced self-diffeomorphism of $Y\times[0,1]$ has nontrivial $\Theta$ invariant. We call such a self-diffeomorphism a \emph{generalised Dehn twist} (see Definition~\ref{def:GDT}). We make a preliminary study of which $3$-manifolds admit generalised Dehn twists in Section~\ref{subsec:which3manifolds}, where we prove the following proposition.

\begin{proposition}\label{prop:lens-spaces-admit-twists-intro}
  Every closed, orientable 3-manifold $Y$ of Heegaard genus at most one admits a generalised Dehn twist for every spin structure on  $Y \times [0,1]$.
\end{proposition}

Recall that the 3-manifolds of Heegaard genus at most one are $S^3$, $S^1 \times S^2$, and the lens spaces~$L(p,q)$.
If all (except possibly one) boundary components of a smooth, spin $4$-manifold admit generalised Dehn twists, we show in Proposition~\ref{prop:switcheroo} that every possible value of $\Theta$ can be realised smoothly. However, we expect that in general, for a smooth, spin $4$-manifold, not all values of $\Theta$ can be realised by a (boundary fixing) self-diffeomorphism. This would contrast with the corresponding statement from the topological category, coming from Theorem~\ref{theoremA}. Our expectation is based on the fact that not all 3-manifolds admit generalised Dehn twists; for example as mentioned above hyperbolic 3-manifolds do not.

On the other hand, we have not proven any results limiting which values can be taken by~$\Theta(F)$, by $\Delta_F$, or indeed by the isometry invariant $F_*\in\Aut_\partial^{\fix}(H_2(X),\lambda_X)$, when $F\in\Diffeo^+(X,\partial X)$.  The question of which isometries of the intersection form can be realised by diffeomorphisms has been well-studied, although so far only in the case that the form is nonsingular. For closed, simply connected $4$-manifolds, Wall \cite{Wall-diffeo} gave conditions guaranteeing all isometries can be realised. In the other direction, there are examples where some isometries cannot be realised; see \cite{MR940111}, \cite{MR1066174}, \cite{zbMATH01019782}, \cite[Remark~1.10]{MR4224743} for closed manifold examples, and~\cite[Remark 1.8]{Konno-Taniguchi} for examples with homology 3-sphere boundary. See \cite[Section~4.3]{Konno-Taniguchi} for a detailed discussion. On the other hand, the $\Theta$ and $\kappa$ invariants have not been studied in this context, and the following questions represent new avenues for finding differences between the smooth and topological categories.

\begin{question}\label{q:theta}
Is there a smooth, spin, simply connected $4$-manifold $X$ for which the homomorphism $\Theta \colon \pi_0\Diffeo^+(X,\partial X) \to H^1(X,\partial X;\Z/2)$ is not surjective?
\end{question}

Similarly, the following question represents a new approach for finding differences between the smooth and topological categories.

\begin{question}\label{q:kappa}
Is there a smooth, simply connected $4$-manifold $X$ for which $\Tor^\infty(X,\partial X)\to \wedge^2{H_1(\partial X)^{\ast}}$, given by $F\mapsto \kappa_F$, is not surjective?
\end{question}

Corollary~\ref{theoremTorelli} implies that for every smooth $X$, and any pair $(\kappa,x)\in \wedge^2{H_1(\partial X)^{\ast}}\times H^1(X,\partial X;\Z/2)$ there exists $g\geq 0$ and $F\in \Diffeo^+(X\#^gS^2\times S^2,\partial X)$ realising the pair $(\kappa,x)$. Thus one might augment Questions~\ref{q:theta} and~\ref{q:kappa} in terms of how many copies of $S^2\times S^2$ are required to realise all values of $\Theta$ or of $\kappa$, for a given $X$. Is one stabilisation always sufficient?

\subsection{Homeomorphisms not restricting to the identity on the boundary}\label{subsection:relaxing-bdy-condition}

We consider the implications of our results when we relax the assumption that homeomorphisms must fix the boundary pointwise. Let $X$ be a compact, oriented, simply connected 4-manifold. There is a fibre sequence
\[
\Homeo^+(X,\partial X) \to \Homeo^+(X) \to \Homeo^+(\partial X)
\]
(see Appendix~\ref{sec:Auckly} for justification).
Consequently there is an exact sequence in homotopy groups, extending to the left,
\begin{equation}\label{eq:exactsequencehtpy}
\pi_1\Homeo^+(\partial X)\to\pi_0\Homeo^+(X,\partial X) \to \pi_0\Homeo^+(X) \to \pi_0\Homeo^+(\partial X).
\end{equation}
Here, the first arrow can be defined by inserting the given loop of diffeomorphisms of $\partial X$ (based at $\Id_{\partial X}$) into a collar of the boundary. Taking the basepoint of each group of homeomorphisms to be the respective identity map, the sequence \eqref{eq:exactsequencehtpy} is an exact sequence of groups. Here the $\pi_0$ terms are also groups because they are connected components of topological groups.
The sequence suggests that the problem of whether two homeomorphisms $F_1,\, F_2 \colon (X,\partial X) \to (X,\partial X)$ are isotopic in $\Homeo^+(X)$ can be decomposed into two stages, as follows.

The first-stage question is purely about $3$-manifolds: are $F_1|_{\partial X}$ and $F_2|_{\partial X}$ isotopic?  This is a highly nontrivial question in general, but thanks to the current spectacular understanding of 3-manifolds, with enough work, it can in principle be answered. First one considers the prime decomposition~\cite{Kneser,Milnor-prime-decomp} and then the JSJ decomposition~\cite{Jaco-Shalen,Johannson,Hatcher-3-manifolds-notes}. For geometric pieces it often suffices to understand the isometry groups (in the sense of Riemannian geometry), by  \cite{Gabai-Smale-conj, Hong-Kalliongis-McCullough-Rubinstein,Bamler-Kleiner-21,Bamler-Kleiner-22} and the references therein.
For Seifert fibred spaces in general see e.g.\ \cite{zbMATH00059795,MR710104}, and for Haken 3-manifolds see \cite{Waldhausen-Haken-3-manifolds,Hatcher-homeomorphisms,Ivanov-diffeomorphisms}.

If there is no isotopy between $F_1|_{\partial X}$ and $F_2|_{\partial X}$, then certainly $F_1$ and $F_2$ are not isotopic. So let us assume that the 3-manifold question has been solved affirmatively. Then, after an isotopy of $F_1$ supported in a collar of $\partial X$ we can assume that $F_1|_{\partial X} =F_2|_{\partial X}$. We may ask the second-stage question: is $G:=F_2\circ F_1^{-1}\in \Homeo^+(X,\partial X)$ in the image of $\pi_1 \Homeo^+(\partial X)$? By Theorem~\ref{theoremA}, when $X$ is not spin (resp.~is spin), this is equivalent to the question of whether $\Delta_G$ (resp.~the pair~$(\Delta_{G},\Theta(G))$), can be killed by inserting a loop of self-homeomorphisms on $\partial X$ in a boundary collar. Equivalently: can the algebraic invariant be \emph{realised} by such a collar loop? As a loop in the collar can never produce a nontrivial automorphism of $H_2(X)$, if $G_*\colon H_2(X)\to H_2(X)$ is nontrivial then certainly $G$ is not in the image of $\pi_1 \Homeo^+(\partial X)$, and thus $F_1$ and $F_2$ are not isotopic in $\Homeo^+(X)$. So, let us assume that $G$ determines an element of the topological Torelli group. By Corollary~\ref{theoremTorelli}, we finally arrive at the question: can the invariant $\kappa_G$ (resp.~the pair $(\kappa_{G},\Theta(G))$, in the spin case) be realised by a loop of homeomorphisms, or equivalently by a loop of diffeomorphisms~\cite{zbMATH03165165,Hatcher-Smale-conj}, in a collar of the boundary?

In some cases, $\pi_1 \Homeo^+(\partial X) =0$ and so it causes no additional complications. A general condition for this is as follows.

%\begin{definition}
%  We say that a closed, connected, orientable 3-manifold $Y$ is \emph{irrotational} if $Y$ is either hyperbolic or is Haken but not Seifert fibred.
%\end{definition}

%with a hyperbolic piece.
%\footnote{MP: could delete this sentence if it doesn't help?  Also I thin}

\begin{proposition}\label{proposition:zero-on-left-for-irrotational}
  Let $X$ be a compact, simply connected, oriented, topological $4$-manifold and suppose that every connected component of $\partial X$ is irreducible but not Seifert fibred.
  %either hyperbolic or is Haken but not Seifert fibred.
  Then $\pi_1 \Homeo^+(\partial X) =0$ and so there is exact sequence of groups
  \[
0\to \pi_0\Homeo^+(X,\partial X) \to \pi_0\Homeo^+(X)\to\pi_0\Homeo^+(\partial X).
\]
\end{proposition}

%The condition on a component $Y$ of $\partial X$ is satisfied, for example, if $Y$ is irreducible and has a nontrivial JSJ decomposition.

%By combining Proposition~\ref{proposition:zero-on-left-for-irrotational} with
Theorem~\ref{theoremA} describes the left group, although in the relevant cases the formulation in Corollary~\ref{theoremA'} could be more useful for applications.
So in the case that every connected component of $\partial X$ is irreducible but not Seifert fibred, the two-stage process discussed above can be completed.
We also note that the image of the right hand map in \eqref{eq:exactsequencehtpy} was described precisely by Boyer~\cite[Theorem~0.7]{MR857447}.
That is, Boyer gave a precise condition for a homeomorphisms of~$\partial X$ to extend to a homeomorphism of $X$.
This was for connected boundary, but the general case follows as in the proof of Theorem~\ref{thm:realisevariational}.

\begin{proof}[Proof of Proposition~\ref{proposition:zero-on-left-for-irrotational}]
Let $Y$ be a connected component of $\partial X$. Since $Y$ is irreducible and is not Seifert fibred, by Perelman's Geometrisation theorem~\cite{Morgan-tian-1,Morgan-Tian-2} it is either hyperbolic, or it has a nontrivial JSJ decomposition.  In the latter case $Y$ is Haken because a JSJ torus is an incompressible surface.
  We argue that the homeomorphism group of  $Y$ has trivial fundamental group.  For hyperbolic 3-manifolds this follows from~\cite{Gabai-Smale-conj}, which shows that $\Homeo(Y) \simeq \operatorname{Isom}(Y)$, since isometry groups of hyperbolic 3-manifolds are discrete.  For a closed, oriented Haken 3-manifold $Y$, Hatcher and Ivanov~\cite{Hatcher-homeomorphisms,Ivanov-diffeomorphisms} showed that $\pi_1\Homeo^+(Y)$ is isomorphic to the centre of $\pi_1(Y)$, and Waldhausen~\cite{Waldhausen-centre-pi1-diff-nontriv-implies-SF} showed that in such cases if $Y$ is not Seifert-fibred then the centre of $\pi_1(Y)$ is trivial.
 So $\pi_1(\Homeo^+(\partial X))$ is trivial as asserted.  Applying this to the exact sequence~\eqref{eq:exactsequencehtpy} yields the displayed exact sequence.
\end{proof}

In Section~\ref{subsec:which3manifolds}, we consider Seifert fibred 3-manifold boundary components, and study the problem of realising the $\Theta$ and $\kappa$ invariants using loops of diffeomorphisms in a boundary collar. A consequence of that work is the following, which we prove in Section~\ref{sec:which}.

\setcounter{corollaryalpha}{5}
\begin{corollaryalpha}\label{corollaryG}
Let $X$ be a compact, simply connected, oriented, topological $4$-manifold.
Suppose that every connected component of $\partial X$ has Heegaard genus at most 1, and at most one of the connected components is $S^1\times S^2$.
Then there is an exact sequence of groups
\[
0\to \Aut^{\fix}_\partial(H_2(X),\lambda_X)\to \pi_0\Homeo^+(X)\to\pi_0\Homeo^+(\partial X).
\]
\end{corollaryalpha}

\begin{remark}
  It is interesting to compare the map $\pi_1\Homeo^+(\partial X)\to\pi_0\Homeo^+(X,\partial X)$ with its smooth counterpart $\pi_1\Diff^+(\partial X)\to\pi_0\Diff^+(X,\partial X)$. Kronheimer-Mrowka~\cite{MR4216604}, Jianfeng Lin~\cite{Lin-Dehn-twist}, and Konno-Mallick-Taniguchi~\cite{KMT23} all found examples of simply-connected 4-manifolds $X$ with loops in $\pi_1\Diff^+(\partial X)$ that act trivially on $\pi_0\Homeo^+(X,\partial X)$ (in the case of \cite{KMT23}, this was shown using Corollary~\ref{theoremA'}), but nontrivially on $\pi_0\Diff^+(X,\partial X)$.
\end{remark}

\subsection{Isotopy of $\Z$-surfaces}

One application for mapping class group computations is to the study of knotted surfaces in 4-manifolds, up to ambient isotopy. A general strategy to show that two locally flat, embedded surfaces $\Sigma_0$ and $\Sigma_1$ in a general $4$-manifold $X$ are topologically isotopic is as follows. First show that the exteriors $X \sm \nu \Sigma_0$ and $X \sm \nu \Sigma_1$ are homeomorphic relative to the identity of $\partial X$, and restricting to a homeomorphism $\partial \ol{\nu} \Sigma_0 \to \partial \ol{\nu} \Sigma_1$ that extends to a bundle isomorphism $\ol{\nu} \Sigma_0 \to \ol{\nu} \Sigma_1$. This yields a homeomorphism $F \colon X \to X$ that restricts to $\Id_{\partial X}$, with $F(\Sigma_0) = \Sigma_1$. One would then like to know whether $F$ is isotopic to the identity.

If we now further assume that $X$ is compact and simply connected, and that $\partial X$ is connected, then the results of this paper may be brought to bear on this final question. Under these assumptions, if $F$ induces the identity on $H_2(X)$ then Corollary~\ref{theoremA'} yields a topological ambient isotopy carrying $\Sigma_1$ to $\Sigma_0$, as we will explain in the proof of Theorem~\ref{theoremD} below.

An immediate application of Corollary~\ref{theoremA'}, via this method, is to strengthen a theorem of Conway and the second author~\cite[Theorem 1.3]{CP} on $\Z$-surfaces. A \emph{$\Z$-surface} is a compact, connected, oriented surface with one boundary component $\Sigma \subseteq X$, locally flatly embedded in a simply connected, compact 4-manifold $X$ with $\partial X = S^3$, with $\partial \Sigma \subseteq \partial X$, such that $\pi_1(X \sm \nu \Sigma) \cong \Z$, generated by a meridian of $\Sigma$. Write $X_{\Sigma} := X \sm \nu \Sigma$.

We need a little more setup. The inclusion induced map $\pi_1(\partial X_{\Sigma}) \to \pi_1(X_{\Sigma})$ induces a $\Z$-cover of $\partial X_{\Sigma}$. Write $\Lambda := \Z[\Z]$,  $Y:= S^3 \sm \nu \partial \Sigma$, and let $Z := \overline{\partial X_{\Sigma} \sm Y}$. Then  $\partial X_{\Sigma} = Y \cup Z$ and the homology $H_1(\partial X_{\Sigma};\Lambda)$ splits correspondingly as $H_1(Y;\Lambda) \oplus H_1(Z;\Lambda)$. An isometry of the Blanchfield pairing $h\colon H_1(\partial X_{\Sigma};\Lambda)\to H_1(\partial X_{\Sigma};\Lambda)$ respects this splitting, decomposing as $h=h_Y\oplus h_Z$~\cite[Proposition~5.7]{CP}.
Given exteriors $X_i:=X_{\Sigma_i}$ for $i=0,1$, of two $\Z$-surfaces in $X$, it is shown in~\cite[Lemma~5.10]{CP} that an isomorphism \[
\phi \colon H_2(X_0;\Lambda) \xrightarrow{\cong} H_2(X_1;\Lambda)
\]
inducing an isometry of the $\Lambda$-valued intersection pairings uniquely determines an element $\phi_\Z\in\Aut_\partial(H_2(X),\lambda_X)$. Moreover, if $\phi$ is induced by a homeomorphism $F\in\Homeo^+(X,\partial X)$, then $F_*=\phi_\Z$. (That lemma shows this when $X$ is closed, but the proof in our case is identical.)

In~\cite[Theorem 1.3]{CP}, hypotheses were given which imply there is a homeomorphism of pairs $(X,\Sigma_0)\cong (X,\Sigma_1)$ inducing $\phi$ on $\Lambda$-coefficient homology. Using Corollary~\ref{theoremA'}, we upgrade this to ambient isotopy, assuming in addition that $\phi_\Z$ is the identity.

\setcounter{theoremalpha}{6}
\begin{theoremalpha}\label{theoremD}
  Let $\Sigma_0,\Sigma_1 \subseteq X$ be $\Z$-surfaces with $\partial \Sigma_0 = \partial \Sigma_1 \subseteq \partial X \cong S^3$.  Write $X_i := X_{\Sigma_i}$ for $i=0,1$.  Let $\phi \colon H_2(X_0;\Lambda) \xrightarrow{\cong} H_2(X_1;\Lambda)$  be an isomorphism inducing an isometry of the $\Lambda$-valued intersection pairings of the exteriors, and write $\partial \phi = h_Y\oplus h_Z \colon H_1(\partial X_{\Sigma};\Lambda) \xrightarrow{\cong} H_1(\partial X_{\Sigma};\Lambda)$ with respect to the splitting $H_1(\partial X_{\Sigma};\Lambda)\cong H_1(Y;\Lambda) \oplus H_1(Z;\Lambda)$.

Suppose that $h_Y$ is induced by a homeomorphism $f\colon Y\xrightarrow{\cong} Y$ and $\phi_\Z=\Id_{H_2(X)}$.  Then $\Sigma_0$ and $\Sigma_1$ are ambiently isotopic. If in addition $h_Y=\Id$ then $\Sigma_0$ and $\Sigma_1$ are ambiently isotopic rel.~boundary.
\end{theoremalpha}

\begin{proof}
Under the additional assumption that $f|_{\partial Y}=\Id_{\partial Y}$, it was proved in~\cite[Theorem 1.3]{CP} that the hypotheses imply there is a homeomorphism of pairs $F \colon (X,\Sigma_0) \to (X,\Sigma_1)$ such that $F$ induces $\phi$, and that in addition if $h_Y=\Id$ then $F$ may be chosen so that $F|_{\partial X}=\Id_{\partial X}$. It is shown in the proof of \cite[Theorem~5.6]{CPP} that the additional assumption $f|_{\partial Y}=\Id_{\partial Y}$ was superfluous, and so we obtain such a homeomorphism of pairs $F$.

If $F|_{\partial X}\neq\Id_{\partial X}$ then we use the fact that $\pi_0\Homeo^+(S^3)$ is trivial to choose an isotopy $\Phi\colon \partial X\times[0,1]\to \partial X$ from $\Phi(-,0)=F|_{\partial X}$ to $\Phi(-,1)=\Id_{\partial X}$. Choose a boundary collar $\partial X\times[0,1]\subseteq X$ so that $\partial X\times\{1\}=\partial X$, and arrange by an isotopy rel.~boundary that $F|_{\partial X\times[0,1]}=F|_{\partial X}\times\Id_{[0,1]}$. Define an isotopy $G\colon X\times[0,1]\to X$ by
$G((x,s),t):=\Phi(x,s-t)$ for $t \leq s \leq 1$ and $G((x,s),t):=(F(x),s)$ for $0 \leq s \leq t$.
This is an isotopy from $F=G(-,1)$ to a new homeomorphism $F':=G(-,0)$ such that $F'|_{\partial X}=\Id_{\partial X}$.

Note that $(F')_*=F_* = \phi_\Z=\Id_{H_2(X)}$. As $\partial X=S^3$, Corollary~\ref{theoremA'} now implies that $F'$ is isotopic rel.~boundary to $\Id_X$. Since $F(\Sigma_0)=\Sigma_1$, we have obtained an ambient isotopy from $\Sigma_0$ to $\Sigma_1$. If in addition $h_Y=\Id$ then we saw above that we may assume $F|_{\partial X}=\Id_{\partial X}$, so the modification from $F$ to $F'$ was unnecessary, and the ambient isotopy from $\Sigma_0$ to $\Sigma_1$ fixes $\partial X$ pointwise throughout.
\end{proof}

\subsection*{Organisation}

In Section~\ref{sec:pseudoisotopy} we define the pseudo-isotopy invariants.
In Section~\ref{sec:newsection} we state our classification of homeomorphisms up to pseudo-isotopy. We also show how to realise the pseudo-isotopy invariants topologically and use this to prove Theorems~\ref{theoremA} and~\ref{theoremB}, with the assumption the main pseudo-isotopy classification is true.
In Section~\ref{section:elements-modified-surgery} we recall the tools from Kreck's modified surgery that we will need, and we start to apply them by determining the relevant normal 2-types.
In Section~\ref{section:james-SS-bordism-groups} we recall the James spectral sequence and use it make computations of bordism groups relevant to the pseudo-isotopy classification.
Section~\ref{sec:proofmain} contains the proof of the pseudo-isotopy classification in both the smooth and topological categories, so completing the proofs of Theorems~\ref{theoremA} and~\ref{theoremB}.
In Section~\ref{sec:isometries} we establish the relationship between Poincar\'{e} variations and isometries of the intersection form, allowing alternative formulations of the main theorems. In Section~\ref{subsec:which3manifolds} we discuss the problem of realising generalised Dehn twists smoothly, in particular proving Proposition~\ref{prop:lens-spaces-admit-twists-intro}.

\subsection*{Acknowledgments}
 We both gratefully acknowledge the MPIM, where we were visitors for part of the time this paper was written. MP warmly thanks ETH Z\"urich and Peter Feller for hospitality during a productive visit.  We thank Dave Auckly, Daniel Kasprowski, Hokuto Konno, Sander Kupers, Jianfeng Lin, Andrew Lobb, Anubhav Mukherjee, Lisa Piccirillo, and Danny Ruberman for interesting discussions or helpful questions about these results. We thank Peter Teichner for suggesting a quicker approach than we were previously using for improving Boyer's homeomorphism realisation result to the case of multiple boundary components. We thank Dave Auckly, Peter Kronheimer, Tom Mrowka, and Danny Ruberman for discussing Theorem~\ref{thm:Auckly} with us. Finally,  we are very grateful to the referees for many insightful comments that helped improve the paper.

PO was partially supported by the SNSF Grant~181199, and  MP was partially supported by EPSRC New Investigator grant EP/T028335/1 and EPSRC New Horizons grant EP/V04821X/1.

%======================================================

%
%\subsection*{Open Access}
%For the purpose of open access, the authors have applied a CC-BY Creative Commons attribution license to this version, and will do the same to any author-accepted manuscript arising.

\section{Conventions and definitions of invariants}\label{sec:pseudoisotopy}

In this section we carefully define the objects and maps involved in our main pseudo-isotopy classification, which is stated in the next section.

\begin{definition}
Write $\Diff$ for the category of smooth manifolds and smooth embeddings. Write $\TOP$ for the category of topological manifolds and locally flat embeddings. We let $\CAT$ denote either $\Diff$ or $\TOP$ unless explicitly specified or the specification is clear from context.
\end{definition}

For this section, fix $X$ a compact, connected, oriented
$\CAT$ $4$-manifold with (possibly empty) boundary.

\begin{definition}
When $\CAT=\TOP$, let $\Homeo^+(X,\partial X)$ denote the space of orientation preserving self-homeomorphisms of $X$ that restrict to the identity map on the boundary, topologised using the compact-open topology (see \cite[Appendix]{Hatcher} for the definition). When $\CAT=\Diff$, let $\Diffeo^+(X,\partial X)$ denote the space of orientation preserving self-diffeomorphisms of $X$ that restrict to the identity map on the boundary, topologised using the Whitney topology (see~\cite[\S~2.1]{Hirsch-Diff-Top} for the definition).
\end{definition}

\begin{remark}
Note that if $\partial X$ is nonempty and $X$ is connected, then for an automorphism $f \colon X \to X$, we have that $f|_{\partial X} = \Id_{\partial X}$ implies that $f$ is automatically orientation preserving. Although this means the ``$+$'' is redundant notation in these cases, we retain this notation so that our statements also hold in the case that $\partial X = \emptyset$.
\end{remark}

\begin{definition}
For $i=0,1$ let $F_i\colon X\xrightarrow{\cong} X$ be an orientation preserving $\CAT$ isomorphism restricting to the identity map on the boundary. We say $F_0$ and $F_1$ are \emph{$\CAT$ pseudo-isotopic rel.~boundary} if there exists a $\CAT$ isomorphism $\Psi\colon X\times [0,1]\to X\times [0,1]$ such that $\Psi(x,i)=(F_i(x),i)$ for $i=0,1$ and $\Psi(x,t)=(x,t)$ for all $x\in \partial X$ and $t\in[0,1]$.

We write
\[
\begin{array}{rclll}
\widetilde{\pi}_0\Homeo^+(X,\partial X)&:=& \Homeo^+(X,\partial X)&/&\text{$\TOP$ pseudo-isotopy rel.~boundary},\\
\widetilde{\pi}_0\Diffeo^+(X,\partial X)&:=& \Diffeo^+(X,\partial X)&/&\text{$\Diff$ pseudo-isotopy rel.~boundary}.
\end{array}
\]
Note these are groups under composition; we will call them the \emph{pseudo mapping class groups}.
\end{definition}

\subsection{Spin structures relative to the boundary}\label{subsec:relboundaryspin}

Given $n>0$, let $\STOP(n)$ be the subgroup of orientation-preserving homeomorphisms of $\R^n$
that fix the origin, topologised using the compact open topology. A principal $\STOP(n)$-bundle
has an associated \emph{topological $\R^n$-bundle}; that is a fibre bundle with fibre~$\R^n$, a preferred 0-section, and fibre automorphism group $\STOP(n)$. Let $\STOP$ be the colimit of the obvious sequence of inclusions $\STOP(n)\subseteq \STOP(n+1)$ and write~$\BSTOP$ for the corresponding classifying space. An \emph{oriented stable topological bundle} is the class of a topological $\R^n$-bundle under this stabilisation process, and is classified by a map~$v\colon X\to \BSTOP$.

Write $\gamma\colon \BTOPSpin\to \BSTOP$ for a principal fibration corresponding to the universal cover $\TOPSpin\to \STOP$. A \emph{spin structure} on an oriented stable topological bundle $v\colon X\to \BSTOP$ is a map~$\mathfrak{s}\colon X\to \BTOPSpin$ such that the following diagram
\[
\begin{tikzcd}
&& \BTOPSpin\ar[d,"\gamma"]\\
X\ar[rru,"\mathfrak{s}"]\ar[rr,"v"]&&\BSTOP
\end{tikzcd}
\]
commutes.
For a stable vector bundle, the definition is the same, with $\BSO$ and $\BSpin$ replacing $\BSTOP$ and $\BTOPSpin$. For $X$ a $\TOP$ (resp.~$\Diff$) manifold, a spin structure \emph{on $X$} means a spin structure on the stable topological normal bundle (resp.~stable normal vector bundle).

Suppose $\mathfrak{s}_0, \mathfrak{s}_1\colon X\to \BTOPSpin$ (resp.~$\BSpin$) are two spin structures on $X$. Write $\partial \mathfrak{s}_0$ and $\partial \mathfrak{s}_1$ for the respective restrictions to the boundary $\partial X$. We say $\mathfrak{s}_0$ and $\mathfrak{s}_1$ are \emph{isomorphic}, and write $\mathfrak{s}_0\simeq \mathfrak{s}_1$, if there is a path of spin structures $\mathfrak{s}_t$ on $v$, from $\mathfrak{s}_0$ to $\mathfrak{s}_1$; in other words if $\mathfrak{s}_0$ and $\mathfrak{s}_1$ are homotopic through spin structures on $v$. Note that an isomorphism between $\mathfrak{s}_0$ and $\mathfrak{s}_1$ induces an isomorphism between $\partial \mathfrak{s}_0$ and $\partial \mathfrak{s}_1$. For example, when a manifold admits a single isomorphism class of spin structure there is a unique isomorphism class of spin structure on the boundary that extends to the interior. By a standard obstruction-theoretic argument~\cite[Chapter~8.2]{Spanier}, there is a free, transitive action of $H^1(X;\Z/2)$ on the set of isomorphism classes of spin structures on $X$; see also~\cite[Proposition~1.4.25]{GS}. In particular, when $X$ is simply connected $H_1(X;\Z/2)=0$ and hence by the universal coefficient theorem $H^1(X;\Z/2)=0$. So such a manifold has a unique isomorphism class of spin structures.

We will also need a more restrictive notion of isomorphism of spin structure. We say that $\mathfrak{s}_0$ and $\mathfrak{s}_1$ are \emph{isomorphic rel.~boundary} if $\partial\mathfrak{s}_0=\partial\mathfrak{s}_1$, and $\mathfrak{s}_0$ is homotopic to $\mathfrak{s}_1$ through spin structures on $v$ that all restrict to $\partial\mathfrak{s}_0$ on $\partial X$. This is, in other words, a spin structure $\mathfrak{S}\colon X\times[0,1]\to \BTOPSpin$ (resp.~$\BSpin$) that restricts to
\[
\partial(X\times[0,1])\to\BTOPSpin\text{ (resp.~$\BSpin$)};\quad (x,t)\mapsto\left\{\begin{array}{ll}
\mathfrak{s}_0(x)&\text{for $(x,t)\in X\times\{1\}$,}\\
\mathfrak{s}_1(x)& \text{for $(x,t)\in X\times\{0\}$,}\\
\partial\mathfrak{s}_0(x) & \text{for $(x,t) \in \partial X \times (0,1)$.}
\end{array}\right.
\]
By a standard obstruction-theoretic argument~\cite[Chapter~8.2]{Spanier}, there is a free, transitive action of $H^1(X,\partial X;\Z/2)$ on the set of homotopy classes of spin structures on $X$, relative to some given spin structure on $\partial X$. In particular, when $X$ is simply connected and has $r>0$ boundary components, it has $|H^1(X,\partial X;\Z/2)|=|(\Z/2)^{r-1}|=2^{r-1}$  rel.\ boundary spin structures up to isomorphism.

\begin{definition}\label{def:spinobstruction}
Let $(X,\mathfrak{s})$ be a $\CAT$ manifold with a spin structure and let $F\colon X\xrightarrow{\cong} X$ be an orientation preserving $\CAT$ isomorphism restricting to the identity map on the boundary. We write the class representing the difference between the spin structures $\mathfrak{s}$ and $\mathfrak{s}\circ F$ rel.~boundary, as
\[
\Theta(F,\mathfrak{s})\in H^1(X,\partial X;\Z/2).
\]
When $X$ is simply connected, this class is independent of the choice of spin structure (see Lemma~\ref{lem:spinbusiness}), and in which case we write
\[
\Theta(F)\in H^1(X,\partial X;\Z/2).
\]
\end{definition}

When $X$ has a smooth structure, Definition~\ref{def:spinobstruction} allows us the option of working with the $\CAT=\Diff$ definition or passing to the underlying topological manifold and working with the $\CAT=\TOP$ definition. In this situation, the two definitions return the same cohomology class, as we now show.

%Recall that $\pi_1(\SO)=\Z/2$, that $\pi_1(\STOP)=\Z/2$ (see~Lemma~\ref{lem:wellknowngroups}) and that the forgetful map $f\colon \SO\to \STOP$ induces the identity map on $\pi_1$. Hence $f$ lifts to a map of the universal covers $g\colon \Spin \to \TOPSpin$.
%
%\begin{lemma}
%Suppose that $F\in\Diffeo^+(X,\partial X)$.  Then $\Theta(F,\mathfrak{s}) = \Theta(F,Bg \circ \mathfrak{s})$, where the former uses the definition in $\Diff$ and latter uses the $\TOP$ definition.
%\end{lemma}
%\begin{proof}
%
%\end{proof}

Recall that $\pi_1(\SO)=\Z/2$ and that $\pi_1(\STOP)=\Z/2$ (see~Lemma~\ref{lem:wellknowngroups}) and that the forgetful map $f\colon \SO\to \STOP$ induces the identity map on $\pi_1$. Hence there is a homotopy commutative diagram
\[
\begin{tikzcd}
\Spin\ar[d, "g"]\ar[r] &\SO\ar[r] \ar[d, "f"]& B(\Z/2)\ar[r] \ar[d, "h"]& \BSpin \ar[r] \ar[d, "Bg"]& \BSO\ar[d, "Bf"]\\
\TOPSpin\ar[r] &\STOP\ar[r] & B(\Z/2)\ar[r] & \BTOPSpin \ar[r] & \BSTOP
\end{tikzcd}
\]
where $g$ is the map of universal covers induced by $f$, and $h$ is a map homotopic to the identity map.

\begin{lemma}
Suppose $X$ is a simply connected, smooth manifold and that there is a spin structure $\mathfrak{s}\colon X\to \BSpin$ on the stable normal vector bundle $\nu_X\colon X\to \BSO$. Let $F\in\Diffeo^+(X,\partial X)$.  Then $\Theta(F,\mathfrak{s}) = \Theta(F,Bg \circ \mathfrak{s})$, where the former uses the definition in $\Diff$ and latter uses the $\TOP$ definition.
\end{lemma}
\begin{proof}
 The induced map $Bf\circ\nu_X$ is the stable topological normal bundle of $X$ and $Bg\circ \mathfrak{s}$ is the induced topological spin structure on $Bf\circ\nu_X$. For each $x\in H^1(X, \partial X;\Z/2)$, homotopy commutativity of the diagram implies that $Bg\circ(x\cdot \mathfrak{s}) \simeq h_*(x)\cdot(Bg\circ\mathfrak{s})$, where ``$\,\cdot\,$'' indicates the action of $H^1(X,\partial X;\Z/2)$ on rel.~boundary isomorphism classes of spin structure, in the respective categories. Since $h$ is homotopic to the identity on $B(\Z/2)$, we conclude that $Bg\circ(x\cdot \mathfrak{s}) \simeq x\cdot(Bg\circ\mathfrak{s})$.

 There is a unique $x\in H^1(X, \partial X;\Z/2)$ such that $x\cdot\mathfrak{s}\simeq \mathfrak{s}\circ F$. There is also a unique $y\in H^1(X,\partial X;\Z/2)$ such that $y\cdot (Bg\circ\mathfrak{s})\simeq Bg\circ\mathfrak{s}\circ F$.

 Combining the facts obtained so far, we have
\[
x\cdot(Bg\circ\mathfrak{s})\simeq Bg\circ(x\cdot \mathfrak{s})\simeq  Bg\circ\mathfrak{s}\circ F\simeq y\cdot (Bg\circ\mathfrak{s}).
\]
So $x=y$, and this is the class $\Theta(F,\mathfrak{s})=\Theta(F,Bg\circ\mathfrak{s})$, using the definition from either category.
\end{proof}

\begin{lemma}\label{lem:spinbusiness}
Suppose that $X$ is a spin $\CAT$ 4-manifold. Let $F, G\colon X\xrightarrow{\cong} X$ be orientation preserving $\CAT$ isomorphisms that restrict to the identity on the boundary. The following assertions hold.
\begin{enumerate}[leftmargin=*]\setlength\itemsep{0em}
\item Suppose $\mathfrak{s}_0$ and $\mathfrak{s}_1$ are isomorphic spin structures on $X$. Then $\Theta(F,\mathfrak{s}_0)=\Theta(F,\mathfrak{s}_1)$.
\item If $X$ is simply connected then $\Theta(F,\mathfrak{s})$ is independent of the choice of $\mathfrak{s}$. $($In this case we write $\Theta(F)$.$)$
\item If $X$ is simply connected then there is an equality
\[
\Theta(G\circ F)= \Theta(G)+ \Theta(F)\in H^1(X,\partial X;\Z/2).
\]
\end{enumerate}
\end{lemma}

\begin{proof}
%We may choose a homotopy $\mathfrak{S}$ from $\mathfrak{s}_0$ and $\mathfrak{s}_1$, as they are isomorphic.
%Consider the exact sequence of chain complexes
%\[
%\begin{tikzcd}[row sep = tiny]
%C_*(X,\partial X)\otimes C_*(\pt)\ar[r]
%& C_*(X,\partial X)\otimes C_*(\{0,1\})\ar[r]
%& C_*(X,\partial X)\otimes C_*([0,1]).\\
%\end{tikzcd}
%\]
%For a singular simplex $\tau$ in $C_*(\pt)$ write its image in $C_*(\{0\})$ as $\tau_0$ and its image in $C_*(\{1\})$ as $\tau_1$.
%The left map above send $\sigma \otimes \tau \mapsto \sigma \otimes \tau_1 - \sigma \otimes \tau_0$.  The right hand map is induced by inclusion.
%Writing $X_i:=X\times\{i\}$ for $i=0,1$, consider the following map in the associated exact $\Z/2$-coefficient cohomology sequence:
%\[
%\begin{tikzcd}[row sep = tiny]
%H^1(X\times[0,1],\partial X\times[0,1];\Z/2)\ar[r] &H^1(X_0,\partial X_0;\Z/2)\oplus H^1(X_1,\partial X_1;\Z/2)\\
%\Theta(F\times[0,1],\mathfrak{S})\arrow[r,mapsto] &(\Theta(F,\mathfrak{s}_0),\Theta(F,\mathfrak{s}_1))
%\end{tikzcd}
%\]
%The subsequent relative term in the homology exact sequence is $H^1(X,\partial X;\Z/2)$, and the onwards map
%$H^1(X_0,\partial X_0;\Z/2)\oplus H^1(X_1,\partial X_1;\Z/2) \to H^1(X,\partial X;\Z/2)$ sends $(a,b)\mapsto a+b$. We conclude that $\Theta(F,\mathfrak{s}_0)+\Theta(F,\mathfrak{s}_1)=0$ by exactness.

For the first item, as $\mathfrak{s}_0$ is isomorphic to $\mathfrak{s}_1$ we may choose a homotopy $\mathfrak{S}$ from $\mathfrak{s}_0$ to $\mathfrak{s}_1$ through spin structures on the $\CAT$ stable normal bundle. Then $\mathfrak{S}\circ F$ is a homotopy from $\mathfrak{s}_0\circ F$  to $\mathfrak{s}_1\circ F$, through spin structures on the $\CAT$ stable normal bundle. This shows $\mathfrak{s}_0\circ F$ is isomorphic to $\mathfrak{s}_1\circ F$. The action of $H^1(X,\partial X;\Z/2)$ is on \emph{isomorphism classes} of rel.~boundary spin structures. Thus the difference classes are equal $\Theta(F,\mathfrak{s}_0)=\Theta(F,\mathfrak{s}_1)$.

The second item follows from the first, because any two spin structures on a simply connected manifold are isomorphic.

For the third item, write $\theta(\mathfrak{u},\mathfrak{v})\in H^1(X,\partial X;\Z/2)$ for the class representing the difference between two spin structures $\mathfrak{u}$ and $\mathfrak{v}$ on $X$ that are equal on $\partial X$.
%The action of $H^1(X,\partial X;\Z/2)$ on the set of spin structures rel.~boundary is natural, in the sense that $\theta(\mathfrak{u}\circ J,\mathfrak{v}\circ J)=J^*\theta(\mathfrak{u},\mathfrak{v})$ for any orientation preserving automorphism $J$ of $X$ that fixes the boundary pointwise.
We compute
\[
\Theta(G\circ F)=\theta(\mathfrak{s}\circ (G\circ F),\mathfrak{s})=\theta((\mathfrak{s}\circ G)\circ F,\mathfrak{s})=\theta((\mathfrak{s}\circ G)\circ F,\mathfrak{s}\circ G)+\theta(\mathfrak{s}\circ G,\mathfrak{s})=\Theta(F)+\Theta(G).
\]
The third equality uses the fact that the difference classes arise from a group action. The fourth equality uses the fact, proved above, that the class $\Theta(F)$ can be defined using any spin structure on $X$.
\end{proof}

\begin{corollary}\label{cor:spinwelldef}
Suppose $(X,\mathfrak{s})$ is a simply connected, $\CAT$ 4-manifold with spin structure. Let $F, F'\colon X\xrightarrow{\cong} X$ be orientation preserving $\CAT$ isomorphisms that restrict to the identity on the boundary and that are $\CAT$ pseudo-isotopic rel.~boundary. Then $\Theta(F)=\Theta(F')$.
\end{corollary}

\begin{proof}
A pseudo-isotopy $\Psi\colon X\times[0,1]\xrightarrow{\cong}X\times[0,1]$ can be used to pull back the product spin structure $\mathfrak{s}\circ\pr_1$ from the target $X\times[0,1]$ to obtain a spin structure $\mathfrak{S}$  restricting to a spin structure on $\partial(X \times [0,1])$ that is precisely the definition of a rel.~boundary isomorphism between $\mathfrak{s}\circ F$ and $\mathfrak{s}\circ F'$. Thus, in the notation of the proof of Lemma~\ref{lem:spinbusiness}, we have $\Theta(F)=\theta(\mathfrak{s}\circ F,\mathfrak{s})=\theta(\mathfrak{s}\circ F',\mathfrak{s})=\Theta(F')$, where the first and last equalities use item two of Lemma~\ref{lem:spinbusiness}.
\end{proof}

\subsection{Poincar\'{e} variations}\label{subsec:variational}

A \emph{variation} is a homomorphism of homology groups associated, as follows, to a continuous self-map $F\colon (A,B)\to (A,B)$  of a pair of topological spaces, that restricts to the identity map on $B$; cf.~\cite{MR390278, MR385872, MR375337}. For $i\geq 0$, let $x\in H_i(A,B)$ and choose $\Sigma\in C_i(A,B)$ a representative cycle for $x$. Because $F$ restricts to the identity on $B$, the chain $\Sigma-F_*(\Sigma)$ is a cycle in $C_i(A)$ and thus defines a homology class. We claim this homology class is independent of the choice of $\Sigma$. Indeed, if $\Sigma'$ is another choice of representative cycle for $x$, then $\Sigma-\Sigma'=\partial M$, for some $M\in C_{i+1}(A,B)$, and so
\[
\Sigma-F_*(\Sigma)-(\Sigma'-F_*(\Sigma'))=(\Sigma-\Sigma')-F_*(\Sigma-\Sigma')=\partial M-F_*(\partial M)=\partial (M-F_*(M)).
\]
Again, as $F$ is the identity on $B$, any choice of lift of $M$ to $C_{i+1}(A)$ gives the same chain $M-F_*(M) \in C_{i+1}(A)$, with the boundary as just computed. Thus the homology class $[\Sigma-F_*(\Sigma)]$ is independent of the choice of $\Sigma$. For each $i\geq 0$, we thus obtain a well-defined homomorphism
\[
\Delta_F\colon H_i(A,B)\to H_i(A);\qquad\qquad [\Sigma]\mapsto [\Sigma-F_*(\Sigma)].
\]
The variation map $\Delta_F$ is invariant under homotopies of $F$ which fix $B$.

While variations are defined in this generality, we shall  apply them to 4-manifolds and when $i=2$. In this setting they satisfy an extra condition, as shown by Saeki~\cite{MR2197449}.
Following Saeki, we define an appropriate value group for these invariants.
%, which will turn out to depend only on the intersection form of $X$.
%commented out, already said below.
Write $(-)^{\ast}:=\Hom(-,\Z)$. Given a homomorphism $\Delta\colon H_2(X,\partial X)\to H_2(X)$, define a composite ``umkehr'' homomorphism
\[
\Delta^!\colon H_2(X,\partial X)\xrightarrow{\PD^{-1}} H^2(X)\xrightarrow{\ev} H_2(X)^{\ast}
\xrightarrow{\Delta^{\ast}} H_2(X,\partial X)^{\ast}\xrightarrow{\ev^{-1}}H^2(X,\partial X)\xrightarrow{\PD}H_2(X).
\]

Write $j_*\colon H_2(X)\to H_2(X,\partial X)$ for the map induced by inclusion.

\begin{definition}\label{def:algebraicvariational}
A homomorphism $\Delta\colon H_2(X,\partial X)\to H_2(X)$ is a \emph{Poincar\'{e} variation} if
\[
\Delta+\Delta^!=\Delta\circ j_*\circ\Delta^!\colon H_2(X,\partial X)\to H_2(X).
\]
Write $\mathcal{V}(H_2(X),\lambda_X)$ for the set of Poincar\'{e} variations.
\end{definition}

\begin{lemma}[{\cite[Lemma~3.5]{MR2197449}}]\label{lem:saekigroup}
 The set $\mathcal{V}(H_2(X),\lambda_X)$ forms a group under the operation
\[
\Delta_1\ast\Delta_2:=\Delta_1+(\Id-\Delta_1\circ j_*)\circ \Delta_2,
\]
with identity element $0$ and inverses $\Delta^{-1}=-(\Id-\Delta^*\circ j_*)\circ \Delta$.
\end{lemma}

\begin{definition}\label{def:variationalmorphism}
Let $X$ be a compact, simply connected 4-manifold.  Given $F\in\Homeo^+(X,\partial X)$, define the \emph{variation} of $F$ to be the homomorphism
\[
\Delta_F\colon H_2(X,\partial X)\to H_2(X);\qquad [\Sigma]\mapsto [\Sigma-F(\Sigma)].
\]
It is shown in \cite[Lemma~3.2(4)]{MR2197449} that $\Delta_F \in \mathcal{V}(H_2(X),\lambda_X)$, i.e.\ it is a Poincar\'{e} variation in the sense of Definition~\ref{def:algebraicvariational}.
%We show in Corollary~\ref{cor:autwelldef} that
%\[
%\pi_0\Homeo^+(X, \partial X) \to\mathcal{V}(H_2(X),\lambda_X);\qquad [F] \mapsto \Delta_F
%\]
%is a well-defined homomorphism.
\end{definition}

\begin{remark}
The notation $\mathcal{V}(H_2(X),\lambda_X)$ is intended to convey that this group depends only on the isometry class of the form $(H_2(X),\lambda_X)$. This assertion is proved in Section~\ref{sec:isometries}, where we recast the group of Poincar\'{e} variations more algebraically.
\end{remark}

\begin{remark}\label{rem:recall}
It is straightforward to see that
\[
\Id-\Delta_F\circ j_*\colon H_2(X)\to H_2(X)
\]
agrees with the map $F_*\colon H_2(X)\to H_2(X)$. We can thus view the Poincar\'{e} variation of a self-homeomorphism as a refinement of the isometry induced by $F$ on the intersection form. Indeed, when $\partial X=\emptyset$, it is easily checked that the definitions of a Poincar\'{e} variation and an isometry coincide. We study the relationship of variations to isometries in greater depth in Section~\ref{sec:isometries}.
\end{remark}

\begin{lemma}\label{lem:welldefhtpy}
 Let $(A,B)$ be a pair of spaces and let $F_0,F_1\colon (A,B)\to (A,B)$ be continuous maps each restricting to the identity on $B$. If $F_0\simeq F_1$ via a homotopy restricting to the identity on $B$,  then $\Delta_{F_0}=\Delta_{F_1}$.
\end{lemma}

\begin{proof}
By assumption, there exists a continuous map $\Psi\colon (A\times[0,1],B\times[0,1])\to (A\times[0,1],B\times[0,1])$  restricting to the identity on $B\times[0,1]$, such that $\Psi(x,i)=(F_i(x),i)$ for $i=0,1$. It is easily seen from the definitions that there is a commuting diagram
\[
\begin{tikzcd}
H_2(A\times\{0\},B\times\{0\})\oplus H_2(A\times\{1\},B\times\{1\})\ar[rr, "(1\,\,-1)"] \ar[d, "\Delta_{F_0}\oplus \Delta_{F_1}"]
&& H_2(A\times[0,1],B\times[0,1]) \ar[d, "\Delta_{\Psi}"]
\\
H_2(A\times\{0\})\oplus H_2(A\times\{1\})\ar[rr, "(1\,\,-1)"]
&& H_2(A\times[0,1]),
\end{tikzcd}
\]
where the horizontal arrows are the inclusion induced maps.
Take $(x,x) \in H_2(A\times\{0\}, B\times\{0\})\oplus H_2(A\times\{1\}, B\times\{1\})$. Mapping $(x,x)$ right, sends it to 0, so mapping $(x,x)$ right and then down sends $(x,x)$ to $0$. Whereas mapping down and then right sends $(x,x)$ to $\Delta_{F_0}(x) -\Delta_{F_1}(x)$. From this and commutativity we see that $\Delta_{F_0}=\Delta_{F_1}$.
\end{proof}

\begin{corollary}\label{cor:autwelldef}
The function
\[
\widetilde{\pi}_0\left(\Homeo^+(X,\partial X)\right)\to \mathcal{V}(H_2(X),\lambda_X);\qquad [F]\mapsto \Delta_F
\]
is a well-defined group homomorphism.
\end{corollary}
\begin{proof}
Suppose $F,G\in\Homeo^+(X,\partial X)$ are pseudo-isotopic rel.~boundary. Then in particular they are homotopic rel.~boundary, implying $\Delta_F=\Delta_{G}\in\mathcal{V}(H_2(X),\lambda_X)$ by Lemma~\ref{lem:welldefhtpy}. This shows the function $[F]\mapsto \Delta_F$ is well-defined on the level of sets. It is straightforward to compute that this is a group homomorphism; for this we refer to e.g.~\cite[Lemma~3.2]{MR2197449}, cf.~\cite[3.1.4]{MR390278}.
\end{proof}

\section{The main pseudo-isotopy classification and proof of Theorems~\ref{theoremA} and~\ref{theoremB}}\label{sec:newsection}

In this section we describe Theorem~\ref{thm:main}, the main technical result of the article, which will be proven in Section~\ref{sec:proofmain}. Assuming this result is true, we will then devote the remainder of this section to proving Theorems~\ref{theoremA} and~\ref{theoremB} from the introduction.

\subsection{The main pseudo-isotopy classification}\label{section:main-PI-classification}

The following is the main technical result of the article, in which we show that the invariants defined in Section~\ref{sec:pseudoisotopy} completely detect the pseudo mapping class groups, in both categories.

\begin{theorem}\label{thm:main}
Suppose $\CAT$ is $\Diff$ or $\TOP$, and let $X$ be a compact, simply connected, oriented $\CAT$ $4$-manifold with (possibly empty) boundary.
\begin{enumerate}[leftmargin=*]\setlength\itemsep{0em}
\item\label{item:spincase-main} When $X$ is spin, and respectively when $\CAT$ is $\TOP$ or $\Diff$, the homomorphisms
\[
\begin{array}{ccl}
\widetilde{\pi}_0\Homeo^+(X,\partial X)&\hookrightarrow&H^1(X,\partial X;\Z/2)\times\mathcal{V}(H_2(X),\lambda_X), \quad{and}\\
\widetilde{\pi}_0\Diffeo^+(X,\partial X)&\hookrightarrow&H^1(X,\partial X;\Z/2)\times\mathcal{V}(H_2(X),\lambda_X),
\end{array}
\]
both given by $F\mapsto (\Theta(F),\Delta_F)$, are injective.
\item \label{item:nonspincase-main}
When $X$ is not spin, and respectively when $\CAT$ is $\TOP$ or $\Diff$, the homomorphisms
\[
\begin{array}{ccl}
\widetilde{\pi}_0\Homeo^+(X,\partial X)&\hookrightarrow&\mathcal{V}(H_2(X),\lambda_X), \quad{and}\\
\widetilde{\pi}_0\Diffeo^+(X,\partial X)&\hookrightarrow&\mathcal{V}(H_2(X),\lambda_X),
\end{array}
\]
both given by $F\mapsto \Delta_F$, are injective.
\end{enumerate}
\end{theorem}

\subsection{Realising obstructions to topological pseudoisotopy}\label{section:realising-obstructions}

In this section we realise the obstructions in Theorem~\ref{thm:main} when $\CAT=\TOP$.
%This completes the computation of the topological mapping class groups promised in Theorem~\ref{theoremA}.
Our objective is to prove the maps from Theorem~\ref{theoremA} are surjective. We begin by discussing some alternative interpretations of the~$\Theta$ invariant.

Let $X$ be a compact, oriented, spin $4$-manifold with spin structure $\mathfrak{s}$. Let $F\in \Homeo^+(X,\partial X)$. Suppose $X$ has at least two boundary components, and choose a properly embedded arc $\gamma\subseteq X$ with endpoints on different boundary components. Let $\mathfrak{t}$ be the spin structure on $D:=\partial (X\times[0,1])$ given by $\mathfrak{t}(x,t)=\mathfrak{s}\circ F$ when $t=1$ and $\mathfrak{t}=\mathfrak{s}$ otherwise. Consider the loop
\[
\delta:=\gamma\times\{1\}\cup \partial\gamma\times[0,1]\cup \gamma\times\{0\}\subseteq D.
\]
Recall $\nu_\delta\colon \delta\to \BSO$ denotes the stable normal vector bundle; note that $\delta$ has a normal vector bundle $\nu(\delta\subseteq D)$ by \cite[Theorem~9.3A]{FQ}. Thus also $\delta$ has a stable normal vector bundle, and also $\gamma \times [0,1] \subseteq X \times [0,1]$ has a normal vector bundle.

\begin{lemma}\label{lem:collection}
The following are equivalent.
\begin{enumerate}[leftmargin=*]\setlength\itemsep{0em}
\item\label{item:collection1} $\Theta(F,\mathfrak{s})([\gamma])=0$.
\item\label{item:collection2} There exists an embedded oriented surface $\Sigma\subseteq X\times [0,1]$ with $\partial \Sigma= \delta$, and such that $\mathfrak{t}|_\delta$ extends to a spin structure $\mathfrak{S}\colon \Sigma\to \BTOPSpin$ on the stable normal bundle $\nu_\Sigma$.
\item\label{item:collection3} The spin structure on $\nu_\delta$ determined by $\mathfrak{t}$ extends over the normal bundle $\nu_{\gamma \times [0,1]}$ of $\gamma \times [0,1] \subseteq X \times [0,1]$.
%is null-bordant.
\item\label{item:collection4} The stable framing on $\nu_\delta$ determined by $\mathfrak{t}$  extends over $\nu_{\gamma \times [0,1]}$.
%is null-bordant.
\item\label{item:collection5} The unique framing of the normal $3$-plane bundle $\nu(\delta\subseteq D)$ that is compatible with $\mathfrak{t}$  extends over $\nu(\gamma \times [0,1] \subseteq X \times [0,1])$.
% is null-bordant.
\end{enumerate}
\end{lemma}

\begin{proof}
It follows from the definition of $\Theta(F,\mathfrak{s})([\gamma])$ that (\ref{item:collection1}) is equivalent to (\ref{item:collection2}).
To see that (\ref{item:collection2})$\implies$(\ref{item:collection3}) note that any two surfaces in $X \times [0,1]$ with boundary $\delta$ induce the same spin structure on $\nu_\delta$, because their union is a closed surface in a spin manifold. Hence $\mathfrak{t}$ extends over $\nu_{\Sigma}$ if and only if it extends over  $\nu_{\gamma \times [0,1]}$.   Conversely,
%(\ref{item:collection3}) implies that we may extend the spin structure over a disc. So,
using the disc $\Sigma:=\gamma\times[0,1]$, we see (\ref{item:collection3})$\implies$(\ref{item:collection2}).

The equivalence of (\ref{item:collection3}) and (\ref{item:collection4}) is implied by the standard one to one correspondence between spin structures and stable framings for a closed 1-manifold, and this correspondence is justified as follows. A choice of null homotopy of a loop $S^1\to \BSTOP$ fixes a choice of lift $S^1\to \BTOPSpin$. Conversely, since any map $S^1\to \BTOPSpin$ is null homotopic, and two null homotopies differ by an element of $\pi_2(\BTOPSpin)=0$, there is (up to homotopy) a unique null homotopy, which projects under $\BTOPSpin\to \BSTOP$  to a framing.

We show (\ref{item:collection4}) $\iff$ (\ref{item:collection5}). First, we recall why there is a unique framing on $\nu(\delta\subseteq D)$, compatible with a \emph{stable} spin structure $\mathfrak{t}$. Frame $\nu(\delta\subseteq D)$ compatibly with $\mathfrak{t}$ in the sense that the stabilisation by the trivial bundle $\nu(\delta\subseteq D)\oplus\underline{\R}^\infty=\nu_\delta$, with the corresponding stabilised framing, is compatible with $\mathfrak{t}$ via the correspondence described in the previous paragraph. To show only one of the framings on $\nu(\delta\subseteq D)$ could have been used here, consider that two choices of framing on a $3$-plane bundle $S^1\to \BSTOP(3)$ differ by an element of $\pi_2(\BSTOP(3))=\Z/2$. As the stable range for $\pi_2(\BSTOP(n))=\Z/2$ is $n\geq 3$, isomorphism classes of framing for the corresponding stable bundle correspond to exactly one of these framings on the $3$-plane bundle. Now exactly one of the framings on $\nu(\delta\subseteq D)$ extends over the trivial $3$-plane bundle $\nu(\gamma\times[0,1]\subseteq X\times[0,1])$. Choosing this framing and stabilising $\nu(\gamma\times[0,1]\subseteq X\times[0,1])$ by $\underline{\R}^\infty$, we see that the stabilisation of this framing extends over the stable normal bundle of $\gamma\times[0,1]$. As exactly one of the stable framings on $\nu_\delta$ has this property, this completes the proof of (\ref{item:collection4}) $\iff$ (\ref{item:collection5}).
\end{proof}

We slightly expand the class of automorphisms $F\colon X\to X$ for which a variation $\Delta_F$ is defined.

\begin{definition}\label{def:variationalmorphism2}
Suppose $F\colon X\to X$ is an orientation preserving homeomorphism of a compact, oriented, simply connected, topological $4$-manifold with connected boundary. Suppose there exists a codimension zero submanifold $A\subseteq \partial X$ such that the inclusion induced map $H_r(A)\to H_r(\partial X)$ is an isomorphism for $r=1,2$ and such that $F|_A=\Id_A$. Let $x\in H_2(X,\partial X)$ and choose a representative cycle $\Sigma\in C_2(X,A)$ for $x$. Define the \emph{variation} of $F$ as the homomorphism
\[
\Delta_F\colon H_2(X,\partial X)\to H_2(X);\qquad [\Sigma]\mapsto [\Sigma-F_*(\Sigma)].
\]
\end{definition}

The proof that this is a well-defined homomorphism, and is moreover a Poincar\'{e} variation in the sense of Definition~\ref{def:algebraicvariational} is identical to the proof for when $F|_{\partial X}=\Id_{\partial X}$, noting that we are using representative cycles in $(X,A)$, where $F|_A=\Id_A$.

With this definition in hand we can produce a slight extension of \cite[Theorem~3.7]{MR2197449}, that we will require momentarily.

\begin{theorem}\label{thm:beefupSaeki} Let $X$ be a compact, oriented, simply connected, topological $4$-manifold with connected boundary and let $\Delta\in\mathcal{V}(H_2(X),\lambda_X)$. Let $f\colon \partial X\to \partial X$ be an orientation preserving homeomorphism with $f|_A=\Id_A$ for some codimension zero submanifold $A\subseteq \partial X$ such that the inclusion induced map $H_r(A)\to H_r(\partial X)$ is an isomorphism for $r=1,2$. Then there exists an orientation preserving homeomorphism $F\colon X\to X$ restricting to $f$ on the boundary and such that $\Delta_F=\Delta$.
\end{theorem}

\begin{proof}It is proved in \cite[Theorem~3.7]{MR2197449} that $\Delta$ can be realised smoothly stably, by a diffeomorphism restricting to the identity on the boundary. In fact, the same proof can be used to construct $F\in\Homeo^+(X,\partial X)$ such that $\Delta_F=\Delta$, as mentioned in \cite[Remark~3.13]{MR2197449}. To instead produce an orientation preserving homeomorphism that restricts to $f$ of the type we desire, the only modification that must be made is that (in the language of that proof), one must set $V$ to be the twisted double of the $4$-manifold, using the map $f$. All other parts of Saeki's proof can be used verbatim.
\end{proof}

We now describe a construction that will help us transform problems about manifolds with more than one boundary component to problems about manifolds with connected boundary.

\begin{construction}\label{constr:internal}
Let $X$ be a compact, oriented, simply connected, topological $4$-manifold with $r>1$ connected boundary components. Enumerate the boundary components $Y_1,\dots,Y_r$. Choose disjoint properly embedded arcs $\gamma_1,\dots,\gamma_{r-1}\subseteq X$ so that $\gamma_i$ has one end on $Y_i$ and the other end on $Y_r$. Excise tubular neighbourhoods of the arcs to obtain a simply connected $4$-manifold $X_0:=X\sm\left(\bigcup_i\nu(\gamma_i)\right)$, with connected boundary $\partial X_0\cong Y_1\#Y_2\#\dots\# Y_r$. We call any $(X_0,\partial X_0)$ obtained in this way an \emph{internal connected sum} manifold for $(X,\partial X)$.
\end{construction}

By a straightforward Mayer-Vietoris calculation, the inclusion of an internal connected sum manifold $\iota\colon (X_0,\partial X_0)\subseteq (X,\partial X)$ results in an isomorphism of exact sequences
\begin{equation}\label{eq:commutingdiagram}
\begin{tikzcd}
0 \ar[r]
& H_2(\partial X_0) \ar[r]\ar[d,"\Id"]
& H_2(X_0) \ar[r, "(j_0)_*"]\ar[d,"\cong"', "\iota_*"]
& H_2(X_0, \partial X_0) \ar[r] \ar[d,"\cong"', "\iota_*"]
& H_1(\partial X_0) \ar[r]\ar[d,"\Id"]
& 0\\
0 \ar[r]
& H_2(\partial X) \ar[r]
& H_2(X) \ar[r, "j_*"]
& H_2(X, \partial X) \ar[r]
& H_1(\partial X) \ar[r]
& 0.
\end{tikzcd}
\end{equation}

\begin{lemma}\label{lem:internalvariational}
An internal connected sum induces an isomorphism under inclusion
\[
\mathcal{V}(H_2(X_0),\lambda_{X_0})\xrightarrow{\cong} \mathcal{V}(H_2(X),\lambda_{X});\qquad \Delta_0\mapsto \iota_*\circ\Delta_0\circ(\iota_*)^{-1}.
\]
\end{lemma}

\begin{proof}
The inclusion induced map $\iota_*\colon H_2(X_0)\xrightarrow{\cong} H_2(X)$ is an isometry of intersection forms. Now apply Corollary~\ref{cor:thebusiness}, below.
\end{proof}

\begin{theorem}\label{thm:realisevariational}
Let $X$ be a compact, oriented, simply connected, topological $4$-manifold. Let $\Delta\in\mathcal{V}(H_2(X),\lambda_X)$ and fix $x\in H^1(X,\partial X;\Z/2)$. Then there exists $F\in\Homeo^+(X,\partial X)$ such that $\Delta_F=\Delta$. If $X$ is spin, it is possible to choose $F$ so that moreover $\Theta(F)=x$.
\end{theorem}

\begin{proof}%[Proof of Theorem~\ref{thm:realisevariational}]
In the case that $\partial X$ is empty, a Poincar\'{e} variation is exactly an isometry of the intersection form, and in this case the theorem is due to Freedman~\cite[Theorem~1.5,~Addendum]{F}.
In the case that the boundary is nonempty and connected, this theorem is due to Saeki; see \cite[Theorem~3.7]{MR2197449} and \cite[Remark~3.13]{MR2197449}.

We now assume that $X$ has $r>1$ connected boundary components. Enumerating these $Y_1,\dots,Y_r$, we form an internal connected sum manifold $X_0$ with boundary $\partial X_0\cong Y_1\#\dots\# Y_r$ as described in Construction~\ref{constr:internal}. Write $\iota\colon (X_0,\partial X_0)\to (X,\partial X)$ for the inclusion. Denote by $\Delta_0\in\mathcal{V}(H_2(X),\lambda_X)$ the variation sent to $\Delta$ under the inclusion induced isomorphism of Lemma~\ref{lem:internalvariational}.

We prove the first claim of the theorem. As the case of connected boundary is already proved, there exists $F_0\in\Homeo^+(X_0,\partial X_0)$ such that $\Delta_{F_0}=\Delta_0$. Extend $F_0$ by the identity map $\Id\colon \bigcup_i\nu(\gamma_i)\to \bigcup_i\nu(\gamma_i)$ to obtain $F\in\Homeo^+(X,\partial X)$. We claim that $\Delta_{F}=\Delta$, or equivalently that $\Delta_F\circ\iota_*=\iota_*\circ \Delta_{F_0}$. To see this, let $z\in H_2(X_0,\partial X_0)$, and represent it by a cycle $\Sigma\in C_2(X_0,\partial X_0)$. We compute that $\iota_*\circ\Delta_{F_0}(x)=\iota_*([\Sigma-F_0(\Sigma)])=[\iota(\Sigma)-\iota\circ F_0(\Sigma)]=[\iota(\Sigma)-F(\iota(\Sigma))]$. But $\iota(\Sigma)\in C_2(X,\partial X)$ is a representative $2$-cycle for $\iota_*(z)$ and hence $[\iota(\Sigma)-F(\iota(\Sigma))]=\Delta_F(\iota_*(z))$, as was required to be shown. This completes the proof of the first claim in the theorem.

For the second claim in the theorem, assume that $X$ is spin. For each $i$, choose a framing $\overline{\nu} (\gamma_i)\cong D^3\times[0,1]$ compatible with the stable framing induced by the spin structure. Write $\phi_t\in \pi_1(\SO(3))$ for a smooth generating loop based at the identity. Denote by $\Phi\colon \bigcup_i\overline{\nu}(\gamma_i)\to \bigcup_i\overline{\nu}(\gamma_i)$ the self-homeomorphism that, under the identification $\overline{\nu} (\gamma_i)\cong D^3\times[0,1]$ is the map $(p,t)\mapsto (\phi_t(p),t)$ when $x([\gamma_i])=1$, and is the identity otherwise. The framing identifies the sphere bundle of $\nu(\gamma_i)$ as $S^2\times[0,1]$, and in $\partial X_0\cong Y_1\#Y_2\#\dots\# Y_r$ these correspond to bicollars on the connected sum spheres. Thus $\Phi|_{S(\nu(\gamma_i))}$ extends by the identity map to a self-diffeomorphism of $\partial X_0$ which is the identity except near the connected sum $2$-spheres corresponding to $x([\gamma_i])=1$, where it is a Dehn twist on $S^2\times[0,1]$. Write $f\colon \partial X_0\to \partial X_0$ for this extended map.
Writing $A\subseteq \partial X$ for the complement of the neighbourhoods of the connected sum spheres, and noting $H_r(A)\to H_r(X_0)$ is an isomorphism for $r=1,2$, we may apply Theorem~\ref{thm:beefupSaeki} to obtain an orientation preserving homeomorphism $F'_0\colon X_0\to X_0$ restricting to $f$ on the boundary and such that $\Delta_{F'_0}=\Delta$. Extend $F'_0$ to a homeomorphism $F'\colon X\to X$ by using the map $\Phi\colon \bigcup_i\overline{\nu}(\gamma_i)\to \bigcup_i\overline{\nu}(\gamma_i)$. Because $\iota\circ F'_0=F\circ\iota$, the same argument we used above can be applied to prove that $\Delta_{F'}=\Delta$.

We now show that $\Theta(F')=x$. In the notation from the beginning of this subsection, the arc~$\gamma_i$ determines a loop $\delta_i\subseteq D:=\partial (X\times[0,1])$. The $D^3$-bundle $\overline{\nu}(\delta_i\subseteq D)$ is framed by $\mathfrak{t}$ (also defined at the start of this subsection), giving rise to an identification $\overline{\nu}(\delta_i\subseteq D)\cong D^3\times S^1$. By construction of the map $\Phi$, this framing extends to a framing of the $D^3$-bundle $D^3\times D^2$ when $x([\gamma_i])=0$. In other words the framing is nullbordant, so \eqref{item:collection1}$\iff$\eqref{item:collection5} of Lemma~\ref{lem:collection} implies $\Theta(F)(\gamma_i)=0$ in this case. For $x([\gamma_i])=1$, when defining $\Phi$ we added a twist over the arc in $S^1$ corresponding to $\gamma_i\times\{1\}$, so that for these $\delta_i$, the framing is not nullbordant. So \eqref{item:collection1}$\iff$\eqref{item:collection5} of Lemma~\ref{lem:collection} implies $\Theta(F)(\gamma_i)=1$ in this case. Using the fact that $\gamma_1,\dots,\gamma_{r-1}$  form a basis for $H^1(X,\partial X;\Z/2)$, we see that~$\Theta(F')=x$.
\end{proof}

\subsection{The proofs of Theorems~\ref{theoremA} and~\ref{theoremB}}\label{sec:ABproofs}

We are now in a position to prove Theorems~\ref{theoremA} and~\ref{theoremB} from the introduction, provided that we assume Theorem~\ref{thm:main} is true.

We begin by recalling a theorem of Perron~\cite{MR862426} and Quinn~\cite{Quinn:isotopy} (cf.~Gabai~\cite[Theorem~2.5]{Gabai-schoenflies} for \eqref{item:quinn2}), which can be combined with Theorem~\ref{thm:main} to yield stronger statements.

\begin{theorem}[Perron-Quinn]\label{thm:quinn}
Let $X$ be a compact, simply connected, topological 4-manifold.
\begin{enumerate}[leftmargin=*]\setlength\itemsep{0em}
\item\label{item:quinn1} Let $F \colon X \xrightarrow{\cong} X$ be a homeomorphism that is pseudo-isotopic to the identity rel.~boundary. Then $F$ is topologically isotopic to the identity rel.~boundary.
\item\label{item:quinn2} Suppose that $X$ admits a smooth structure, and let $F \colon X \xrightarrow{\cong} X$ be a diffeomorphism that is smoothly pseudo-isotopic to the identity rel.~boundary. Then $F$ is smoothly stably isotopic to the identity rel.~boundary.
\end{enumerate}
\end{theorem}

Theorem~\ref{thm:quinn} (\ref{item:quinn1}) shows that the forgetful map
$\pi_0\Homeo^+(X,\partial X) \xrightarrow{\cong}\widetilde{\pi}_0\Homeo^+(X,\partial X)$ from the pseudo mapping class group to the mapping class group is an isomorphism when $X$ is simply connected.

%Combining Theorem~\ref{thm:quinn}~\eqref{item:quinn1} with Theorem~\ref{thm:main} proves the injectivity in Theorem~\ref{theoremA}, which implies injectivity in Corollary~\ref{theoremA'}.

We now prove the main topological theorem of the paper, whose statement we recall for the convenience of the reader.

\begin{theoremA}
Let $X$ be a compact, simply connected, oriented, topological $4$-manifold with (possibly empty) boundary.
\begin{enumerate}[leftmargin=*]\setlength\itemsep{0em}
\item \label{item:spincase-intro}
When $X$ is spin, the map $F\mapsto  (\Theta(F),\Delta_F)$ induces a group isomorphism
\[
\pi_0\Homeo^+(X,\partial X)\xrightarrow{\cong}H^1(X,\partial X;\Z/2)\times\mathcal{V}(H_2(X),\lambda_X).
\]
\item \label{item:nonspincase-intro}
When $X$ is not spin, the map $F\mapsto  \Delta_F$ induces a group isomorphism
\[
\pi_0\Homeo^+(X,\partial X)\xrightarrow{\cong}\mathcal{V}(H_2(X),\lambda_X).
\]
\end{enumerate}
\end{theoremA}

\begin{proof}
Combine Theorem~\ref{thm:main} with the Perron-Quinn result (Theorem~\ref{thm:quinn}~\eqref{item:quinn1}) that pseudo-isotopy implies isotopy for simply connected, compact 4-manifolds, to obtain injectivity of the maps in the statement. Surjectivity is provided by Theorem~\ref{thm:realisevariational}.
\end{proof}

Next we use Theorem~\ref{thm:quinn}~\eqref{item:quinn2} to deduce Theorem~\ref{theoremB} from the introduction, computing the smooth stable mapping class group.

\begin{theoremB}\label{thm:stable-isotopy}
Let $X$ be a compact, simply connected, oriented, smooth 4-manifold.
\begin{enumerate}[leftmargin=*]\setlength\itemsep{0em}
\item \label{item:spincase-intro-thmB-2}
When $X$ is spin, the map $F \mapsto  (\Theta(F),\Delta_F)$ induces a group isomorphism
\[
\underset{g\to\infty}{\colim}\, \pi_0\Diffeo^+(X \#^g S^2 \times S^2,\partial X) \xrightarrow{\cong} H^1(X,\partial X;\Z/2)\times \underset{g\to\infty}{\colim}\, \mathcal{V}((H_2(X),\lambda_X)\oplus\mathcal{H}^{\oplus g}).
\]
\item \label{item:nonspincase-intro-thmB-2}
When $X$ is not spin, the map $F \mapsto  \Delta_F$ induces a group isomorphism
\[
\underset{g\to\infty}{\colim}\, \pi_0\Diffeo^+(X \#^g S^2 \times S^2,\partial X) \xrightarrow{\cong} \underset{g\to\infty}{\colim}\, \mathcal{V}((H_2(X),\lambda_X)\oplus\mathcal{H}^{\oplus g}).
\]
\end{enumerate}
\end{theoremB}

\begin{proof}
 An element of the colimit in the domain is represented by a diffeomorphism  $F$ of $X \#^g S^2 \times S^2$ for some $g$.  For injectivity, assume its image in $\mathcal{V}((H_2(X),\lambda_X)\oplus\mathcal{H}^{\oplus g})$ is trivial, and in the spin case assume additionally that $\Theta(F)=0$.  Then by Theorem~\ref{thm:main}, $F$ is smoothly pseudo-isotopic to $\Id_{X \#^g S^2 \times S^2}$ rel.~boundary.  By Theorem~\ref{thm:quinn}~\eqref{item:quinn2}, $F$ is stably smoothly isotopic rel.\ boundary to the identity, meaning that for some $h$ the image of $F$ in $\pi_0\Diffeo^+(X \#^{g+h} S^2 \times S^2)$ is trivial. It follows that $[F]$ is trivial in the colimit.

 For surjectivity, fix an element of the colimit in the codomain. This is represented by a Poincar\'{e} variation $\Delta\in \mathcal{V}((H_2(X),\lambda_X)\oplus\mathcal{H}^{\oplus g})$ for some~$g$, and in the spin case together with an element $x\in H^1(X,\partial X;\Z/2)$.  By Theorem~\ref{theoremA}, there is a homeomorphism $F'$ of $X \#^g S^2 \times S^2$ restricting to $\Id_{\partial X}$ inducing $\Delta$, respectively $(x,\Delta)$ in the spin case.

 Next, by \cite[Theorem~8.6]{FQ}, there exists an $h \in \mathbb{N}$ such that $F' \# \Id \in \Homeo^+(X \#^{g+h} S^2 \times S^2)$ is topologically isotopic to a diffeomorphism if and only if the Casson-Sullivan invariant $\operatorname{cs}(F') \in H^3(X \#^g S^2 \times S^2,\partial;\Z/2)$ vanishes; see \cite[Proposition~2.23]{GalvinCS} for more details. Since  $H^3(X \#^g S^2 \times S^2,\partial;\Z/2) \cong H_1(X \#^g S^2 \times S^2;\Z/2) =0$, we have that $\operatorname{cs}(F')=0$, and so there is an integer $h$ as above.
 We have that~$[F]$ maps to $[\Delta_F]=[\Delta_F']=[\Delta] \in \colim_{g\to\infty} \mathcal{V}((H_2(X),\lambda_X)\oplus\mathcal{H}^{\oplus g})$, completing the proof of surjectivity.
\end{proof}

 %
%there is a lift of the associated map of stable tangent microbundles to a map of the corresponding smooth stable tangent bundles.  We now show such a lift exists.
%Write $X' := X \#^g S^2 \times S^2$, let $\tau_{X'} \colon X' \to \BTOP$ be a classifying map for the stable tangent microbundle of $X'$, and let $T_{X'} \colon X' \to \BO$ classify the tangent bundle of $X'$.  Since $F' \colon X' \to X'$ is a homeomorphism, $\tau_{X'} \circ F$ also classifies the tangent microbundle and so there is a homotopy $X'\times I \to \BTOP$ between $\tau_{X'}$ and $\tau_{X'} \circ F'$. On $\partial(X'\times I)$, since $X'$ is smooth and since $F'$ restricts to the identity on $\partial X$, we can lift this map along the fibration $\BO \to \BTOP$. The homotopy fibre of $\BO \to \BTOP$ is $\TOP/\operatorname{O}$. We want to extend this lift over all of $X' \times I$. Since $X \times I$ is 5-dimensional and there is a $7$-connected map $\TOP/\operatorname{O} \to K(\Z/2,3)$, there is a unique obstruction in $H^4(X' \times I,\partial(X' \times I);\Z/2) \cong H_1(X' \times I;\Z/2) =0$ to extending the homotopy. Thus there is in fact no obstruction and a homotopy exists as desired.
%
% It follows that there is an integer $h$ as above such that $F' \# \Id$ is isotopic to a diffeomorphism,~$F$ say. We have that~$[F]$ maps to $[\Delta_F]=[\Delta_F']=[\Delta] \in \colim_{g\to\infty} \mathcal{V}((H_2(X),\lambda_X)\oplus\mathcal{H}^{\oplus g})$, completing the proof of surjectivity.

\section{Elements of modified surgery theory}\label{section:elements-modified-surgery}

To prove the main pseudo-isotopy classification Theorem~\ref{thm:main}, we will use the ideas of Kreck's modified surgery theory~\cite{MR1709301}. We will now recall, in some generality, relevant concepts from this theory. After this we make some specific calculations that we will need later.

\begin{notation}
When $\CAT=\Diff$ (resp.~$\TOP$), we introduce the notation $\SCAT$ and $\CATSpin$ to indicate $\SO$ and $\Spin$ (resp.~$\STOP$ and $\TOPSpin$). Similarly for the corresponding classifying spaces, $\BSCAT$ and $\BCATSpin$.
\end{notation}

\subsection{Definitions and theorems from modified surgery theory}

In this section, all manifolds are compact and oriented.

Suppose $\CAT=\Diff$. A \emph{cobordism rel.~boundary} between $n$-dimensional manifolds with boundary $X_0$ and $X_1$ is a triple $(W;\phi_0,\phi_1)$. Here $W$ is an $(n+1)$-dimensional manifold with corners, with a decomposition
\[
\partial W=\partial_0W\cup(Y\times[0,1])\cup-\partial_1W,
\]
where the interiors of $\partial_0W$, $Y\times[0,1]$, and $\partial_1 W$ form the codimension 1 stratum of $W$, and the codimension 2 corners are $\partial_0W=Y$ and $\partial_1W=Y$. The maps $\phi_i$ are diffeomorphisms $\phi_i\colon (X_i,\partial X_i)\to (\partial_i W,Y)$. In case that $\partial X_i=\emptyset$ for $i=0,1$ (and thus $Y=\emptyset$), a cobordism rel.~boundary is called a \emph{cobordism}. In a common abuse of notation, we will often suppress the maps $\phi_0$ and $\phi_1$ from the notation, however we include this level of detail in the definition because we will use it later on. When $\CAT=\TOP$, the same definitions are used, except it is simpler as now the decomposition of $\partial W$ only refers to a decomposition by codimension 0 submanifolds, without reference to the concept of corners.

For a smooth $n$-manifold $X$ we will write the stable normal vector bundle via its classifying map $\nu_X \colon X\to\BSO$. Given a pair $(B,p)$, where $B$ has the homotopy type of a CW complex and $p\colon B\to \BSO$ is a fibration, a map $\xi\colon X\to B$ is called a \emph{$B$-structure} on $X$ if the diagram
\[
\begin{tikzcd}
&& B\ar[d,"p"]\\
X\ar[rru,"\xi"]\ar[rr,"\nu_X"]&&\BSO
\end{tikzcd}
\]
commutes.
%up to homotopy.
When $X$ is a $\TOP$ manifold, the same definitions and notation are used, except with $\BSTOP$ in place of $\BSO$ and $\nu_X$ now representing the stable topological normal bundle (see e.g.~\cite[\textsection 7]{TheGuide} for definitions).

Two closed $\CAT$ $n$-manifolds with $B$-structure $(X_i,\xi_i)$, $i=0,1$ are \emph{cobordant over $B$} if there exists $(W,\Xi)$, consisting of a cobordism $W$ between $X_0$ and $X_1$, together with a $B$-structure~$\Xi$ such that $\partial (W,\Xi)=(X_0,\xi_0)\sqcup-(X_1,\xi_1)$, so that in particular the diagram
\[
\begin{tikzcd}
W  \ar[rr,"\Xi"]&& B\\
X_i\ar[rru,"\xi_i"']\ar[u,"\phi_i"]&&
\end{tikzcd}
\]
commutes for $i=0,1$.
The corresponding bordism group is denoted~$\Omega_n(B,p)$, or $\Omega^{\CAT}_n(B,p)$ to emphasise the category.

Fix a $B$-structure $\zeta\colon Y\to B$ on a closed $(n-1)$-manifold $Y$. Suppose $X$ is an $n$-manifold with boundary and that $f\colon \partial X\to Y$ is a homeomorphism. A \emph{$B$-structure rel.~$\zeta$} on $X$ (implicitly with respect to the chosen $f$) is a $B$-structure $(X,\xi)$ that restricts to $(\partial X,\zeta\circ f)$ on the boundary.
Given $B$-structures rel.~$\zeta$ on two such manifolds, $(X_i,\xi_i)$, $i=0,1$, we say they are \emph{cobordant rel.~$\zeta$} if there exists $(W,\Xi)$, consisting of a cobordism rel.~boundary $(W;\phi_0,\phi_1)$ between $X_0$ and $X_1$, with $\partial W=\partial_0W\cup(Y\times[0,1])\cup-\partial_1W$, and a $B$-structure $\Xi$ on $W$ that restricts to $\xi_0\circ\phi_0^{-1}$, $\xi_1\circ\phi_1^{-1}$, and $\zeta\circ\pr_1$ on $\partial_0W$, $\partial_1W$ and $Y\times[0,1]$ respectively.

We record the following lemma for use later on; the proof is straightforward from the definitions and we omit it (we note that a little care must be taken in the smooth case to show a bicollar on $Y$ is equivalent to the desired cornered structure).

\begin{lemma}\label{lem:sanitycheck}
%Let $\CAT=\Diff$.
Suppose $B$ is a space with the homotopy type of a CW complex and let $p\colon B\to\BSCAT$ be a fibration. Let $Y$ be a closed $\CAT$ $(n-1)$-manifold. For $i=0,1$, suppose that~$X_i$ is a $\CAT$ $n$-manifold and that $f_i\colon \partial X_i\to Y$ is a $\CAT$ isomorphism. Suppose there is a $B$-structure $(Y,\zeta)$ and $B$-structures rel.~$\zeta$ $(X_i,\xi_i)$. Define a closed $\CAT$ $n$-manifold with $B$-structure
\[
(X,\xi):=(X_0\cup-X_1, \xi_0\cup\xi_1)
\]
using the $\CAT$ isomorphism $f_1^{-1}\circ f_0\colon \partial X_0\to \partial X_1$ to glue along the boundary. Then $(X,\xi)\sim0\in \Omega^{\CAT}_n(B,p)$ is null-bordant if and only if $(X_0,\xi_0)$ is cobordant to $(X_1, \xi_1)$ rel.~$\zeta$.

\end{lemma}

Recall that a topological space $A$ is \emph{$m$-connected} if $\pi_k(A)=0$ for $1\leq k\leq m$ and is \emph{$m$-coconnected} if $\pi_k(A)=0$ for $k\geq m$. A map of spaces $f\colon A\to B$ is \emph{$m$-connected} if the homotopy cofibre (i.e.~the mapping cone) is $m$-connected; equivalently $f_*\colon\pi_k(A)\to \pi_k(B)$ is an isomorphism for $k<m$ and is surjective for $k=m$. A map of spaces $f\colon A\to B$ is \emph{$m$-coconnected} if the homotopy fibre is $m$-coconnected; equivalently $f_*\colon\pi_k(A)\to \pi_k(B)$ is an isomorphism for $k>m$ and is injective for~$k=m$.

\begin{definition}\label{def:normal2type}
A \emph{normal $2$-type} for a $\CAT$ manifold $X$ is a pair $(B,p)$, where~$B$ is a space with the homotopy type of a CW complex, $p\colon B\to \BSCAT$ is a $3$-coconnected fibration with connected fibre, and $(B,p)$ is such that there exists a $3$-connected $B$-structure $\overline{\nu}_X\colon X\to B$. Such a $B$-structure $(X,\overline{\nu}_X)$ is called a \emph{normal $2$-smoothing} of $X$. If $(X,\overline{\nu}_X)$ is moreover a $B$-structure rel.~$\zeta$, for some $\zeta\colon Y\to B$ with $\partial X\cong Y$, we say $(X,\overline{\nu}_X)$ is a \emph{normal $2$-smoothing rel.~$\zeta$} of $X$.
\end{definition}

\begin{remark}\label{rem:moorepost}
Given a manifold $X$, the existence of a normal $2$-type follows from the theory of Moore-Postnikov decompositions; see~\cite{zbMATH03562120}. This theory furthermore guarantees that any two normal $2$-types for a given $X$ are fibre homotopy equivalent to one another.
\end{remark}

\begin{definition}\label{def:hcobordism} An \emph{$h$-cobordism rel.~boundary} between $(X_0, \partial X_0)$ and $(X_1,\partial X_1)$ is a cobordism rel.~boundary $(W;\phi_0,\phi_1)$ such that the inclusions $\partial_iW\hookrightarrow W$ are homotopy equivalences for $i=0,1$.
\end{definition}

The following definition is due to a technical component of modified surgery theory that will show up for us later.

\begin{definition}\label{def:w4triv} Given a map of CW complexes $p\colon B\to \BSCAT$ we say a homotopy class $\alpha\in \pi_4(B)$ is \emph{$w_4$-trivial} if the bundle $\xi$ classified by $p\circ\alpha\colon S^4\to \BSCAT$ has $w_4(\xi)=0$.
\end{definition}

In classical surgery theory~\cite{Wall-surgery-book}, obstructions to doing surgery to improve certain classes of maps to simple homotopy equivalences are situated in the \emph{surgery obstruction groups}. These are abelian groups $\displaystyle{L^s_n(\Z[\pi])}$, depending on the dimension $n$, and fundamental group $\pi$, of the manifold.
In modified surgery, Kreck considers a slightly altered version $\displaystyle{L^s_n(\Z[\pi],S)}$ of Wall's classical surgery obstruction groups $\displaystyle{L^s_n(\Z[\pi])}$ (see~\cite[\textsection 4]{kreckmonograph}). Here $S\subseteq \Z[\pi]$ is a subgroup, with certain properties, called a \emph{form parameter} by Bak~\cite{MR632404}.
Under certain circumstances, Kreck's modified surgery obstructions live in these alternative surgery obstruction groups, for example in the following theorem.

\begin{theorem}[{\cite[Theorem 5.2(b)]{kreckmonograph}}]\label{thm:kreckmain2} Let $\CAT$ be either $\Diff$ or $\TOP$. Let $(M_0,\partial M_0)$ and $(M_1,\partial M_1)$ be simply connected $5$-manifolds with boundary, each with normal $2$-type $(B,p)$, where $B$ is homotopy equivalent to a CW complex with finite $2$-skeleton. Let $\partial M_0\cong N\cong\partial M_1$ and fix a $B$-structure $\zeta\colon N\to B$. Then a bordism $(Z,\Xi)$ over $B$ rel.~boundary, between normal $2$-smoothings $(M_0,\xi_0)$ and $(M_1,\xi_1)$ rel.~$\zeta$, determines a surgery obstruction
\[
\theta(Z,\Xi)\in \left\{\begin{array}{ll}\displaystyle{L^s_6(\Z)}&\text{if $\alpha$ is $w_4$-trivial for all $\alpha\in \pi_4(B)$}\\
\displaystyle{L^s_6(\Z, S=\Z)}&\text{otherwise.}\end{array}\right.
\]
The obstruction $\theta(Z,\Xi)$ vanishes if and only if $(Z,\Xi)$ is cobordant rel.~boundary to some $(Z',\Xi')$ over $B$, where $Z'$ is an $h$-cobordism rel.~boundary.
\end{theorem}

\begin{remark}
In \cite[Theorem 5.2(b)]{kreckmonograph}, Kreck states the value groups for the surgery obstruction are rather certain intermediate $L$-groups $\displaystyle{L^{s,\tau}_6(\Z)}$ and $\displaystyle{L^{s,\tau}_6(\Z, S=\Z)}$. But by \cite[Lemma~4.1]{kreckmonograph}, these intermediate groups are isomorphic to the simple $L$-groups $\displaystyle{L^{s}_6(\Z)}$ and $\displaystyle{L^{s}_6(\Z, S=\Z)}$ (respectively), because the Whitehead group of the trivial group is trivial.
\end{remark}

\subsection{The normal 2-type of a compact oriented simply connected 4-manifold}\label{subsec:2type}

For this subsection, fix $X$, a compact, oriented, simply connected 4-manifold with (possibly empty) boundary.

In this section, whose content is known to experts, we will recall the normal $2$-type of $X$. In the context of $4$-manifold topology, the version of modified surgery that uses the normal $2$-type has been previously applied in manifold classification, for example in~\cite{MR2541758}. The $2$-type was successfully used by Kreck-Su~\cite{MR3687112} to classify certain $5$-manifolds with free fundamental group. For the problem of classifying automorphisms of a manifold up to pseudo-isotopy, the results of~\cite{MR561244} (which predate the full modified surgery theory) can be viewed in retrospect as a proto-version of the $2$-type approach that we employ here.

Let $K$ be a fixed model for $K(\pi_2(X),2)$ and choose a map $k\colon X\to K$ that is an isomorphism on $\pi_2$.
Let $w\in H^2(K;\Z/2)$ be the cohomology class such that $k^*w=w_2(X)$, and let $w\colon K\to K(\Z/2,2)$ be a map representing this class. Consider the following commutative diagrams
\[
\begin{tikzcd}[column sep = scriptsize]
\BTOPSpin\ar[r,"i"] \ar[d,"="]&B(X)^{\TOP} \arrow[r, "q"] \arrow[d, "p"]
\arrow[dr, phantom, "\scalebox{1.5}{$\lrcorner$}" , very near start, color=black]
& K \arrow[d, "w"] \\
\BTOPSpin\ar[r, "p\circ i"] &\BSTOP \arrow[r, "w_2"] & K(\Z/2,2),
\end{tikzcd}
\quad\quad
\begin{tikzcd}[column sep = scriptsize]
\BSpin\ar[r,"i"] \ar[d,"="]&B(X)^{\Diff}  \arrow[r, "q"] \arrow[d, "p"]
\arrow[dr, phantom, "\scalebox{1.5}{$\lrcorner$}" , very near start, color=black]
& K \arrow[d, "w"] \\
\BSpin\ar[r, "p\circ i"] &\BSO \arrow[r, "w_2"] & K(\Z/2,2).
\end{tikzcd}
\]
The CW complexes $B(X)^{\TOP}$, $B(X)^{\Diff}$, and maps $p,q$, are chosen to make the right-hand squares homotopy pullbacks. The maps $w_2$ are chosen to be fibrations, with fibre $\BTOPSpin$ and $\BSpin$ respectively. This implies each map $q$ is also a fibration. The respective fibres of the $w_2$ maps are, by definition, $\BTOPSpin$ and $\BSpin$ respectively. By a general property of pullback squares, the fibres of $q$ and $w_2$ are weakly homotopy equivalent. But as classifying spaces are CW complexes, this implies the fibres are moreover homotopy equivalent. Thus we may choose $\BTOPSpin$ and $\BSpin$ as a models for the homotopy fibre of $q$ in the respective categories. Finally, we may rechoose the spaces $B(X)^{\TOP}$ and $B(X)^{\Diff}$ so that in each diagram the map $i$ is the inclusion of a fibre of $q$, and so that the left-hand downward map is indeed an equality.

In order to prove we have written down the correct normal 2-types, and for later arguments in the paper, we will need to reference the following well-known computation.

\begin{lemma}\label{lem:wellknowngroups}
The groups $\pi_n(\STOP)$ are $\Z/2,0,\Z\oplus \Z/2,0$ for $n=1,2,3,4$ respectively.
\end{lemma}

\begin{proof}
We explain briefly how these groups are computed. First, there is a pullback square
\[\begin{tikzcd}
  \BSO \ar[r] \ar[d] & \BSTOP \ar[d] \\ \BO \ar[r] & \BTOP.
\end{tikzcd}\]
Because it is a pullback square, the respective horizontal homotopy fibres $\STOP/\SO$ and $\TOP/\Orth$ are (weakly) homotopy equivalent to one another.
Thus the homotopy groups of $\STOP/\SO$ satisfy $\pi_i(\STOP/\SO) \cong \pi_i(\TOP/\Orth) \cong \pi_i(K(\Z/2,3))$ for $i=1,2,3,4$, with the latter isomorphism coming from~\cite[Annex C, p.~317]{MR0645390}. It follows from the long exact sequence of the fibration $\STOP/\SO \to \BSO \to \BSTOP$ that  $\pi_i(\BSTOP) \cong \pi_i(\BSO)$ for $i=2,3,5$.  We also have a short exact sequence
\[
\underbrace{\pi_4(\STOP/\SO)}_{=\,0} \to \underbrace{\pi_4(\BSO)}_{\cong\, \Z} \to \pi_4(\BSTOP) \to \underbrace{\pi_3(\STOP/\SO)}_{\cong\, \Z/2} \to \underbrace{\pi_3(\BSO)}_{=\,0}.
\]
Siebenmann~\cite[Annex C, p.~318]{MR0645390} showed that this sequence splits, so $\pi_4(\BSTOP) \cong \Z \oplus \Z/2$.  The desired computation then follows from $\pi_n(\STOP) \cong \pi_{n+1}(\BSTOP)$.
\end{proof}

We confirm we have the correct normal $2$-types.

\begin{proposition}\label{prop:2-typecalc}
In either category, the pair $(B(X)^{\CAT},p)$ is a normal $2$-type for the $\CAT$ manifold $X$.
\end{proposition}

\begin{proof}
The homotopy groups $\pi_r(\BSO)$ and $\pi_r(\BSpin)$ agree for $r=1,2,3$, by Lemma~\ref{lem:wellknowngroups}. So the proof below can proceed in $\CAT$ notation.  Write $B:=B(X)^{\CAT}$ for brevity.
A homotopy pullback diagram determines a long exact sequence in homotopy groups
\[
\cdots\to\pi_{r+1}(K(\Z/2,2))\to \pi_r(B)\xrightarrow{(p_*,q_*)} \pi_r(\BSCAT)\times\pi_r(K)\to \pi_r(K(\Z/2,2)\to\cdots
\]
Using that $\pi_1(\BSCAT)$ is trivial, we deduce there is an exact sequence
\[
0\to \pi_2(B)\xrightarrow{(p_*,q_*)} \pi_2(\BSCAT)\times\pi_2(K)\xrightarrow{(w_2)_*-w_*} \pi_2(K(\Z/2,2))\to\pi_1(B)\to 0,
\]
and for $r\geq 3$ we see that $p_*\colon\pi_r(B)\to\pi_r(\BSCAT)$ is an isomorphism, showing $p$ is 3-coconnected.
The map $(w_2)_* \colon \pi_2(\BSCAT) \to \pi_2(K(\Z/2,2)) \cong \Z/2$ is an isomorphism, the inverse of which determines a splitting $\pi_2(K(\Z/2,2)) \to \pi_2(\BSCAT) \to \pi_2(\BSCAT)\times\pi_2(K)$ of $(w_2)_*-w_*$.
We deduce that $\pi_1(B)=0$ and that $q_*\colon \pi_2(B)\to \pi_2(K)$ is an isomorphism.

The existence of a lift $\overline{\nu}_X\colon X\to B$ of $\nu_X$ is implied by the universal property of the pullback applied to $(\nu_X\,\,k)\colon X\to \BSCAT\times K$. That one obtains homotopic maps $w_2 \circ \nu_X$ and $w \circ k$ from $X$ to $K(\Z/2,2)$ uses that $w_2(X) = w_2(\nu_X)$, which follows from $w_1(X) = w_1(\nu_X) =0$ and the Whitney sum formula.
To check this lift is $3$-connected, consider that $\pi_1(B)=0$, and $\pi_3(B)\cong\pi_3(\BSCAT)=0$, so it is automatically a $\pi_1$ isomorphism and a surjection on $\pi_3$. On $\pi_2$, consider that
\[
((\nu_X)_*\,\,k_*)=(p_*\,\,q_*)\circ(\overline{\nu}_X)_*\colon \pi_2(X)\to \pi_2(\BSCAT)\times\pi_2(K),
\]
so $k_* = q_* \circ \overline{\nu}_X \colon \pi_2(X) \to \pi_2(K)$.
But as both $q_*$ and $k_*$ are isomorphisms, so is $(\overline{\nu}_X)_*$. So $\overline{\nu}_X$ is 3-connected, as desired.
\end{proof}

Recall the notation
\[
\gamma\colon \BTOPSpin\to \BSTOP\qquad\text{ and } \qquad \gamma\colon \BSpin\to \BSO
\]
for principal fibrations corresponding to the universal covering maps $\TOPSpin\to\STOP$ and $\Spin\to\SO$. In the case that $X$ is spin, here is an alternative model for the normal $2$-type.

\begin{lemma}\label{lem:alternative2type} When $X$ is spin, the following are also models for the normal $2$-type of $X$:
\[
\begin{array}{lclcl}
\gamma\circ\pr_2\colon &\BTOPSpin\times K&\to& \BSTOP &\qquad\text{when $\CAT=\TOP$,}\\
\gamma\circ\pr_2\colon &\BSpin\times K&\to& \BSO &\qquad\text{when $\CAT=\Diff$.}
\end{array}
\]
\end{lemma}

\begin{proof}
Consider the homotopy pullback squares defining $B(X)^{\TOP}$ and $B(X)^{\Diff}$. As $X$ is spin, in each case the map $w\colon K\to K(\Z/2,2)$ is constant. The fact that the pullback is homotopy equivalent to product of $K$ and the homotopy fibre of $w_2$ now follows from a general statement. For a based space $A$, write $A^I$ for the free path space of $A$, that is the space of unbased maps $[0,1]\to A$, where $A^I$ is based at the constant map to the basepoint. For maps $f\colon U\to A$ and $g\colon V\to A$, a model for the homotopy pullback of $U\rightarrow A\leftarrow V$ is the double mapping path fibration $U\times_f A^I\times_g V$; see~\cite[Definition~2.2.1]{MR2884233}. When $g$ is constant, this is a product $(U\times_f A^I)\times V$. Moreover the mapping path space $U\times_f A^I$ is a model for the homotopy fibre of~$f$. The result follows.
%We work in the topological category, the smooth argument being once again entirely similar. The homotopy fibre of $\gamma\circ\pr_2$ is homotopy equivalent to $K(\Z/2,1)\times K$, which it is easy to see is $3$-coconnected. Now as $X$ is spin, we may choose a spin structure $\mathfrak{s} \colon X\to \BTOPSpin$ on $\nu_X$. We may also always choose a $3$-connected map $h\colon X\to K$. As $\BTOPSpin$ is $3$-connected, the map $\mathfrak{s}\times h$, which is clearly a lift of $\nu_X$ under $\gamma\circ\pr_2$, is also $3$-connected. This shows normal $2$-smoothings in $(\BTOPSpin\times K,\gamma\circ\pr_2)$ exist, and hence it is a normal $2$-type.
%%
\end{proof}

In the proof of Theorem~\ref{thm:main} later, we will apply Theorem~\ref{thm:kreckmain2}. This means we will need to know whether it is the case that all classes in $\pi_4(B(X)^{\CAT})$ are $w_4$-trivial (recall Definition~\ref{def:w4triv}). In fact, they are not.

\begin{proposition}\label{prop:w4}
For both $\CAT=\TOP$ and $\CAT=\Diff$, there exists a class in $\pi_4(B(X)^{\CAT})$ that is not $w_4$-trivial.
\end{proposition}

\begin{proof} By Proposition~\ref{prop:2-typecalc}, the map $p\colon B(X)^{\CAT}\to \BSCAT$ is 3-coconnected, so in particular $p_*\colon\pi_4(B(X)^{\CAT})\to\pi_4(\BSCAT)$ is an isomorphism. There is a unique $4$-plane bundle $\sigma$ over~$S^4$ with $p_1(\sigma)=2$ and $e(\sigma)=1$~\cite[\textsection 22.6, \textsection 22.7, \textsection 23.6]{MR0039258}. Hence $1=e(\sigma)\equiv w_4(\sigma)\pmod 2$. As Stiefel-Whitney classes are stable, this implies $\pi_4(B(X)^{\Diff})$ contains a class with nontrivial~$w_4$. The stable class of the underlying topological bundle of $\sigma$ still has $w_4(\sigma)$ nontrivial. Thus also in $\pi_4(B(X)^{\TOP})$ we find an element that is not $w_4$-trivial. (In fact, for $\CAT=\Diff$, this element generates $\pi_4(\BSO)\cong\Z$ and for $\CAT=\TOP$, generates the $\Z$-summand of $\pi_4(\BSTOP)\cong\Z\oplus\Z/2$, but we will not need this.)
\end{proof}

\section{Bordism computations and the James spectral sequence}\label{section:james-SS-bordism-groups}

The ultimate purpose of this section is to prove the following proposition.

\begin{proposition}\label{prop:bordismisboring}
For both $\CAT=\TOP$ and $\CAT=\Diff$, the $2$-types from Proposition~\ref{prop:2-typecalc} have $\Omega_5(B(X)^{\CAT},p)=0$.
\end{proposition}

In order to do this, we will use Teichner's \emph{James spectral sequence}~\cite[\textsection II]{MR1214960}. We will need to know some of the differentials in the spectral sequence, for computational purposes.
To find these, we will follow Teichner's proof of~\cite[Proposition 1]{MR1214960}, which also computes differentials in a James spectral sequence. We will require an extra differential ($r=6$ of Proposition~\ref{prop:teichnerdifferentials}), not covered by Teichner's statement. For this reason, in Section~\ref{subsec:bordismcalculation}, we investigate the generality in which the ideas of~\cite[Proposition 1]{MR1214960} work, filling in some details of the argument, and giving the most general statements we can.
We apply these general statements to our specific situation in Section~\ref{sec:bordism2typecomp}.

To ease notation, in this section only, we denote the functor $(-)^\vee:=\Hom(-;\Z/2)$.

\subsection{Basic properties of the James spectral sequence}\label{subsec:bordismcalculation}

Let $B$ be a CW complex and let~$B^{(k)}$ be its $k$-skeleton. Given a stable topological bundle $v\colon B\to \BSTOP$, after a homotopy we may assume that for each $k$, $v(B^{(k)}) \subseteq \BSTOP(n)$ for some $n$.   Define a CW filtration $B_0\subseteq B_1\subseteq \dots\subseteq B$, where $B_n\subseteq B$ denotes the maximal CW subcomplex contained in the preimage $v^{-1}(\BSTOP(n))$. The sequence $v_n:=v|_{B_n}\colon B_n\to \BSTOP(n)$ is then a canonical representative for the stable class $v\colon B\to \BSTOP$ (see~\cite[IV.5.12(c)]{MR1627486}\footnote{For the reader navigating the citation: Thom isomorphisms can be developed in the generality of spherical fibrations, and Rudyak's notation $\mathcal{F}$ is the fibre automorphism group corresponding to a stable spherical fibration. To apply the citation, we pass to the underlying spherical fibration of our stable topological bundle, as explained in~\cite[IV.5.12(e)]{MR1627486}, where his notation $\mathcal{V}$ can stand for either $\STOP$ or $\SO$.}). The Thom spectrum $M(v)$ has $n$th space the Thom space $\Th(v_n):=D(v_n)/S(v_n)$. The structure maps $\Sigma\Th(v_n) = \Th(v_n \oplus \varepsilon) \to \Th(v_{n+1})$ arise from the square
\[\begin{tikzcd}
  B_n \ar[r] \ar[d,"v_n"] & B_{n+1} \ar[d,"v_{n+1}"] \\ \BSTOP(n) \ar[r] & \BSTOP(n+1),
\end{tikzcd}\]
which implies that $v_{n} \oplus \varepsilon$ is the
%fact that $B_n \xrightarrow{v_n} \BSTOP(n)  \to \BSTOP(n+1)$ is
the pullback of $v_{n+1}$ under the inclusion $B_n \to B_{n+1}$.
As $v_n$ is orientable for all $n$ there exist Thom classes $U_n\in H^n(\Th(v_n);\Z)$. Writing $\rho_n\colon D(v_n)\to B_n$ for the projection of the associated disc bundle, we obtain unstable Thom isomorphisms
\[
\begin{array}{rccccc}
\Phi^n&=&U_n\cup(\rho_n^*(-))\colon&H^{*}(B_n;\Z)&\xrightarrow{\cong}& H^{*+n}(\Th(v_n);\Z),\\
\Phi_n&=&(\rho_n)_*(U_n\cap-)\colon&H_{*+n}(\Th(v_n);\Z)&\xrightarrow{\cong}& H_{*}(B_n;\Z).
\end{array}
\]
%The Thom spectrum $M(v_n)$ is homotopy equivalent to the (desuspended) suspension spectrum $\Sigma^{-n}\Sigma^\infty\Th(v_n)$, which implies $H^{n+*}(\Th(v_n);\Z)\cong H^*(M(v_n);\Z)$ and $H_{*+n}(\Th(v_n);\Z)\cong H_*(M(v_n);\Z)$.
In, respectively, the limit and colimit, we obtain the stable Thom isomorphisms (see~\cite[IV.5.23(ii)]{MR1627486}):
\[
\begin{array}{rccc}
\Phi^\infty\colon&H^{*}(B;\Z/2)&\xrightarrow{\cong}& H^{*}(M(v);\Z/2),\\
\Phi_\infty\colon&H_{*}(M(v);\Z)&\xrightarrow{\cong}& H_{*}(B;\Z),
\end{array}
\]
noting that in each displayed isomorphism, the side involving $B$ is the (co)homology of a \emph{space} and the side involving $M(v)$ is the (co)homology of a \emph{spectrum} (in particular, explaining the apparent grading shift from the unstable version).

\begin{remark}We are only stating the $\Z/2$-coefficient Thom isomorphism on cohomology, as this is the only one we will need in the sequel. There are Thom isomorphisms with other coefficients (see~\cite[IV.5.23(ii)]{MR1627486}), but in these settings it may not be possible to use a limit to state them. In general, for a ring $R$, there is a short exact sequence
\[
0\to {\lim}^1H^{*+n-1}(\Th(v_n);R)\to H^*(M(v);R)\to{\lim}^0 H^{*+n}(\Th(v_n);R)\to 0;
\]
cf.~\cite[Proposition~7.66]{MR0385836}. When $R$ is a field, the ${\lim}^1$ term vanishes, which permitted the limit definition of the stable Thom isomorphism on cohomology.
\end{remark}

For a topological space $X$, and $y\in H^2(X;\Z/2)$, write
\[
\Sq^2_y\colon H^*(X;\Z/2)\to H^{*+2}(X;\Z/2);\qquad x\mapsto \Sq^2(x)+y\cup x.
\]

\begin{lemma}\label{lem:sqv} Let $v\colon B\to \BSCAT$ be a stable bundle over a CW complex.
Then $\Sq^2\circ\Phi^\infty=\Phi^\infty\circ\Sq^2_{w_2(v)}$.
\end{lemma}

\begin{proof}For all $i\geq 0$, we have Thom's formula $\Sq^i(U_n)=U_n\cup(\rho_n^*w_i(v_n))$
%\cite[Chapitre~II]{MR0054960}
 (see e.g.~\cite[p.~91]{MR440554}, and note that the formula is effectively the \emph{definition} of the Stiefel-Whitney classes in the topological category; see e.g.~\cite[Definition~7.2]{TheGuide}). We combine this with the Cartan formula to calculate for any $x\in H^{r-2}(B_n;\Z/2)$ that
\begin{align*}
\Sq^2(U_n\cup(\rho_n^*(x)))&=\Sq^0(U_n)\cup\Sq^2(\rho_n^*(x))+ \Sq^1(U_n)\cup\Sq^1(\rho_n^*(x)) + \Sq^2(U_n)\cup\Sq^0(\rho_n^*(x))\\
&= U_n\cup\Sq^2(\rho_n^*(x))+ U_n\cup \rho_n^*(w_1(v_n))\cup\Sq^1(\rho_n^*(x)) + U_n\cup \rho_n^*(w_2(v_n))\cup\rho_n^*(x)\\
&=U_n\cup\rho_n^*\left((\Sq^2(x))+0+w_2(v_n)\cup x)\right)\\
&=U_n\cup\rho_n^*\left(\Sq^2_{w_2(v_n)}(x)\right),
\end{align*}
so that $\Sq^2\circ\Phi^n=\Phi^n\circ\Sq^2_{w_2(v_n)}$. For the penultimate inequality we used that $w_1(v_n)=0$. Taking the limit of this equality we obtain the desired formula. Here we have used that Steenrod squares are stable and natural, that $\Phi^\infty$ is defined as the limit of $\Phi^n$, and that Stiefel-Whitney class are stable.
\end{proof}

The James spectral sequence is a first-quadrant spectral sequence which takes as input a connective generalised homology theory~$\mathbb{H}$, together with
\[
\begin{tikzcd}
F\ar[r,"i"]&B\arrow[r, "f"] & K & \text{and} & v\colon B\ar[r]&\BSTOP,
\end{tikzcd}
\]
where the left diagram is a fibration of CW complexes. (For the smooth version, replace~$v$ with a stable vector bundle $B\to \BSO$.) It is furthermore required that the fibration be \emph{$\mathbb{H}$-orientable} in the sense that the monodromy-induced action of $\pi_1(K)$ on $\mathbb{H}(F)$ is trivial. For us, the connective homology theory $\mathbb{H}$ will always be given by stable homotopy groups $\mathbb{H}_*(-)=\pi^{\operatorname{st}}_*(-)$; that is, we use the generalised homology theory corresponding to the sphere spectrum $\mathbb{S}$. As the homotopy groups of spectra are by definition stable homotopy groups, we will omit the ``$st$'' superscript whenever we take homotopy groups of any spectrum.  We refer the reader to~\cite[Theorem]{MR1214960} for the construction of the James spectral sequence, but we record that it has the property
\[
E_2^{r,s}\cong H_r(K;\pi_s(M(v \circ i)))\implies \pi_{r+s}(M(v)),
\]
and that the spectral sequence is natural with respect to maps of fibrations, that are commutative diagrams.

The identity map $\Id_B\colon B\to B$, considered as a fibration with point fibres, is obviously $\pi_*^{\operatorname{st}}$-orientable because the monodromy is trivial. It is thus meaningful to consider the following James spectral sequence.

\begin{lemma}\label{lem:differentials} Let $B$ be a CW complex and consider the input for the James spectral sequence
\[
\begin{tikzcd}
\pt\ar[r,"j"]&B\arrow[r, "\Id"] & B & \text{and} & v\colon B\ar[r]&\BSCAT.
\end{tikzcd}
\]
The James spectral sequence has the following differentials $d_2^{r,s}\colon E_2^{r,s}\to E_2^{r-2,s+1}$ for all $r\geq 0$.
\begin{enumerate}[leftmargin=*]\setlength\itemsep{0em}
\item\label{item:one} Via the universal coefficient theorem, $d^{r,1}_2\colon H_r(B;\Z/2)\to H_{r-2}(B;\Z/2)$ is identified with $(\Sq^2_{w_2(v)})^\vee$. More precisely, the following square commutes:
\[\begin{tikzcd}
    H_r(B;\Z/2) \ar[rr,"d_2^{r,1}"] \ar[d,"\cong"]
    && H_{r-2}(B;\Z/2) \ar[d,"\cong"] \\
H^r(B;\Z/2)^\vee \ar[rr,"(\Sq^2_{w_2(v)})^\vee"]
   && H^{r-2}(B;\Z/2)^\vee.
  \end{tikzcd}\]
\item\label{item:two} The map $d_2^{r,0}\colon H_r(B;\Z)\to H_{r-2}(B;\Z/2)$ is identified with reduction mod 2 followed by $(\Sq^2_{w_2(v)})^\vee$. More precisely, the following diagram commutes:
  \[\begin{tikzcd}
    H_r(B;\Z) \ar[rrrr,"d_2^{r,0}"] \ar[d,"\operatorname{red}_2"] &&&& H_{r-2}(B;\Z/2) \ar[d,"\cong"] \\
   H_r(B;\Z/2)\ar[rr,"\cong"] &&H^r(B;\Z/2)^\vee \ar[rr, "(\Sq^2_{w_2(v)})^\vee"] && H^{r-2}(B;\Z/2)^\vee.
  \end{tikzcd}\]
\end{enumerate}
\end{lemma}

\begin{proof}
%In this proof we write the functor $\Hom(-,\Z/2)=:(-)^{{\ast}}$, for brevity.

Consider that $M(v\circ j)$ is homotopy equivalent to the sphere spectrum and hence we get $\pi_*(M(v\circ j))\cong\pi_*^{\operatorname{st}}$, the stable homotopy groups of spheres. In particular, this explains the coefficients $\pi_0^{\operatorname{st}}\cong\Z$ and $\pi_1^{\operatorname{st}}, \pi_2^{\operatorname{st}}\cong\Z/2$ in items (\ref{item:one}) and (\ref{item:two}).

By the construction of the James spectral sequence (see~\cite[Theorem]{MR1214960}), the homology Thom isomorphisms $\Phi_\infty$ induce a natural isomorphism between the Atiyah-Hirzebruch spectral sequence of the Thom spectrum
\[
E_2^{r,s}\cong H_r(M(v);\pi_s^{\operatorname{st}})\implies \pi_{r+s}(M(v))
\]
and the James spectral sequence from the statement of the lemma.
In the former, it is well-known that $d^{r,1}_2=(\Sq^2)^\vee$, and $d_2^{r,0}=(\Sq^2)^\vee\circ \operatorname{red}_2$, where $\operatorname{red}_2$ denotes reduction mod 2; see e.g.~\cite[proof of Lemma]{MR1214960}. We must compute what these correspond to under $\Phi_\infty$.

We claim the following diagram is commutative.
\begin{equation}\label{eq:largediagram}
\begin{tikzcd}
H_{r}(M(v);\Z)\ar[d,"\operatorname{red}_2"]\ar[rrrrr, "\Phi_\infty", "\cong"'] &&&&& H_r(B;\Z)\ar[d, "\operatorname{red}_2"]\\
H_{r}(M(v);\Z/2)\ar[d,"\ev", "\cong"']\ar[rrrrr, "\Phi_\infty", "\cong"'] &&&&& H_r(B;\Z/2)\ar[d, "\ev", "\cong"']\\
H^{r}(M(v);\Z/2)^{{\vee}}\ar[d,"{(\Sq^2)}^\vee"]\ar[rrrrr, "{(\Phi^\infty)}^\vee", "\cong"'] &&&&& H^r(B;\Z/2)^{{\vee}}\ar[d,"{(\Sq^2_{w_2(v)})}^\vee"]\\
H^{r-2}(M(v);\Z/2)^{{\vee}}\ar[rrrrr, "{(\Phi^\infty)}^\vee", "\cong"'] &&&&& H^{r-2}(B;\Z/2)^{{\vee}}\\
H_{r-2}(M(v);\Z/2)\ar[u,"\ev"', "\cong"]\ar[rrrrr, "\Phi_\infty", "\cong"'] &&&&& H_{r-2}(B;\Z/2)\ar[u, "\ev"', "\cong"]
\end{tikzcd}
\end{equation}
Commutativity of diagram (\ref{eq:largediagram}) is enough to complete the proof, as $d^{r,0}_2$ and $d_2^{r,1}$ in the Atiyah-Hirzebruch spectral sequence are respectively the composition along the full left column and the lower four entries of the left column, and these map in the right column to the respective differentials claimed in items (\ref{item:one}) and (\ref{item:two}).

The top square of diagram (\ref{eq:largediagram}) is evidently commutative. Commutativity in the third square from the top is furnished by Lemma~\ref{lem:sqv}. To see commutativity of the remaining two squares, fix $n\geq 0$ and consider the diagram
\begin{equation}\label{eq:smalldiagram}
\begin{tikzcd}
H_{r+n}(\Th(v_n);\Z/2)\ar[d,"\ev", "\cong"']\ar[rr, "{U_n\cap-}"]
&&H_r(D(v_n);\Z/2)\ar[rr, "{(\rho_n)}_*"] \ar[d, "\ev", "\cong"']
&&H_r(B_n;\Z/2)\ar[d, "\ev", "\cong"']
\\
H^{r+n}(\Th(v_n);\Z/2)^{{\vee}}\ar[rr, "{(U_n\cup-)}^\vee"]
&& H^r(D(v_n);\Z/2)^{{\vee}}\ar[rr, "{({(\rho_n)}^*)}^\vee"]
&& H^r(B_n;\Z/2)^{{\vee}}.
\end{tikzcd}
\end{equation}
The right-hand square commutes by naturality of the Universal Coefficient Theorem. To check commutativity of the left-hand square we compute
\[
(\ev(U_n\cap x))(\phi)=\phi(U_n\cap x)=(U_n\cup \phi)(x)=\ev(x)(U_n\cup \phi)=(U_n\cup-)^\vee \ev(x)(\phi)
\]
for any $\phi\in H^r(D(v_n);\Z/2)$ and $x\in H_{r+n}(\Th(v_n);\Z/2)$. The second equality is the ``cup/cap'' formula~\cite[p.~249]{Hatcher}. Thus the diagram (\ref{eq:smalldiagram}) commutes for all $n$. Take the colimit of the outside rectangle of diagram (\ref{eq:smalldiagram}) to obtain commutativity for the remaining two squares of diagram (\ref{eq:largediagram}).
\end{proof}

\begin{corollary}\label{cor:themeat}Suppose we have a $\pi_*^{\operatorname{st}}$-orientable fibration and stable  bundle
\[
\begin{tikzcd}
F\ar[r,"i"]&B\arrow[r, "f"] & K & \text{and} & v\colon B\ar[r]&\BSCAT,
\end{tikzcd}
\]
such that $\pt\to F$ induces isomorphisms $\pi_s^{\operatorname{st}}\xrightarrow{\cong} \pi_s(M(v \circ i))$ for $s=0,1,2$. Suppose there exists $w\in H^2(K;\Z/2)$ such that $f^*(w)=w_2(v)$.
\begin{enumerate}[leftmargin=*]\setlength\itemsep{0em}
\item\label{item:themeat1} Suppose for some $r\in\Z$ that $f_*\colon H_r(B;\Z/2)\to H_r(K;\Z/2)$ is surjective. Then, in the sense of Lemma~\ref{lem:differentials}, $d^{r,1}_2\colon H_r(K;\Z/2)\to H_{r-2}(K;\Z/2)$ is $(\Sq^2_{w})^\vee$.
\item\label{item:themeat2} Suppose for some $r\in\Z$ that $f_*\colon H_r(B;\Z)\to H_r(K;\Z)$ is surjective. Then, in the sense of Lemma~\ref{lem:differentials}, $d_2^{r,0}\colon H_r(K;\Z)\to H_{r-2}(K;\Z/2)$ is reduction mod 2 followed by $(\Sq^2_{w})^\vee$.
\end{enumerate}
\end{corollary}

\begin{proof} Consider the map of fibrations given by the commuting diagram
\[
\begin{tikzcd}
\pt\ar[r,"j"]\ar[d, "k"]&B\arrow[r, "\Id"]\ar[d, "\Id"] & B\ar[d, "f"] \\
F\ar[r,"i"]&B\arrow[r, "f"] & K.
\end{tikzcd}
\]
where $k$ is arbitrary and $j$ is defined to make the diagram commute. By the construction of the James spectral sequence (see~\cite[Theorem]{MR1214960}), this induces a map of spectral sequences which on the $E_1$-page is of the form
\[
f_*\otimes k_*\colon C_r(B)\otimes \pi^{\operatorname{st}}_s\to C_r(K)\otimes \pi_s(M(v\circ i)).
\]
Taking homology, we see for $s=0,1,2$ and for all $r\geq 0$ the map on the $E_2$ page factors as
\[
\begin{tikzcd}
H_r(B;\pi_s^{\operatorname{st}})\ar[r, "f_*"] &H_r(K;\pi_s^{\operatorname{st}})\ar[r, "k_*", "\cong"'] &H_r(K;\pi_s(M(v\circ i))).
\end{tikzcd}
\]
When $s=1,2$, we may identify the coefficients in the above sequence with $\Z/2$. When $s=0$, we choose an isomorphism $\pi_0^{\operatorname{st}}\cong\Z$ and use $k_*$ to identify $\pi_0(M(v\circ i))\cong\Z$. With these identifications, we may identify the map of spectral sequences on the $E_2$ page as $f_*\colon H_r(B;\pi_s^{\operatorname{st}})\to H_r(K;\pi_s^{\operatorname{st}})$, when $s=0,1,2$.

We consider item (\ref{item:themeat1}). Let $x\in H_r(K;\Z/2)$. Then by hypothesis there exists $y\in H_r(B;\Z/2)$ such that $f_*(y)=x$. We thus have
\begin{align*}
d^{r,1}_2(x)=d^{r,1}_2(f_*(y))=f_*\circ d^{r,1}_2(y)=f_*\circ(\Sq^2_{w_2(v)})^\vee(y)\in H_{r-2}(K;\Z/2),
\end{align*}
where the final equality comes from Lemma~\ref{lem:differentials}. Applying the evaluation isomorphism $H_{r-2}(K;\Z/2)\cong H^{r-2}(K;\Z/2)^\vee$, and recalling $f_*=(f^*)^\vee$, we have shown that
$d^{r,1}_2(x) = \left((f^*)^\vee\circ \Sq^2_{w_2(v)}\right)(y)$.  Next we show that the right hand side equals $(\Sq^2_w(x))^\vee$. To do this we evaluate on an arbitrary class $u\in H^{r-2}(K;\Z/2)$:
\begin{align*}
\left((f^*)^\vee\circ \Sq^2_{w_2(v)}(u)\right)(y)&= \left(\Sq^2_{w_2(v)}(f^*u)\right)(y)\\
&= \Sq^2(f^*u)(y)+ \left(w_2(v)\cup (f^*u)\right)(y)\\
&= \Sq^2(u)(f_*y)+ \left((f^*w)\cup (f^*u)\right)(y)\\
&= \Sq^2(u)(f_*y)+ \left(w\cup u\right)(f_*y)\\
&= \Sq^2(u)(x)+ \left(w\cup u\right)(x).
\end{align*}
This shows $d^{r,1}_2(x)= \left((f^*)^\vee\circ \Sq^2_{w_2(v)}\right)(y) = (\Sq^2_w(x))^\vee$, as claimed.

A similar calculation to prove item (\ref{item:themeat2}) is left to the reader. The case when $v$ is a stable vector bundle is proved entirely similarly.
\end{proof}

\subsection{Bordism over the normal 2-type of a simply connected 4-manifold}\label{sec:bordism2typecomp}

From now on, let $(B,p)$ denote either $(B(X)^{\TOP},p)$ or $(B(X)^{\Diff},p)$, unless one is specified.  Refer to the start of Section~\ref{subsec:2type} to recall the notation. We will be interested in the following inputs for the James spectral sequence. When $\CAT=\TOP$, we will use
\[
\begin{tikzcd}
\BTOPSpin\ar[r,"i"]&B\arrow[r, "q"] & K & \text{and} & p\colon B\ar[r]&\BSTOPdot
\end{tikzcd}
\]
When $\CAT=\Diff$, we use as input
\[
\begin{tikzcd}
\BSpin\ar[r,"i"]&B\arrow[r, "q"] & K & \text{and} & p\colon B\ar[r]&\BSOdot
\end{tikzcd}
\]
Note that as $K$ is simply connected, these fibrations are both $\pi_*^{\operatorname{st}}$-orientable, so the James spectral sequence applies. We also not note that the maps $p\circ i$ classify the universal bundle over $\BTOPSpin $ and $\BSpin$ respectively (cf.~beginning of Section~\ref{subsec:2type}). With these inputs, we have $E_2$ page:
\[
E_2^{r,s}\cong H_r(K;\pi_s(M(p\circ i)))
\cong \left\{\begin{array}{ll}H_r(K;\Omega_s^{\TOPSpin})&\text{when $\CAT=\TOP$,}\\[3pt]
H_r(K;\Omega_s^{\Spin})&\text{when $\CAT=\Diff$.}
\end{array}\right.
\]
The second isomorphism follows from the Pontryagin-Thom theorem. Again by the Pontryagin-Thom theorem, the spectral sequence converges to $\pi_{r+s}(M(p))\cong \Omega_{r+s}(B,p)$.

We will need to reference the following well-known computation.

\begin{lemma}\label{lem:wellknowngroups2}
For $n=1,2,3,4,5$, the Hurewicz map is an isomorphism $\pi_n(\BTOPSpin)\cong H_n(\BTOPSpin)$, and these groups are $0, 0, 0, \Z\oplus\Z/2, 0$, respectively.

For $n=1,2,3,4,5$, the Hurewicz map is an isomorphism $\pi_n(\BSpin)\cong H_n(\BSpin)$, and these groups are $0, 0, 0, \Z, 0$, respectively.

\end{lemma}

\begin{proof}
Since $\TOPSpin$ is the universal cover of $\STOP$, and using Lemma~\ref{lem:wellknowngroups}, it follows that the groups $\pi_n(\TOPSpin)$ are $0,0,\Z\oplus \Z/2,0$ for $n=1,2,3,4$. This implies that the groups $\pi_n(\BTOPSpin)$ are $0,0,0,\Z\oplus \Z/2,0$ for $n=1,2,3,4,5$. The Hurewicz Theorem shows that the map $\pi_n(\BTOPSpin)\to H_n(\BTOPSpin;\Z)$, is an isomorphism for $n=1,2,3,4$ and a surjection for $n=5$. But as $\pi_5(\BTOPSpin)=0$, this implies $H_5(\BTOPSpin;\Z)=0$.

The result for $\BSpin$ follows from the identical argument.
\end{proof}

\begin{proposition}\label{prop:teichnerdifferentials}Recall the definition of $w\colon K\to K(\Z/2,2)$ from the construction of $B$ in Subsection~\ref{subsec:2type}, and regard it as a class $w\in H^2(K;\Z/2)$. The following assertions hold.
\begin{enumerate}[leftmargin=*]\setlength\itemsep{0em}
\item\label{item:teichnerdifferentials1} For $r$ odd or $r\leq 4$, the differential $d_2\colon H_r(K;\Omega_1^{\BCATSpin})\to H_{r-2}(K;\Omega_2^{\BCATSpin})$ is~$(\Sq^2_w)^\vee$.
\item\label{item:teichnerdifferentials2} For $r$ odd or $r\leq 6$, the differential $d_2\colon H_r(K;\Omega_0^{\BCATSpin})\to H_{r-2}(K;\Omega_1^{\BCATSpin})$ is reduction modulo 2, followed by $(\Sq^2_w)^\vee$.
\end{enumerate}
\end{proposition}

\begin{proof}
%By the definition of $B$, there is guaranteed to be a class $w\in H^2(K;\Z/2)$ such that $q^*w=w_2(p)$.
%\footnote{PO: Used to be a sentence here, but I realised the statement of the proposition didn't make sense without definition of $w$, so I moved it into the statement.}
Observe that $q^*w=w_2(p)$, by definition, and that $\pt\to \BCATSpin$ induces isomorphisms $\pi_s^{\operatorname{st}}\cong \pi_s(M(\gamma))\cong\Omega^{\CATSpin}_s$ for $s=0,1,2$. Thus we see that the initial hypotheses of Corollary~\ref{cor:themeat} are satisfied, so we are interested in analysing for which $r$ the maps $q_*\colon H_r(B;\Z/2)\to H_r(K;\Z/2)$ and $q_*\colon H_r(B;\Z)\to H_r(K;\Z)$ are surjective.

A model for $K$ is given by $\prod^n \CP^\infty$, where $n = b_2(X)$ is the second Betti number. This model has a cell structure with only even-dimensional cells. In particular $H_r(K;\Z)=0=H_r(K;\Z/2)$ for $r$ odd, and the maps in question are then necessarily surjective.

Using the groups computed in Lemma~\ref{lem:wellknowngroups2} and the long exact sequence of the fibration $q\colon B\to K$, we see that $q$ is $4$-connected. The Hurewicz theorem now implies that $q_*\colon H_r(B;\Z)\to H_r(K;\Z)$ is surjective for $r=0,1,2,3,4$.

We show surjectivity when $r=6$ by analysing the Leray-Serre spectral sequence for $q\colon B\to K$, which recall has differentials $d_n^{r,s}\colon E_n^{r,s}\to E_n^{r-n, s+n-1}$, and has $E_2^{r,s}\cong H_r(K;H_s(F))\implies H_{r+s}(B)$. The homotopy fibre of $q$ is $F\simeq\BCATSpin$. Looking at the $r+s=6$ line, for the Leray-Serre spectral sequence and noting that $E_2^{1,4}\cong H_1(K;H_4(\BCATSpin))=0$, we see that $E_\infty^{6,0}= E_2^{6,0}\cong H_6(K;\Z)$, from which we deduce that $q_*\colon H_6(B;\Z)\to H_6(K;\Z)$ is surjective.
\end{proof}

\begin{remark}
 We compute that $H_5(\BSpin;\Z/2)=0$, using the Universal Coefficient Theorem and Lemma~\ref{lem:wellknowngroups2}. So for $\CAT=\Diff$, the same argument as above proves the statement of Proposition~\ref{prop:teichnerdifferentials}~\eqref{item:teichnerdifferentials1} holds for $r=6$. But we compute that $H_5(\BTOPSpin;\Z/2)=\Z/2 \neq 0$, so for $\CAT=\TOP$ we cannot use the same argument to conclude that $H_6(B;\Z/2) \to H_6(K;\Z/2)$ is surjective. We do not know if Proposition~\ref{prop:teichnerdifferentials}~\eqref{item:teichnerdifferentials1} holds for $r=6$ when $\CAT=\TOP$.
\end{remark}

We now apply the James spectral sequence to prove Proposition~\ref{prop:bordismisboring}, which stated that for the $2$-types from Proposition~\ref{prop:2-typecalc}, we have $\Omega_5(B(X)^{\CAT},p)=0$. This computation will be required later.
%We remark that the corresponding $4$-dimensional bordism group will not be used later, but it is little extra effort to compute this group from what we do below. However, the $4$-dimensional bordism group is already a special case of~\cite[Lemma~5.3]{MR2541758} and so we do not include it here.

\begin{proof}[Proof of Proposition~\ref{prop:bordismisboring}]
%We work initially in the topological category.
%In the smooth category, the proof is entirely similar, noting only that the use of $\Omega_\ast^{\TOPSpin}$ below is replaced with $\Omega_\ast^{\Spin}$, and that this latter group also vanishes for $\ast=3, 5$~\cite{MR157388}.
%% then indicate which parts of the proof change in the smooth category.

A model for $K(\Z^n,2)$ is given by $\prod^n_{i=1} \CP^\infty$, where here $n = b_2(X)$ is the second Betti number. Here are the relevant terms on the $E_2$ page of the James spectral sequence. Note the 0 entries in the $q=3, 5$ lines come from $\Omega_5^{\TOPSpin}=0$~\cite[Theorem 13.1]{MR0645390}, and $\Omega_3^{\TOPSpin}=\Omega_3^{\Spin}=0$, and~$\Omega_5^{\Spin}=0$~\cite{MR157388}.

\begin{center}
\begin{tikzpicture}[scale=0.935]

\draw[step=2.0,black,xshift=1cm,yshift=1cm] (0,0) grid (14.5, 6.5);

\draw (1,2) -- (15.5,2);
\draw (1,4) -- (15.5,4);
\draw (1,6) -- (15.5,6);

\node at (0.5,1.5) {$0$};
\node at (0.5,2.5) {$1$};
\node at (0.5,3.5) {$2$};
\node at (0.5,4.5) {$3$};
\node at (0.5,5.5) {$4$};
\node at (0.5,6.5) {$5$};
\node at (0.5,7.5) {$q$};

\node at (2,0.5) {$0$};
\node at (4,0.5) {$1$};
\node at (6,0.5) {$2$};
\node at (8,0.5) {$3$};
\node at (10,0.5) {$4$};
\node at (12,0.5) {$5$};
\node at (14,0.5) {$6$};
\node at (15.5,0.5) {$p$};

\node at (4,1.5) {$0$};
\node at (8,1.5) {$0$};
%\node at (10,1.5) {$H_4(K;\Z)$};
\node at (12,1.5) {$0$};
\node at (14,1.5) {$H_6(K;\Z)$};

\node at (4,2.5) {$0$};
%\node at (6,2.5) {$H_2(K;\Z/2)$};
\node at (8,2.5) {$0$};
\node at (10,2.5) {$H_4(K;\Z/2)$};
\node at (12,2.5) {$0$};

\node at (4,3.5) {$0$};
\node at (6,3.5) {$H_2(K;\Z/2)$};
\node at (8,3.5) {$0$};
\node at (12,3.5) {$0$};

\node at (2,4.5) {$0$};
\node at (4,4.5) {$0$};
\node at (6,4.5) {$0$};
\node at (8,4.5) {$0$};
\node at (10,4.5) {$0$};
\node at (12,4.5) {$0$};
\node at (14,4.5) {$0$};

%\node at (2,5.5) {$\Omega_4^{\TOPSpin}$};
\node at (4,5.5) {$0$};
\node at (8,5.5) {$0$};
\node at (12,5.5) {$0$};

\node at (2,6.5) {$0$};
\node at (4,6.5) {$0$};
\node at (6,6.5) {$0$};
\node at (8,6.5) {$0$};
\node at (10,6.5) {$0$};
\node at (12,6.5) {$0$};
\node at (14,6.5) {$0$};

% arrows
\draw [thick, -latex](13.12,1.72) -- (10.88,2.28);
\draw [thick, -latex](9.12,2.72) -- (6.88,3.28);
%\draw [thick, -latex](9.12,1.72) -- (6.88,2.28);
\node at (12.5,2.28) {$\boldsymbol{d_2^{\,6,0}}$};
\node at (8.5,3.28) {$\boldsymbol{d_2^{\,4,1}}$};
%\node at (8.5,2.28) {$\boldsymbol{d_2^{\,4,0}}$};

\end{tikzpicture}
\end{center}

In the case that $X$ is spin, choose any basis $\langle y_1,\dots,y_n\rangle$ of $\pi_2(X)$. In the case that $X$ is not spin, we choose a basis more carefully, as follows. Let $y_1\in \pi_2(X)$ be such that $w_2(X)(y_1)=1$. Extend to a basis $\langle y_1,\dots, y_n\rangle$ of $\pi_2(X)$. If $w_2(X)(y_i)\neq 0$ for some $i=2,\dots, n$, we replace~$y_i$ with $y_1+y_i$, so that without loss of generality it may be assumed $w_2(X)(y_i)=0$ for all $i=2,\dots,n$.

Recall that in the definition of $(B,p)$ there is a reference map $k\colon X\to K$ that is an isomorphism on $\pi_2$. Choose a basis for $H^2(K;\Z/2)$ by taking $x_i$ such that $k^*(x_i)=y_i$, for each $i=1, \dots,n$.
%Thus in both the spin and non spin case, $\langle y_1,\dots,y_n\rangle$ determines a basis $\langle k_*(y_1),\dots,k_*(y_n)\rangle$ of $\pi_2(K)$. Write $e^0_i\cup e^2_i \cup e^4_i\cup\cdots$ for the standard cell structure of the $i^{\text{th}}$ $\CP^\infty$ factor of $\prod_{i=1}^n\CP^\infty$.
There is a homotopy equivalence $K\simeq \prod_{i=1}^n\CP^\infty$, and from this we deduce that as a graded algebra we have $H^*(K;\Z/2)=(\Z/2)[x_1,\dots,x_n]$. Thus, we have bases
\begin{align*}
H^2(K;\Z/2)&\cong\langle x_i\,|\,1\leq i\leq n\rangle,\\
H^4(K;\Z/2)&\cong\langle x_i\cup x_j\,|\,1\leq i\leq j\leq n\rangle,\\
H^6(K;\Z/2)&\cong\langle x_i\cup x_j\cup x_k\,|\,1\leq i\leq j\leq k\leq n\rangle.
\end{align*}
We now split the calculations into the case $X$ is spin and the case $X$ is not spin.

Firstly, assume $X$ is spin, so that $w=0$. By Proposition~\ref{prop:teichnerdifferentials}, the differentials $d_2^{4,1}$ is the $\Hom$-dual to the map $x_i\mapsto \Sq^2(x_i)=x_i\cup x_i$. We see that
\[
d_2^{4,1}([x_i\cup x_j]^\vee)(x_k)=[x_i\cup x_j]^\vee(x_k\cup x_k)=\left\{\begin{array}{ll} 1&\text{if $i=j=k$}\\0&\text{otherwise.}\end{array}\right.
\]
In particular, this shows $\ker(d_2^{4,1})$ has basis $\langle [x_i\cup x_j]^\vee\,|\, 1\leq i< j\leq n\rangle$. By Proposition~\ref{prop:teichnerdifferentials}, the differential $d_2^{6,0}$ is reduction modulo 2 followed by the dual to the map $\Sq^2$. The Cartan formula shows
\begin{align*}
\Sq^2(x_i\cup x_j)&=\Sq^0(x_i)\cup\Sq^2(x_j)+ \Sq^1(x_i)\cup\Sq^1(x_j)+ \Sq^2(x_i)\cup\Sq^1(x_j)\\
&=x_i\cup x_j\cup x_j+x_i\cup x_i\cup x_j.
\end{align*}
Using this, we have
\begin{align*}
(\Sq^2)^\vee([x_i\cup x_j\cup x_k]^\vee)(x_r\cup x_s)&=[x_i\cup x_j\cup x_k]^\vee(x_r\cup x_s\cup x_s+x_r\cup x_r\cup x_s).
\end{align*}
By inspection, for $1\leq i<j\leq n$ and $1\leq r\leq s\leq n$, we see that $(\Sq^2)^\vee([x_i\cup x_j\cup x_j]^\vee)$, projects nontrivially to the span of $[x_r\cup x_s]^\vee$ if and only if $(i,j)=(r,s)$. Thus $(\Sq^2)^\vee([x_i\cup x_j\cup x_j]^\vee)=[x_i\cup x_j]^\vee$. From this we can conclude that $\im(d_2^{6,0})=\ker(d_2^{4,1})$, so that $E_3^{4,1}=0$. As $\Omega_5^{\CATSpin}=0$, and hence the $r+s=5$ line dies already on the~$E_3$ page and we have shown that $\Omega_5(B(X)^{\CAT},p)=0$ in the spin case.

Now assume $X$ is not spin, so that $w\neq 0$. Our construction of a basis was such that $x_1=w$. By Proposition~\ref{prop:teichnerdifferentials}, the differential $d_2^{4,1}$ is dual to the map $x_i\mapsto \Sq^2_w(x_i)=\Sq^2(x_i)+x_i\cup x_1=x_i\cup x_i+x_1\cup x_i$.
From this we see that for $1\leq i\leq j\leq n$ and $1\leq k\leq n$
\[
d_2^{4,1}([x_i\cup x_j]^\vee)(x_k)=[x_i\cup x_j]^\vee(x_k\cup x_k+x_1\cup x_k)=
\left\{\begin{array}{ll} 1&\text{if $i=j=k\neq 1$,}\\
1&\text{if $i=1$ and $j,k\neq 1$,}\\
0&\text{otherwise.}\end{array}\right.
\]
This suggests a basis change for $H^2(K;\Z/2)$, namely that we replace $x_1\cup x_i$ with $x_i\cup x_i+x_1\cup x_i$ for all $1<i\leq n$.
This basis splits off the kernel $\ker(d_2^{4,1})$ as a direct summand of $H^4(K;\Z/2)$, with basis $\langle [x_1\cup x_1]^\vee, [x_k\cup x_k]^\vee+[x_1\cup x_k]^\vee, [x_i\cup x_j]^\vee\,|\, 1<i<j\leq n, 1<k\leq n\rangle$. By Proposition~\ref{prop:teichnerdifferentials}, the map~$d_2^{6,0}$ is reduction modulo 2 followed by the dual to $\Sq^2_w$. As $H_6(K;\Z)$ is free, reduction modulo 2 is surjective to $H_6(K;\Z/2)$. We now show that the dual of $\Sq_w^2$ is surjective to $\ker(d_2^{4,1})$. Using the Cartan formula computation from above, we have
\begin{align*}
(\Sq_w^2)^\vee([x_i\cup x_j\cup x_k]^\vee)(x_r\cup x_s)&=[x_i\cup x_j\cup x_k]^\vee(\Sq^2(x_r\cup x_s)+(x_r\cup x_s)\cup w)\\
&=[x_i\cup x_j\cup x_k]^\vee(x_r\cup x_s\cup x_s+x_r\cup x_r\cup x_s+x_r\cup x_s\cup x_1).
\end{align*}
Observe, for $1\leq r\leq s\leq n$, that  $(\Sq_w^2)^\vee([x_1\cup x_1\cup x_1]^\vee)$ projects non-trivially to $[x_r\cup x_s]^\vee$ exactly when $r=s=1$, so that $(\Sq_w^2)^\vee([x_1\cup x_1\cup x_1]^\vee)=[x_1\cup x_1]^\vee$. By a similar inspection, when $1<i<j\leq n$, we observe that $(\Sq_w^2)^\vee([x_i\cup x_j\cup x_1]^\vee)=[x_i\cup x_j]^\vee$. By a similar inspection, for $1<k\leq n$, we have that $(\Sq_w^2)^\vee([x_1\cup x_k\cup x_k]^\vee)=[x_k\cup x_k]^\vee+[x_1\cup x_k]^\vee$. We can conclude that $\im(d_2^{6,0})=\ker(d_2^{4,1})$, so that $E_3^{4,1}=0$. From this and the fact that $\Omega_5^{\CATSpin}=0$, the $r+s=5$ line dies already on the $E_3$ page and we have shown that $\Omega_5(B(X)^{\CAT},p)=0$ in the non spin case.
\end{proof}

\section{Proof of the main pseudo-isotopy classification}\label{sec:proofmain}

We establish an elementary technical result which we then use to extend a $B$-structure from $\partial(X\times [0,1])$ to $X\times [0,1]$. With this extension in hand, we then prove Theorem~\ref{thm:main}, the main pseudo-isotopy classification.

\subsection{Extending a $B$-structure from the boundary}\label{section:extending-B-structure}

Let $X$ be a compact 4-manifold with (possibly empty) boundary. Write $I:=[0,1]$, $X_i:=X\times\{i\}$ for $i=0,1$, and \[D:=X_0\cup(\partial X\times I)\cup X_1,\] i.e.\ $D=\partial(X \times I)$. Choose a finite cell complex homotopy equivalent to $X$, relative to a finite cell complex homotopy equivalent to $\partial X$. For each $r\in\Z$, write the $r$-cells of $X$ relative to $\partial X$ as $e^r_{\alpha}$, for $1\leq \alpha\leq N_r$ some choice of ordering. Write $I=e^0_0\cup e^1\cup e^0_1$ for the standard cell structure of $I$. Take the corresponding product cell complex for $X\times I$. Let $C_r(-)$ denote the $r$th  cellular chain group. Fix an abelian group $G$ and write $C^r(-):=\Hom(C_r(-),G)$. Our choice of cell complex for $X\times I$ means
\[
\begin{array}{lcl}
C_r(X\times I,D)&=&\Z\langle e^{r-1}_\alpha\times e^1\mid 1\leq \alpha\leq N_{r-1}\rangle,\\
C_r(D, X_0\cup (\partial X\times I))&=&\Z\langle e^{r}_\beta\times e^0_1\mid 1\leq \beta\leq N_{r}\rangle,\\
C_r(X\times I,X_0 \cup \partial X_0 \times I)&=&\Z\langle e^{r-1}_\alpha\times e^1, e^{r}_\beta\times e^0_1\mid 1\leq \alpha\leq N_{r-1}, 1\leq \beta\leq N_{r}\rangle.
\end{array}
\]
We note that in the long exact cohomology sequence of the triple $(X\times I, D, X_0\cup (\partial X\times I))$, the connecting map $\delta$ is an isomorphism, because the inclusion map $X_0\cup (\partial X\times I)\to X\times I$ is a homotopy equivalence.

\begin{proposition}\label{prop:basicthing}
Consider the composition of maps
\[
\begin{tikzcd}
H^r(X\times I,D;G) \ar[r, "\delta^{-1}", "\cong"'] &H^{r-1}(D, X_0\cup (\partial X\times I);G) \ar[r, "\cong"', "\operatorname{exc}"]  & H^{r-1}(X_1,\partial X_1;G),
\end{tikzcd}
\]
where $\delta$ denotes the connecting map in the long exact cohomology sequence of the triple $(X\times I, D, X_0\cup (\partial X\times I))$ and $\operatorname{exc.}$ denotes the (inverse of the) excision isomorphism. Given a cocycle $\chi\in C^r(X\times I, D;G)$, the image $\operatorname{exc}(\delta^{-1}([\chi]))$ is represented by the cocycle
\[
\phi\colon C_{r-1}(X_1,\partial X_1)\to G;\qquad  e^{r-1}_\alpha\times e^0_1 \mapsto \chi(e^{r-1}_\alpha\times e^1).
\]
\end{proposition}

\begin{proof}
We first analyse the map $\delta^{-1}$. Fix $r\in\Z$ and consider the short exact sequence
\[
\begin{tikzcd}[column sep= small]
0\ar[r]
& C^r(X\times I,D)\ar[r, "j^*"]
& C^r(X\times I,X_0\cup (\partial X\times I))\ar[r, "i^*"]
& C^r(D,X_0\cup (\partial X\times I))\ar[r]
& 0.
\end{tikzcd}
\]
This is the short exact sequence of cochain complexes which induces the long exact sequence of the triple in cohomology.
Define a cochain $\psi\colon C_{r-1}(X\times I,X_0 \cup (\partial X\times I))\to G$ by extending linearly from
\[\psi(e^{r-1}_\beta\times e^0_1):=\chi(e^{r-1}_\beta\times e^1) \text{ and }\psi(e^{r-2}_\alpha\times e^1):=0.\]
 From this we compute the codifferential
\[
d^*(\psi)(e^{r-1}_\alpha\times e^1)=\psi(e^{r-1}_\alpha\times e^0_1) + \psi(\partial e^{r-1}_\alpha \times e^1) =\chi(e^{r-1}_\alpha\times e^1) + 0 = \chi(e^{r-1}_\alpha\times e^1).
\]
Now for $1\leq \beta\leq N_r$, write $d_*(e^r_\beta)=\sum_\gamma n_\gamma e^{r-1}_\gamma\in C_*(X,\partial X)$ for $n_\gamma\in\Z$. We compute the codifferential
\begin{align*}
d^*(\psi)(e^r_\beta\times e^0_1)&=\psi(d_*(e^r_\beta\times e^0_1))=\psi\left(\displaystyle{\sum}_\gamma n_\gamma e^{r-1}_\gamma\times e^0_1\right)=\displaystyle{\sum}_\gamma n_\gamma\chi(e^{r-1}_\gamma\times e^1)\\
&=\chi\left(\displaystyle{\sum}_\gamma n_\gamma e^{r-1}_\gamma\times e^1\right)=\chi(d_*(e^r_\beta\times e^1))=d^*(\chi)(e^r_\beta\times e^1)=0,
\end{align*}
where in the final equality we have used our assumption that $\chi$ is a cocycle. These computations confirm that $d^*(\psi)=j^*(\chi)$ as these two cochains agree on generators. From this it follows from the definition of $\delta$ that $i^*(\psi)\in C^{r-1}(D,X_0\cup (\partial X\times I))$ is a representative cocycle for~$\delta^{-1}([\chi])$.

Finally, due to our choice of cell structure, it is easy to see that $i^*(\psi)$ is sent to the cochain
\[
\phi\colon C_{r-1}(X_1,\partial X_1)\to G;\qquad e^{r-1}_\alpha\times e^0_1\mapsto \chi(e^{r-1}_\alpha\times e^1).
\]
under the excision map. We check that $\phi$ is a cocycle. To do this we use that $\chi$ is 0 on $D$, and therefore $\chi(e^r_\alpha \times e^0_0) = 0 = \chi(e^r_\alpha \times e^0_1)$ for every $\alpha$. Then we compute:
\begin{align*}
  d^*(\phi)(e^{r}_{\alpha} \times e^0_1) &= \phi(\partial e^r_\alpha \times e^0_1) = \chi(\partial e^r_\alpha \times e^1)
  = \chi(\partial e^r_\alpha \times e^1) - \chi(e^r_\alpha \times e^0_0) + \chi(e^r_\alpha \times e^0_1)  \\ & = \chi(\partial(e^r_\alpha \times e^1)) = d^*(\chi)(e^r_\alpha \times e^1) = 0.
\end{align*}
Therefore $\phi$ is a cocycle as desired.
\end{proof}

\begin{proposition}\label{prop:extensionfromboundary}
Let $X$ be a simply connected, $\TOP$ $4$-manifold and suppose that $F\in\Homeo^+(X,\partial X)$.
Fix a normal $2$-smoothing $\overline{\nu}_X\colon X\to B$.
There are two obstructions to extending the $B$-structure
\[
\phi\colon\partial (X\times[0,1])\to B;\qquad(x,t)\mapsto\left\{
\begin{array}{ll} \overline{\nu}_X\circ F(x) & \text{when $t=1$,}\\
\overline{\nu}_X(x) & \text{otherwise,}
\end{array}\right.
\]
to a normal $2$-smoothing
\[
\Phi\colon X\times[0,1]\to B.
\]
The primary obstruction is trivial in the non-spin case and is $\Theta(F)$ in the spin case. If this vanishes, the secondary obstruction is the Poincar\'{e} variation $\Delta_F$.

The analogous statement holds in the smooth category, where $X$ is a smooth manifold and $F\in\Diffeo^+(X,\partial X)$.
\end{proposition}

\begin{proof}
Let $\CAT=\Diff$ or $\TOP$.

 In the proof, we will set up a lifting problem in standard obstruction theory. We will prove there can only be two possible obstructions to solving this problem. The primary obstruction to lifting will be identified with $\Theta(F)$ in the spin case, and will automatically vanish in the non-spin case. The secondary obstruction to lifting will be identified with $\Delta_F$.

We make an initial observation that if we can produce a map $\Phi\colon X\times[0,1]\to B$ restricting to the specified map $\phi$ on the boundary then it will automatically be a normal $2$-smoothing of $X\times[0,1]$. To see this, let $\Phi$ be such a map and write $i\colon X\times\{0\}\hookrightarrow X\times[0,1]$ for the inclusion map. Noting that $i\circ\pr_1\simeq \Id_{X\times[0,1]}$, we have that $\Phi\circ i=\overline{\nu}_X$ implies $\Phi\simeq \overline{\nu}_X\circ\pr_1$.  From this, we see $p\circ\Phi\simeq (p\circ \overline{\nu}_X)\circ \pr_1\simeq \nu_X\circ\pr_1\simeq \nu_{X\times[0,1]}$, where recall $p\colon B\to \BSCAT$ is the normal $2$-type. This shows that $\Phi$ is a lift of $\nu_{X\times[0,1]}$ as required. Moreover, as $\pr_1$ is a homotopy equivalence and $\overline{\nu}_X$ is $3$-connected, it follows from the homotopy $\Phi\simeq \overline{\nu}_X\circ\pr_1$ that $\Phi$ is $3$-connected as required.

We turn to the problem of extending the given $B$-structure $\phi\colon \partial (X\times[0,1])\to B$ to a $B$-structure on the  whole of $X\times[0,1]$. Write $D=\partial (X\times[0,1])$. This is standard relative lifting problem in obstruction theory:
\[
\begin{tikzcd}
D\ar[rr, "\phi"]\ar[d, hook]\ar[rrd, "\nu_D"' near start]
&& B\ar[d, "p"]\\
X\times[0,1]\ar[rr, "\nu_{X\times[0,1]}"'] \ar[urr,dashed]
&& \BSCAT
\end{tikzcd}
\]
determining a sequence of obstructions $\mathfrak{o}_{r+1}\in H^{r+1}(X\times [0,1], D;\pi_r(\mathcal{F}))$, where $\mathcal{F}$ denotes the homotopy fibre of $p$.

We compute the homotopy groups of $\mathcal{F}$ and then show the cohomology groups above are untwisted by $\pi_1(\mathcal{F})$. Recall from Section~\ref{subsec:2type} the definition of $B$ as a certain homotopy pullback. A general property of a homotopy pullback square is that there is a weak homotopy equivalence between the homotopy fibres. Denoting  the homotopy fibre of the map $w$ used to define $B$ by $\mathcal{F}(w)$, there is a homotopy commutative diagram
\begin{equation}\label{eq:commutative}
\begin{tikzcd}
\mathcal{F}\ar[rr, "\alpha", "\simeq_{\text{weak}}"']\ar[d, "i"]
&& \mathcal{F}(w)\ar[d, "j"]\\
B\ar[rr, "q"] \ar[d,"p"]
&& K \ar[d,"w"] \\
\BSCAT \ar[rr] && K(\Z/2,2).
\end{tikzcd}
\end{equation}
Here $i$ and $j$ denote the respective inclusions of homotopy fibres and $\alpha$ is some weak homotopy equivalence. The long exact sequence associated with the right hand column contains $0\to \pi_r(\mathcal{F}(w))\to 0$ for $r \geq 3$, and ends with
\begin{align*}
  0\to \pi_2(\mathcal{F}(w))\xrightarrow{j_*} \pi_2(K)\xrightarrow{w_*}\pi_2(K(\Z/2,2))\to\pi_1(\mathcal{F}(w))\to 0.
\end{align*}
We deduce that $\pi_r(\mathcal{F}(w))=0$ for $r\geq 3$ and thus there are at most two nontrivial obstructions~$\mathfrak{o}_1$ and $\mathfrak{o}_2$. The map $w_*$ is surjective if and only if $X$ is not spin. We deduce that
\[
\pi_1(\mathcal{F})=\left\{\begin{array}{ll} \Z/2 &\text{when $X$ is spin,}\\ 0&\text{when $X$ is not spin,}\end{array}\right.\quad\text{and}\quad j_*\circ\alpha_*\colon \pi_2(\mathcal{F})\hookrightarrow \pi_2(K)\cong\Z
\]
is injective, with image consisting of classes on which $w_* \colon \pi_2(K) \to \pi_2(K(\Z/2,2))$ is zero.

To compute the action of $\pi_1(\mathcal{F})$ on $\pi_2(\mathcal{F})$ in the spin case, recall the alternative model for the  $2$-type of Lemma~\ref{lem:alternative2type} is a product, so has trivial $\pi_1$ action on $\pi_2$. As all normal $2$-types are fibre homotopy equivalent (Remark~\ref{rem:moorepost}), this shows that $\pi_2(\mathcal{F})$ has the trivial $\pi_1(\mathcal{F})$ action.

Now we analyse the obstructions~$\mathfrak{o}_1$ and $\mathfrak{o}_2$.

If $X$ is not spin, $\pi_1(\mathcal{F}) =0$ so $\mathfrak{o}_1=0$.  In the case that $X$ is spin, the obstruction $\mathfrak{o}_1\in H^2(X\times I, D;\Z/2)$ is potentially nontrivial. We claim that under the sequence of maps in Proposition~\ref{prop:basicthing}, this obstruction is sent to $\Theta(F)\in H^1(X,\partial X;\Z/2)$. To see this, recall by Lemma~\ref{lem:alternative2type} there is a model for the normal $2$-type of $X$ given by $(\BCATSpin\times K,\gamma\circ\pr_2)$. Moore-Postnikov theory~\cite[Theorem~1.5.8]{zbMATH03562120} shows there is a fibre homotopy equivalence $(B,p)\simeq(\BCATSpin\times K,\gamma\circ\pr_2)$.  Under this, the $B$-structure $\phi\colon D\to B$ corresponds to a pair $(\mathfrak{s},h)$, where $\mathfrak{s}$ is a spin structure on $\nu_D$ and $h\colon D\to K$ is a map. The obstruction $\mathfrak{o}_1\in H^2(X\times I, D;\Z/2)$ is identified with the primary obstruction to extending $(\mathfrak{s},h)$ to a $(\BCATSpin\times K)$-structure on all of $X\times I$. But since $\pi_1(K)=0$, this is just the obstruction to extending the spin structure $\mathfrak{s}$ to a spin structure $\mathfrak{S}$ on $\nu_{X\times I}$. In turn, this is exactly the obstruction to the existence of a rel.~boundary homotopy from $(\mathfrak{s},\partial \mathfrak{s})$ to $(\mathfrak{s}\circ F,\partial \mathfrak{s})$. In other words $\mathfrak{o}_1$ is identified with $\Theta(F)\in H^1(X_1,\partial X_1;\Z/2)$ under
the isomorphisms
$H^1(X\times I,D;\Z/2)\cong H^1(D,X_0\cup(\partial X\times I);\Z/2)\cong H^1(X_1,\partial X_1;\Z/2)$
from Proposition~\ref{prop:basicthing}.
By assumption, this latter obstruction vanishes.

From here on, we  use the cell complex and notation from Proposition~\ref{prop:basicthing}.
In both the case that $X$ is spin and the case it is not, we may now assume we have extended $\phi$ to the relative $2$-skeleton $\Phi^{(2)}\colon(X\times I)^{(2)}\to B$, lifting the restriction $\nu_{X\times I}|_{(X\times I)^{(2)}}$. In each case, the class $\mathfrak{o}_2\in H^3(X\times [0,1],D;\pi_2(\mathcal{F}))$ is now the unique remaining obstruction to extending the lift $\Phi^{(2)}$ to the whole of $X\times [0,1]$, relative to $\phi$ on $D$.

The remainder of the proof is devoted to describing the following sequence of maps, and proving that $\mathfrak{o}_2$ is sent to $-\Delta_F$ along the sequence:
\begin{equation}\label{eq:newdiagram}
\begin{tikzcd}
H^3(X\times [0,1], D;\pi_2(\mathcal{F}))\ar[r,hookrightarrow,"{j_*\circ\alpha_*}"]
& H^3(X\times [0,1], D;\pi_2(K))\ar[r, "(\text{Hur})_*", "\cong"']
&H^3(X\times [0,1], D;H_2(K))\\
\ar[r, "{\text{Prop.~\ref{prop:basicthing}}}", "\cong"']
&H^3(X\times [0,1], D;H_2(K))\ar[r, "\ev", "\cong"']
& \Hom(H_2(X_1, \partial X_1),H_2(K))\\
\ar[r, "\cong"']
& \Hom(H_2(X, \partial X),H_2(K))\ar[r, "(h_*)^{-1}", "\cong"']
& \Hom(H_2(X, \partial X),H_2(X)).
\end{tikzcd}
\end{equation}

To begin working our way along the sequence \eqref{eq:newdiagram}, recall that $j_*\circ\alpha_*\colon \pi_2(\mathcal{F})\to \pi_2(K)$ is an injective map of free abelian groups. As free abelian groups are flat over $\Z$, there is a commutative diagram
\[
\begin{tikzcd}
H^3(X\times I,D;\pi_2(\mathcal{F}))\ar[rr, "{(j_*\circ\alpha_*)_*}"]\ar[d, "\cong"]
&&  H^3(X\times I,D;\pi_2(K)) \ar[d, "\cong"]\\
H^3(X\times I,D;\Z)\otimes\pi_2(\mathcal{F})\ar[rr, "{\Id\otimes(j_*\circ\alpha_*)}"]
&&  H^3(X\times I,D;\Z)\otimes \pi_2(K)
\end{tikzcd}
\]
Now $j_*\circ\alpha_*$ is injective and $H^3(X\times I,D;\Z)\cong H_2(X\times I;\Z)\cong H_2(X;\Z)$ is a free $\Z$-module, so the lower horizontal map is injective. Hence the upper horizontal map is injective.

We now describe a cocycle representative for $j_*(\alpha_*(\mathfrak{o}_2))$. We may use a CW complex homotopy equivalent to $K$, with a single basepoint and no 1-cells (for instance, use the standard CW structure for $\prod \CP^{\infty}\simeq K$), and we may assume that both $\phi$ and $\Phi^{(2)}$ are cellular. As at the start of Section~\ref{section:extending-B-structure}, we work with cellular chain and cochain groups.  The obstruction $\mathfrak{o}_2$ is represented by a cocycle
\[
C_3(X\times I, D;\Z)\to \pi_2(\mathcal{F});\qquad e^2_\alpha\times e^1\mapsto \Phi^{(2)}(e^2_\alpha\times e^1).
\]
We may assume $q\colon B\to K$ is cellular. Write $J := q\circ \phi\colon D\to K$ and $H^{(2)}: = q\circ \Phi^{(2)}\colon (X\times I)^{(2)}\to K$, which we note are again cellular. Considering the homotopy commutative diagram~(\ref{eq:commutative}),  we obtain a cocycle representative for $j_*(\alpha_*(\mathfrak{o}_2))=q_*(i_*(\mathfrak{o}_2))$ as
\[
\eta\colon C_3(X\times I, D;\Z)\to \pi_2(K);\qquad e^2_\alpha\times e^1\mapsto H^{(2)}(\partial(e^2_\alpha\times e^1)).
\]
This is now a cocycle representative for the unique obstruction in $H^3(X\times I;\pi_2(K))$ for the simpler extension problem of extending $J\colon D\to K$ to a map $H\colon X\times I\to K$. We now proceed to work  ``by hand'' to relate this cohomology class to $\Delta_F$.

We may assume $\overline{\nu}_X$ is cellular. We write $h:=q\circ\overline{\nu}_X\colon X\to K$ for brevity, and note this is an isomorphism on $H_2$. Under the Hurewicz isomorphism $H^3(X\times I,D;\pi_2(K))\cong H^3(X\times I,D;H_2(K))$ the obstruction $[\eta]$ is sent to a class $[\chi]$, where $\chi$ is a cocycle $\chi\colon C_3(X\times I, D;\Z)\to H_2(K)$ such that
\begin{align*}
\chi(e_\alpha^2\times e^1)&=[H^{(2)}(\partial (e_\alpha^2\times e^1))]\\
&=\left[ J(e^2_\alpha\times e^0_1) \right] - \left[ J(e^2_\alpha\times e^0_0) \right] + [H^{(2)}(\partial e^2_\alpha\times e^1)]\\
&=\left[ h(F(e^2_\alpha)) \right] - \left[ h(e^2_\alpha) \right] + [H^{(2)}(\partial e^2_\alpha\times e^1)].
\end{align*}
By Proposition~\ref{prop:basicthing}, under the isomorphism $H^3(X\times I,D;H_2(K))\cong H^2(X_1,\partial X_1;H_2(K))$, the class $[\chi]$ is sent to $[\theta]$ where
\[
\theta\colon C_2(X_1,\partial X_1)\to H_2(K);\qquad e^2_\alpha\times e^0_1\mapsto \chi(e^2_\alpha\times e^1).
\]
Because $H_1(X_1,\partial X_1;\Z)\cong\Z^{r-1}$ is torsion-free, the universal coefficient theorem implies the evaluation map
\[
\begin{tikzcd}
H^2(X_1,\partial X_1;H_2(K))\ar[r, "\operatorname{ev}", "\cong"']
& \Hom(H_2(X_1,\partial X_1;\Z),H_2(K))
\end{tikzcd}
\]
is an isomorphism.
Let $\sum_{\alpha}N_\alpha (e^2_\alpha\times e^0_1)\in C_2(X_1,\partial X_1;\Z)$ be a cycle. We compute
\begin{align*}
\MoveEqLeft[9]\operatorname{ev}([\theta])\left( \left[ \textstyle{\sum_{\alpha}}N_\alpha (e^2_\alpha\times e^0_1) \right] \right)\\
={}& \theta\left(\textstyle{\sum_{\alpha}}N_\alpha (e^2_\alpha\times e^0_1))\right)=\chi\left(\textstyle{\sum_{\alpha}}N_\alpha (e^2_\alpha\times e^1)\right)\\
={}&\left[ h(F(\textstyle{\sum_{\alpha}}N_\alpha e^2_\alpha)) \right] - \left[ h(\textstyle{\sum_{\alpha}}N_\alpha e^2_\alpha) \right]+[H^{(2)}(\textstyle{\sum_{\alpha}}N_\alpha (\partial e^2_\alpha\times e^1)) ]\\
={}&h_*\left[F(\textstyle{\sum_{\alpha}}N_\alpha e^2_\alpha) - \textstyle{\sum_{\alpha}}N_\alpha e^2_\alpha)\right]+[H^{(2)}(\textstyle{\sum_{\alpha}}N_\alpha (\partial e^2_\alpha\times e^1) )]\\
={}&h_*\left[F(\textstyle{\sum_{\alpha}}N_\alpha e^2_\alpha) - \textstyle{\sum_{\alpha}}N_\alpha e^2_\alpha)\right]\\
={}&h_*\Delta_F\left(-[\textstyle{\sum_{\alpha}}N_\alpha e^2_\alpha]\right).
\end{align*}
To see the penultimate equality, consider that $\partial \left(\textstyle{\sum_{\alpha}}N_\alpha e^2_\alpha\right)= \textstyle{\sum_{\alpha}}N_\alpha \partial e^2_\alpha=0$, by definition of a cycle. This implies that $\textstyle{\sum_{\alpha}}N_\alpha \left(\partial e^2_\alpha\times e^1\right)=0$. The final equality comes from the definition of $\Delta_F$. As $\sum_{\alpha}N_\alpha (e^2_\alpha\times e^0_1)$ was an arbitrary cycle, and using the obvious identification of $(X,\partial X)$ with $(X_1,\partial X_1)$, we have shown $h_*(-\Delta_F)=\ev([\theta])$. As $h_*=q_*\circ(\overline{\nu}_X)_*\colon H_2(X)\to H_2(K)$ is an isomorphism, we obtain $(h_*)^{-1}\ev([\theta])=-\Delta_F$.

This completes the proof that under the sequence \eqref{eq:newdiagram}, the obstruction $\mathfrak{o}_2$ is sent to $-\Delta_F$ via $\mathfrak{o}_2\mapsto [\eta]\mapsto[\chi]\mapsto[\theta]\mapsto\ev([\theta])\mapsto-\Delta_F$. As all maps in the sequence \eqref{eq:newdiagram} have been shown to be either isomorphisms or injections, this completes the proof that $\Delta_F$ is the secondary obstruction to the lifting problem. This completes the proof of Proposition~\ref{prop:extensionfromboundary}.
\end{proof}

\subsection{Proof of Theorem~\ref{thm:main}}\label{subsection:proof-thm-main}

We now have all the ingredients to complete the proof of Theorem~\ref{thm:main}, which states that our invariants detect pseudo-isotopy of homeomorphisms and diffeomorphisms, in the appropriate categories.

\begin{proof}[Proof of Theorem~\ref{thm:main}]
Let $X$ be a compact, simply connected, smooth 4-manifold with (possibly empty) boundary and suppose that $F\in\Diffeo^+(X,\partial X)$ such that $\Delta_F\colon H_2(X,\partial X)\to H_2(X)$ is the identity map. In the case that $X$ is spin, furthermore assume that $\Theta(F)=0$. To prove the theorem, we will show $F$ is pseudo-isotopic to $\Id_X$ rel.~$\Id_{\partial X}$.

Consider the smooth, closed 5-manifold
\begin{equation}\label{eq:boundarydecomp}
T:= (X\times[0,1])_\ell\cup\left( \partial (X\times[0,1])\times[0,1]\right)\cup(X\times[0,1])_r.
\end{equation}
The subscript $\ell$ and $r$ are to indicate ``left'' and ``right'', and are simply a bookkeeping device. The second union in \eqref{eq:boundarydecomp} is given by the map
\[
\partial(X\times[0,1])_r\mapsto  \partial (X\times[0,1])\times\{1\};\qquad (x,t)\mapsto
\left\{\begin{array}{cl}
(F^{-1}(x),1,1) &\text{for $t=1$},\\
(x,t, 1)&\text{for $t\in[0,1)$}.
\end{array}\right.
\]
and first union in \eqref{eq:boundarydecomp} is given by $(x,t)\mapsto (x,t,0)$ for all $(x,t)\in \partial (X\times[0,1])_\ell$. Fix a normal $2$-smoothing $\overline{\nu}_X\colon X\to B$. Endow $(X\times[0,1])_r$ with the product normal 2-smoothing $\overline{\nu}_X\circ\pr_1$. Write $M$ for the union $(X\times[0,1])_\ell\cup\left( \partial (X\times[0,1])\times[0,1]\right)$. We see that~$M$ is $(X\times[0,1])_\ell$ with an external collar attached, and thus there is a diffeomorphism of manifolds with corners $M\cong X\times[0,1]$. Using this diffeomorphism, choose a normal $2$-smoothing $(M,\Phi)$ with the properties described in Proposition~\ref{prop:extensionfromboundary}. These properties guarantee that the normal $2$-smoothings $\Phi$ and $\overline{\nu}_X\circ\pr_1$ glue to form a $B$-structure $\xi$ on $T$. By Proposition~\ref{prop:bordismisboring}, $(T,\xi)$ is $B$-null-bordant. By Lemma~\ref{lem:sanitycheck}, a choice of $B$-null-bordism determines a $B$-bordism $(Z,\Xi)$ rel.~boundary between the normal 2-smoothings $((X\times[0,1])_\ell,\Phi|)$ and $((X\times[0,1])_r,\overline{\nu}_X\circ\pr_1)$. The bordism $Z$ is a manifold with corners, where the corner structure of $\partial Z = T$ corresponds to the decomposition in \eqref{eq:boundarydecomp}. By Theorem~\ref{thm:kreckmain2}, $(Z,\Xi)$ determines a surgery obstruction in
\[
\left\{\begin{array}{ll}\displaystyle{L^s_6(\Z)}&\text{if $B$ has only trivial $w_4$-spheres,}\\
\displaystyle{L^s_6(\Z, S=\Z)}&\text{otherwise,}\end{array}\right.
\]
that vanishes if and only if $(Z,\Xi)$ is $B$-bordant rel.~boundary to an $h$-cobordism rel.~boundary.
By Proposition~\ref{prop:w4} we are in the latter case, and by \cite[Example 4.10]{kreckmonograph} the obstruction group $L^s_6(\Z, S=\Z)$ is trivial. Thus there is a 6-dimensional $h$-cobordism rel.~boundary $B$-bordant rel.~boundary to $Z$.  Let $W$ be such an $h$-cobordism. By Smale's $h$-cobordism theorem~\cite{Smale-h-cob-thm} there is a commutative diagram of maps,  with vertical diffeomorphisms, identifying $W$ with a product cobordism rel.\ boundary from $X \times [0,1]$ to itself:
\[
\begin{tikzcd}
X\times[0,1]\ar[r, "i_0"] \ar[d, "="']& X\times[0,1]\times[0,1]\ar[d, "\cong"', "G"] &\ar[l, "i_1"'] X\times[0,1]\ar[d, "\cong"', "g"]\\
(X\times[0,1])_\ell\ar[r, "j_0"] & W &\ar[l, "j_1"'] (X\times[0,1])_r.
\end{tikzcd}
\]
Here $i_k(x,t)=(x,t,k)$ for $k=0,1$, and the maps $j_k$ are the inclusion maps.

We verify that $g$ is pseudo-isotopic to the identity. Note $G$ is the identity upon restriction to $\partial(X\times[0,1])\times[0,1]$. As $W$ is rel.~boundary bordant to~$Z$, we have that $\partial W = T$. By commutativity of the diagram, we compute that $j_1\circ g(x,1)= G\circ i_1(x,1)=(x,1)$ for all $x\in X$. But by the definition of $T$ we have that $j_1(z,1)=(F^{-1}(z),1)$ for all $z\in X$, whence $g$ restricts to $F$ on $X\times\{1\}$. Similarly $g$ restricts to $\Id_X$ on $X\times\{0\}$. Thus $g$ is a pseudo-isotopy from $F$ to $\Id_X$.

The proof in the $\CAT=\TOP$ case is entirely analogous. Note that we perform topological surgery on a 6-manifold $Z$, in order to convert it into an $h$-cobordism. We then apply the topological 6-dimensional $h$-cobordism theorem. To apply surgery in the topological category one needs \cite[Essay~III,~Appendix~C]{MR0645390}. The details necessary for adapting the proof of the smooth $h$-cobordism theorem to the topological case are explained in~\cite[Essay~III, \S 3.4]{MR0645390}.
\end{proof}

\section{The role played by isometries of the intersection form}\label{sec:isometries}

In this section we first describe a purely algebraic formulation for groups of Poincar\'{e} variations. We thereby show that the isomorphism class of such a group depends only on the underlying symmetric bilinear form. We then recall some results of Boyer and Saeki relating Poincar\'{e} variations to isometries. Finally, we describe their consequences in light of Theorem~\ref{theoremA}.

\subsection{An algebraic formulation of Poincar\'{e} variations}

Let $P$ be a finitely generated free $\Z$-module, and let $\alpha\colon P\times P\to \Z$ be a symmetric bilinear form. In this section, write $(-)^{\ast}:=\Hom(-,\Z)$ and write $\alpha^{\ad}\colon P\to P^{\ast}$ for the adjoint homomorphism, defined by $\alpha^{\ad}(x)(y):=\alpha(x,y)$.

\begin{definition}
Given $(P,\alpha)$, define the set of \emph{form variations} as
\[
\FV(P,\alpha):=\{V\colon P^{\ast}\to P\mid V+V^{\ast}=V\alpha^{\ad}V^{\ast}\}\subseteq\Hom_\Z(
P^*,P).
\]
\end{definition}

\begin{lemma}\label{lem:FVisgroup}
The set $\FV(P,\alpha)$ is a group under the operation $V_1\ast V_2:=V_1+(\Id-V_1\alpha^{\ad})V_2$, with identity element 0 and inverses $V^{-1}:=-(\Id-V^*\alpha^{\ad})V$.
\end{lemma}

The proof of Lemma~\ref{lem:FVisgroup} is essentially the same as that of Lemma~\ref{lem:saekigroup}, and we omit it.

\begin{lemma}\label{lem:isometryinv}
Let $\psi\colon (P,\alpha)\xrightarrow{\cong} (Q,\beta)$ be an isometry of symmetric bilinear forms. Then
\[
\widehat{\psi}\colon\FV(P,\alpha)\xrightarrow{\cong}\FV(Q,\beta);\qquad \widehat{\psi}(V):=\psi V \psi^{\ast}
\]
is an isomorphism of groups.
\end{lemma}

\begin{proof}
Write $U:= \widehat{\psi}(V)=\psi V \psi^{\ast}$.
Note that $U^{\ast}=\psi V^{\ast}\psi^{\ast}$, we have that
\[
U + U^{\ast} = \psi (V+V^{\ast})\psi^{\ast}=\psi (V\alpha^{\ad} V^{\ast})\psi^{\ast}=(\psi V \psi^{\ast})((\psi^{\ast})^{-1}\alpha^{\ad}\psi^{-1})(\psi V^{\ast}\psi^{\ast})=U\beta^{\ad}U^{\ast},
\]
so $U$ lies in the codomain as claimed. Given this, $\widehat{\psi}$ is clearly bijective with inverse $\widehat{\psi^{-1}}$. To check that $\widehat{\psi}$ is a homomorphism, we compute
\[
\begin{array}{rcl}
\widehat{\psi}(V_1\ast V_2)&=&\psi(V_1+(\Id-V_1\alpha^{\ad})V_2)\psi^{\ast}\\
&=&(\psi V_1 \psi^{\ast})+\psi(\Id - V_1 (\psi^{\ast}(\psi^{\ast})^{-1}\alpha^{\ad})\psi^{-1}\psi V_2 \psi^{\ast}\\
&=&(\psi V_1 \psi^{\ast})+(\Id-\psi V_1 \psi^{\ast}\beta^{\ad})\psi V_2 \psi^{\ast}\\
&=& \widehat{\psi}(V_1)\ast \widehat{\psi}(V_2).
\end{array}
\]
Thus $\widehat{\psi}$ is an isomorphism as claimed.
\end{proof}

\begin{lemma}\label{lem:omitproof}
Suppose that $X$ is a compact, simply connected $\CAT$ $4$-manifold with (possibly empty) boundary. Then there is an isomorphism
\[
\FV(H_2(X),\lambda_X)\to \mathcal{V}(H_2(X),\lambda_X);\qquad V\mapsto V\circ \ev\circ\PD^{-1}.
\]
\end{lemma}

\begin{proof}Write $m(V):=V\circ \ev\circ\PD^{-1}$ for the map in the statement. For brevity, write the isomorphisms $\delta=\ev\circ\PD^{-1}\colon H_2(X,\partial X)\to H_2(X)^*$ and $\beta=\ev\circ\PD^{-1}\colon H_2(X)\to H_2(X,\partial X)^*$. Write $j_*\colon H_2(X)\to H_2(X,\partial X)$ for the inclusion induced map. Recall that if $\Delta$ is a variation then $\Delta^!:=\beta^{-1}\circ \Delta^*\circ \delta$. Note in this notation that $m(V)=V\circ\delta$.

Given $V\in \mathcal{FV}(H_2(X),\lambda_X)$, define $\Delta:=V\circ \delta$. We verify that the condition $V+V^*=V\circ\lambda_X^{\ad}\circ V^*$ is equivalent to the condition that $\Delta$ is a Poincar\'{e} variation: $\Delta+\Delta^!=\Delta\circ j_*\circ \Delta^!$ (Definition~\ref{def:algebraicvariational}). The reader will confirm that, by definition, $\lambda_X^{\ad}=\delta\circ j_*$. The reader may also confirm that, by naturality under inclusion, a diagram chase shows $(j_*)^*=\delta\circ j_*\circ \beta^{-1}$. Thus the symmetric condition $\lambda_X^{\ad}=(\lambda_X^{\ad})^*$ becomes
\[
\delta\circ j_*=(j_*)^*\circ\delta^*=\delta\circ j_*\circ \beta^{-1}\circ\delta^*,
\]
so that $\Id=\beta^{-1}\circ\delta^*$. We compute
\[
\begin{array}{rrcl}
& \Delta+\Delta^!
&=&
\Delta\circ j_*\circ \Delta^!
\\
\iff
& V\circ\delta+\beta^{-1}\circ (V\circ\delta)^*\circ \delta
&=&
(V\circ\delta)\circ j_*\circ (\beta^{-1}\circ (V\circ\delta)^*\circ \delta)
\\
\iff
& V\circ\delta+(\beta^{-1}\circ \delta^*)\circ V^*\circ \delta
&=&
V\circ(\delta\circ j_*)\circ (\beta^{-1}\circ \delta^*)\circ V^*\circ \delta
\\
\iff
& V\circ\delta+ V^*\circ \delta
&=&
V\circ\lambda_X^{\ad}\circ V^*\circ \delta
\\
\iff
& V+ V^*
&=&
V\circ\lambda_X^{\ad}\circ V^*
\end{array}
\]
So the conditions are indeed equivalent.

This verification confirms both that the function $m\colon V\mapsto V\circ\delta$ is well-defined and that the inverse map $\Delta\mapsto \Delta\circ\delta^{-1}$ is well-defined. This shows $m$ is a well-defined bijection. As the group structure on $\mathcal{FV}(H_2(X),\lambda_X)$ was specifically designed to make $m$ a homomorphism, the proof that $m$ is an isomorphism is complete.
\end{proof}

\begin{corollary}\label{cor:thebusiness}
Let $X$ be a compact, simply connected $\CAT$ $4$-manifold with (possibly empty) boundary. The group of variations $\mathcal{V}(H_2(X),\lambda_X)$ depends only on the isometry class of the intersection form. Moreover, if $\iota\colon (X,\partial X)\to (X',\partial X')$ is a map of compact, simply connected $\CAT$ $4$-manifolds, inducing an isometry of the intersection forms, then there is an induced isomorphism
\[
\mathcal{V}(H_2(X),\lambda_X)\xrightarrow{\cong} \mathcal{V}(H_2(X'),\lambda_{X'});\qquad \Delta\mapsto \iota_*\circ\Delta\circ(\iota_*)^{-1}.
\]
\end{corollary}

\begin{proof} The first statement is an immediate corollary of Lemmas~\ref{lem:isometryinv} and~\ref{lem:omitproof}. For the second statement, it is sufficient to consider the following composition from left to right
\[
\begin{tikzcd}[row sep = 1em]
\mathcal{V}(H_2(X),\lambda_X)
&\ar[l, "\cong"']\mathcal{FV}(H_2(X),\lambda_X)\ar[r, "\widehat{\iota}_*", "\cong"']
& \mathcal{FV}(H_2(X'),\lambda_{X'})\ar[r, "\cong"]
&\mathcal{V}(H_2(X'),\lambda_{X'})\\
\Delta \ar[r, mapsto]
& \Delta\circ\PD\circ \ev^{-1}\ar[r, mapsto]
& \iota_*\circ\Delta\circ\PD\circ \ev^{-1}\circ(\iota_*)^{{\ast}}\ar[r,mapsto]
& \iota_*\circ\Delta\circ(\iota_*)^{-1},
\end{tikzcd}
\]
where in the last map we have used the fact that $\PD\circ \ev^{-1}\circ(\iota_*)^{{\ast}}\circ \ev \circ\PD^{-1}=(\iota_*)^{-1}$, which follows from naturality of the cap product.
\end{proof}

\subsection{Isometries rel.~boundary and Saeki's exact sequence}

Suppose that $X$ is a compact, simply connected $\CAT$ $4$-manifold with (possibly empty) boundary. Let $\phi\colon H_2(X)\to H_2(X)$ be a group isomorphism.
Write $(-)^{\ast}:=\Hom(-,\Z)$, and define a composite ``umkehr'' isomorphism
\[
\phi^!\colon H_2(X,\partial X)\xrightarrow{\PD^{-1}}H^2(X) \xrightarrow{\ev} H_2(X)^{\ast}\xrightarrow{\phi^{\ast}} H_2(X)^{\ast}\xrightarrow{\ev^{-1}} H^2(X) \xrightarrow{\PD} H_2(X,\partial X).
\]
The long exact sequence of $(X,\partial X)$ determines the following short exact sequences, because $X$ is simply connected.
\[
\begin{tikzcd}
0 \ar[r]
& H_2(\partial X) \ar[r]\ar[d,"\Id"]
& H_2(X) \ar[r]\ar[d,"\phi"]
& H_2(X, \partial X) \ar[r]
& H_1(\partial X) \ar[r]\ar[d,"\Id"]
& 0\\
0 \ar[r]
& H_2(\partial X) \ar[r]
& H_2(X) \ar[r]
& H_2(X, \partial X) \ar[r]\ar[u,"\phi^!"]
& H_1(\partial X) \ar[r]
& 0.
\end{tikzcd}
\]
The reader may check that the map $\phi$ is an isometry of $(H_2(X),\lambda_X)$ if and only if the central square of the diagram commutes. We say $\phi$ is an \emph{isometry rel.~boundary} if the entire diagram above commutes.

\begin{definition}\label{def:algebraicidentity} Write $\Aut(H_2(X),\lambda_X)$ for the group of isometries of the intersection form, where the group operation is by composition. Write $\Aut_\partial(H_2(X),\lambda_X)\subseteq \Aut(H_2(X),\lambda_X)$ for the subgroup consisting of isometries rel.~boundary.
\end{definition}

Given a closed oriented $3$-manifold $Y$, denote by $\Spin(Y)$ the set of spin structures on $Y$ up to isomorphism. Given $\mathfrak{s}\in\Spin(\partial X)$, the \emph{relative second Stiefel-Whitney class} $w_2(X,\partial X;\mathfrak{s})\in H^2(X,\partial X;\Z/2)$ is defined as the unique obstruction to solving the following relative lifting problem in obstruction theory
\[
\begin{tikzcd}
\partial X\ar[rr, "\mathfrak{s}"]\ar[d, hook]\ar[rrd, "\nu_{\partial X}"' near start]
&& \BTOPSpin\ar[d, "\gamma"]\\
X\ar[rr, "\nu_{X}"'] \ar[urr,dashed]
&& \BSTOP.
\end{tikzcd}
\]

Let $\phi\in\Aut(H_2(X),\lambda_X)$. It is shown in \cite[Lemma~2.9]{MR857447} that there is a unique permutation $\pi_\phi$ of $\Spin (\partial X)$ such that the following diagram commutes.
\begin{equation}\label{eq:permute}
\begin{tikzcd}
\Spin(\partial X)\ar[rrr, "{w_2(X,\partial X;-)}"]\ar[d, "\pi_\phi"]
&&& H^2(X,\partial X;\Z/2)\ar[r, "\PD"]
& H_2(X;\Z/2)\ar[d, "\phi"]
\\
\Spin(\partial X)\ar[rrr, "{w_2(X,\partial X;-)}"]
&&& H^2(X,\partial X;\Z/2)\ar[r, "\PD"]
& H_2(X;\Z/2).
\end{tikzcd}
\end{equation}

We remark that in \cite[Lemma~2.9]{MR857447}, Boyer is assuming that $\partial X$ is connected, but that the proof is entirely formal and goes through identically in our case.

\begin{definition}\label{def:algebraicfix}
Write $\Aut^{\fix}_{\partial}(H_2(X),\lambda_X)\subseteq \Aut_\partial(H_2(X),\lambda_X)$ for the subgroup of isometries rel.~boundary $\phi$ such that $\pi_\phi$ is the identity permutation.
\end{definition}

\begin{remark}\label{rem:thetaclass}
We explain why $\pi_\phi = \Id$ is a necessary condition for $\phi \in \Aut_\partial(H_2(X),\lambda_X)$ to be induced by a homeomorphism.
Suppose $f\colon \partial X\to \partial X$ is an orientation preserving homeomorphism. Following Boyer's convention, we denote by $f_\#\colon \Spin(\partial X)\to \Spin(\partial X)$ the \emph{inverse} of the permutation described by $\mathfrak{s}\mapsto f^*\mathfrak{s}$. Boyer~\cite[\textsection4]{MR857447} proves that there exists a cohomology class $\theta(f,\phi)\in \im(H^1(\partial X)\to H^1(\partial X;\Z/2))$ such that the permutation $f_\#\circ\pi_\phi^{-1}$ agrees with the permutation determined by $\theta(f,\phi)$ under the free, transitive action of $H^1(\partial X;\Z/2)$ on $\Spin(\partial X)$. We note that although Boyer is assuming that $\partial X$ is connected, the proof of this result works equally well in the case of non-connected boundary. By Boyer~\cite[Theorem~0.7]{MR857447}, the class $\theta(f,\phi)$ is an obstruction to the existence of an orientation preserving homeomorphism $F\colon X\to X$, such that $F_*=\phi$ and $F|_{\partial X}=f$. We recall the details to confirm this holds even for non-connected boundary. Indeed, if such an $F$ exists, then the following diagram is commutative
\[
\begin{tikzcd}
\Spin(\partial X)\ar[rrr, "{w_2(X,\partial X;-)}"]\ar[d, "\pi_\phi"]
&&& H^2(X,\partial X;\Z/2)\ar[r, "\PD"]
& H_2(X;\Z/2)\ar[d, "\phi"]
\\
\Spin(\partial X)\ar[rrr, "{w_2(X,\partial X;-)}"]
&&& H^2(X,\partial X;\Z/2)\ar[r, "\PD"]\ar[u,"F^*"']
& H_2(X;\Z/2).
\end{tikzcd}
\]
Moreover it is easily checked from the definitions that for any $\mathfrak{s}\in\Spin(\partial X)$, we have
\[
F^*w_2(X,\partial X;\mathfrak{s})=w_2(X,\partial X;f^*\mathfrak{s}).
\]
Hence the permutations $\pi_\phi$ and $f_\#$ agree, and from this we can conclude $\theta(f,\phi)=0$.

In particular, $(\Id_{\partial X})_\#=\Id$, and so $\Id_{\partial X}$ extends to a homeomorphism inducing a given isometry $\phi$ only if $\pi_\phi$ is the identity permutation.
\end{remark}

\begin{lemma}\label{lem:spinkillsfix}
When $X$ is spin, $\Aut^{\fix}_{\partial}(H_2(X),\lambda_X)= \Aut_{\partial}(H_2(X),\lambda_X)$.
\end{lemma}

\begin{proof}
Fix $\phi\in\Aut_\partial(H_2(X),\lambda_X)$. We must show that $\pi_\phi=\Id$. Fix a spin structure $\mathfrak{s}$ on $X$ and let $\mathfrak{t}$ be the induced spin structure on $\partial X$. Then $w_2(X,\partial X;\mathfrak{t})=0$, so $\mathfrak{t'}:=\pi_\phi(\mathfrak{t})$ must be such that $w_2(X,\partial X;\mathfrak{t}')=0$, by commutativity of diagram~\eqref{eq:permute}. As $X$ is simply connected, it has a unique spin structure up to isomorphism, so only one spin structure on $\partial X$ can extend to $X$. Hence $\mathfrak{t}=\mathfrak{t}'$, up to isomorphism. By Remark~\ref{rem:thetaclass}, the permutation $\pi_\phi$ arises from the action of a class $-\theta(\Id_{\partial X},\phi)\in H^1(\partial X;\Z/2)$. As the action of $H^1(\partial X;\Z/2)$ on $\Spin(\partial X)$ is free, the fact that $(-\theta)\cdot \mathfrak{t}=\mathfrak{t}$ implies $-\theta=0$ and hence the permutation $\pi_\phi$ is trivial.
\end{proof}

\begin{remark}
We do not know an example of $(X,\partial X)$ and $\phi\in\Aut_\partial(H_2(X),\lambda_X)$ such that~$\pi_\phi\neq \Id$.
\end{remark}

Following \cite[\textsection4]{MR2197449}, we will now precisely describe how Poincar\'{e} variations refine isometries of the intersection form.
Denote by $\wedge^2{H_1(\partial X)^{\ast}}$ the group of (possibly degenerate) skew-symmetric pairings $\kappa\colon H_1(\partial X)\times H_1(\partial X)\to\Z$. Given $\kappa \in \wedge^2{H_1(\partial X)^{\ast}}$, write $\kappa^{\ad}$ for its adjoint, and define a composite morphism
\begin{equation}\label{eqn:kappa-tilde}
\widetilde{\kappa}\colon H_2(X,\partial X)\xrightarrow{\partial}H_1(\partial X)\xrightarrow{\kappa^{\ad}}H_1(\partial X)^{\ast}\xrightarrow{\ev^{-1}}H^1(\partial X)\xrightarrow{\PD}H_2(\partial X)\xrightarrow{i_*}H_2(X).
\end{equation}
The morphism $\wt{\kappa}$ is a Poincar\'{e} variation, and this defines a homomorphism
\[
S\colon \wedge^2{H_1(\partial X)^{\ast}}\to\mathcal{V}(H_2(X),\lambda_X);\qquad \kappa\mapsto \wt{\kappa};
\]
see \cite[\textsection4]{MR2197449}. Write $j_*\colon H_2(X)\to H_2(X,\partial X)$ for the inclusion induced map. Given $\Delta\in\mathcal{V}(H_2(X),\lambda_X)$, it is shown in \cite[Lemma~3.4]{MR2197449} that the induced map
\[
\Xi(\Delta):=\Id-\Delta\circ j_*\colon H_2(X)\to H_2(X)
\]
is an isometry of the intersection form $(H_2(X),\lambda_X)$.

\begin{theorem}\label{thm:Saekialgebraic}
The following is a short exact sequence
\[
\begin{tikzcd}
0\ar[r]
& \wedge^2{H_1(\partial X)^{\ast}}\ar[r, "S"]
&\mathcal{V}(H_2(X),\lambda_X)\ar[r, "\Xi"]
&\Aut_\partial^{\fix}(H_2(X),\lambda_X)\ar[r]
& 0.
\end{tikzcd}
\]
\end{theorem}

\begin{proof}
When $\partial X$ is connected, this result is due to Saeki, and follows from combining~\cite[Propositions~4.2 and 4.3]{MR2197449}.

We now assume that $X$ has $r>1$ connected boundary components. The proof of \cite[Proposition~4.2]{MR2197449} is entirely algebraic and does not use the assumption of connected boundary. This shows there is an exact sequence
\[
\begin{tikzcd}
0\ar[r]
& \wedge^2{H_1(\partial X)^{\ast}}\ar[r, "S"]
&\mathcal{V}(H_2(X),\lambda_X)\ar[r, "\Xi"]
&\Aut(H_2(X),\lambda_X).
\end{tikzcd}
\]
To complete the proof, we claim that an isometry $\phi\in \Aut(H_2(X),\lambda_X)$ lies in the image of $\Xi$ if and only if $\phi\in \Aut_\partial^{\fix}(H_2(X),\lambda_X)$.

First, assume that $\phi=\Xi(\Delta)$, for some $\Delta$. By Theorem~\ref{thm:realisevariational}, there exists $F\in\Homeo^+(X,\partial X)$ such that $\Delta_F=\Delta$, and thus $\phi=F_*$. This implies that $\phi\in \Aut_\partial(H_2(X),\lambda_X)$. Moreover, as discussed in Remark~\ref{rem:thetaclass}, the existence of $F$ implies $\pi_\phi=\Id$, so that moreover $\phi\in \Aut_\partial^{\fix}(H_2(X),\lambda_X)$.

Conversely, suppose that $\phi\in \Aut_\partial^{\fix}(H_2(X),\lambda_X)$. Form an internal connected sum manifold $(X_0,\partial X_0)$ as described in Construction~\ref{constr:internal}. Write $\iota\colon (X_0,\partial X_0)\to (X,\partial X)$ for the inclusion. The commutative diagram \eqref{eq:commutingdiagram} just above Lemma~\ref{lem:internalvariational} shows that $\phi_0:=(\iota_*)^{-1}\circ \phi\circ \iota_*\in \Aut_\partial(H_2(X_0),\lambda_{X_0})$. Moreover, the following diagram commutes
\[
\begin{tikzcd}
\Spin(\partial X_0)\ar[rrr, "{w_2(X_0,\partial X_0;-)}"]\ar[ddd,"\pi_{\phi_0}"', bend right = 70]
&&& H^2(X_0,\partial X_0;\Z/2)\ar[r, "\PD"]
& H_2(X_0;\Z/2)\ar[d, "\iota_*"]\ar[ddd,"\phi_0", bend left = 70]
\\
\Spin(\partial X)\ar[rrr, "{w_2(X,\partial X;-)}"]\ar[d, "\pi_\phi"]\ar[u,"\iota^*"']
&&& H^2(X,\partial X;\Z/2)\ar[r, "\PD"]\ar[u,"\iota^*"']
& H_2(X;\Z/2)\ar[d, "\phi"]
\\
\Spin(\partial X)\ar[rrr, "{w_2(X,\partial X;-)}"]\ar[d,"\iota^*"]
&&& H^2(X,\partial X;\Z/2)\ar[r, "\PD"]\ar[d,"\iota^*"]
& H_2(X;\Z/2)
\\
\Spin(\partial X_0)\ar[rrr, "{w_2(X_0,\partial X_0;-)}"]
&&& H^2(X_0,\partial X_0;\Z/2)\ar[r, "\PD"]
& H_2(X_0;\Z/2).\ar[u,"\iota_*"']
\end{tikzcd}
\]
As $\pi_\phi=\Id$, this shows that $\pi_{\phi_0}=\Id$ and hence $\phi_0\in \Aut_\partial^{\fix}(H_2(X_0),\lambda_{X_0})$. By~\cite[Theorem~0.7]{MR857447}, there exists $F_0\in\Homeo^+(X_0,\partial X_0)$ such that $(F_0)_*=\phi_0$. Extend $F_0$ by the identity to $F\in\Homeo^+(X,\partial X)$. Note that $F_*=\phi$, and hence $\Xi(\Delta_F)=\Id-\Delta_F\circ j_*=F_*=\phi$ as required (recall Remark~\ref{rem:recall}).
\end{proof}

Combining Theorem~\ref{thm:Saekialgebraic} with Theorem~\ref{theoremA}, and recalling Lemma~\ref{lem:spinkillsfix} in the spin case, we arrive at the following result.

For $\kappa\in \wedge^2{H_1(\partial X)^{\ast}}$,
denote by $F_\kappa\in\Homeo^+(X,\partial X)$ a homeomorphism arising as the effect of applying the realisation result Theorem~\ref{thm:realisevariational} to $(0,S(\kappa)) \in H^1(X,\partial X;\Z/2)\times\mathcal{V}(H_2(X),\lambda_X)$ in the spin case, respectively $S(\kappa) \in \mathcal{V}(H_2(X),\lambda_X)$ in the nonspin case. By Theorem~\ref{theoremA} we know that the class of  $F_\kappa \in \pi_0\Homeo^+(X,\partial X)$ is well-defined.

\begin{theorem}\label{thm:main2}
Let $X$ be a compact, simply connected, oriented, topological $4$-manifold with (possibly empty) boundary.

\begin{enumerate}[leftmargin=*]\setlength\itemsep{0em}
\item\label{item:spincase-main2} When $X$ is spin there is a short exact sequence
\[
0\to\wedge^2{H_1(\partial X)^{\ast}}\to\pi_0\Homeo^+(X,\partial X)\to H^1(X,\partial X;\Z/2)\times\Aut_\partial(H_2(X),\lambda_X)\to 0,
\]
where the first map is given by $\kappa\mapsto F_\kappa$ and the second map is given by $F\mapsto (\Theta(F),F_*)$.
\item \label{item:nonspincase-main2}
When $X$ is not spin there is a short exact sequence
\[
0\to \wedge^2{H_1(\partial X)^{\ast}}\to\pi_0\Homeo^+(X,\partial X)\to \Aut^{\fix}_\partial(H_2(X),\lambda_X)\to 0,
\]
where the first map is given by $\kappa\mapsto F_\kappa$ and the second map is given by $F\mapsto F_*$.
\end{enumerate}
\end{theorem}

\begin{corollary}\label{cor:variationreduction}
Let $X$ be a compact, simply connected, oriented, topological $4$-manifold, with (possibly empty) boundary, such that $H_1(\partial X;\Q)$ has rank at most 1.

\begin{enumerate}[leftmargin=*]\setlength\itemsep{0em}
\item\label{item:spincase-corollary} When $X$ is spin, the map $F\mapsto (\Theta(F),F_*)$ determines an isomorphism
\[
\pi_0\Homeo^+(X,\partial X)\xrightarrow{\cong} H^1(X,\partial X;\Z/2)\times\Aut_\partial(H_2(X),\lambda_X).
\]
\item \label{item:nonspincase-corollary}
When $X$ is not spin, the map $F\mapsto F_*$ determines an isomorphism
\[
\pi_0\Homeo^+(X,\partial X)\xrightarrow{\cong} \Aut^{\fix}_\partial(H_2(X),\lambda_X).
\]
\end{enumerate}
\end{corollary}

We can also now prove Corollaries~\ref{theoremA'} and~\ref{theoremTorelli} from the introduction.

\begin{proof}[Proof of Corollary~\ref{theoremA'}]
When $\partial X$ is a rational homology $S^3$ or $S^1\times S^2$, the free part of $H_1(\partial X)$ has rank at most 1. Thus the only skew-symmetric from $H_1(\partial X)\times H_1(\partial X)\to \Z$ is the 0 form, i.e.~$\wedge^2{H_1(\partial X)^{\ast}}=0$. Item \eqref{item:QHS31} is then clearly a special case of Theorem~\ref{thm:main2}.

For item \eqref{item:QHS32}, we first consider the general behaviour of the short exact sequence of Theorem~\ref{thm:Saekialgebraic} under algebraic stabilisation. Observe that for all $g\geq 0$ there is a commuting diagram where the rows are short exact sequences
\[
\begin{tikzcd}
\wedge^2H_1(\partial X)^*\ar[d, "\Id"]\ar[r, "S"]
&\mathcal{V}((H_2(X),\lambda_X)\oplus \mathcal{H}^{\oplus g})\ar[r, "\Xi"]\ar[d]
& \Aut_\partial^{\fix}((H_2(X),\lambda_X)\oplus \mathcal{H}^{\oplus g})\ar[d]
\\
\wedge^2H_1(\partial X)^*\ar[r, "S"]
& \mathcal{V}((H_2(X),\lambda_X)\oplus \mathcal{H}^{\oplus g+1})\ar[r, "\Xi"]
& \Aut_\partial^{\fix}((H_2(X),\lambda_X)\oplus \mathcal{H}^{\oplus g+1})
\end{tikzcd}
\]
where the unlabelled downwards maps are the stabilisation maps. Exact sequences are preserved under colimits in the category of groups. We obtain a short exact sequence
\[
\wedge^2H_1(\partial X)^*\to\underset{g\to\infty}{\colim}\, \mathcal{V}\big((H_2(X),\lambda_X)\oplus \mathcal{H}^{\oplus g}\big) \to \underset{g\to\infty}{\colim}\, \Aut_\partial^{\fix} \big((H_2(X),\lambda_X)\oplus \mathcal{H}^{\oplus g}\big).
\]
Now assuming the hypotheses of Corollary~\ref{theoremA'} \eqref{item:QHS32}, the sequence above shows that the stable $\Xi$ map is an isomorphism, which combines with Theorem~\ref{theoremB} to prove item \eqref{item:QHS32}.
\end{proof}

\begin{proof}[Proof of Corollary~\ref{theoremTorelli}]
Corollary~\ref{theoremTorelli}~\eqref{item:cor-torelli-2} is immediate from the nonspin case of Theorem~\ref{thm:main2}. The spin case of Theorem~\ref{thm:main2} yields a short exact sequence
\[
0\to\wedge^2{H_1(\partial X)^{\ast}}\to\Tor^0(X,\partial X)\to H^1(X,\partial X;\Z/2)\to 0.
\]
The exact sequence of Theorem~\ref{thm:Saekialgebraic} gives rise to a splitting $\Tor^0(X,\partial X) \to \wedge^2{H_1(\partial X)^{\ast}}$, completing the proof of Corollary~\ref{theoremTorelli}~\eqref{item:cor-torelli-1}.

For the proof of Corollary~\ref{theoremTorelli}~\eqref{item:cor-torelli-3}, assume that $X$ is spin. Recall that for all $g\geq 0$, there is a canonical identification $H^1(X\#^gS^2\times S^2,\partial X;\Z/2)\cong H^1(X,\partial X;\Z/2)$. For all $g\geq 0$ there is a commutative diagram
\[
\begin{tikzcd}
\Tor^\infty(X\#^gS^2\times S^2,\partial X)\ar[d, hookrightarrow, "\text{incl.}"]\ar[rrr, "{F\mapsto (\kappa_F, \Theta(F))}"]
&&& \wedge^2H_1(\partial X)^*\times H^1(X,\partial X;\Z/2)\ar[d,hookrightarrow, "S\times\Id"]
\\
\pi_0(\Diffeo^+(X\#^gS^2\times S^2,\partial X))\ar[rrr, "{F\mapsto (\Delta_F, \Theta(F))}"]\ar[d, "F\mapsto F_*"]
&&& \mathcal{V}\big((H_2(X),\lambda_X)\oplus \mathcal{H}^{\oplus g}\big)\times H^1(X,\partial X;\Z/2)\ar[d, twoheadrightarrow,"\Xi\circ \pr_1"]
\\
\Aut_\partial\big((H_2(X),\lambda_X)\oplus \mathcal{H}^{\oplus g}\big)\ar[rrr, "\Id"]
&&& \Aut_\partial\big((H_2(X),\lambda_X)\oplus \mathcal{H}^{\oplus g}\big)
\end{tikzcd}
\]
The left-hand column is an exact sequence by definition of the Torelli group, and the right-hand column is a short exact sequence by Theorem~\ref{thm:Saekialgebraic}. The diagram is natural under stabilisation, so we may take the colimit of the diagram. Exactness is preserved under colimits, and in the colimit, the central horizontal map becomes an isomorphism by Theorem~\ref{theoremB}. We deduce from this that the bottom-left vertical map becomes surjective in the colimit. We may conclude that in the colimit the diagram is an isomorphism of short exact sequences, in particular the upper horizontal map is an isomorphism in the colimit, as claimed.

This argument is easily modified to prove Corollary~\ref{theoremTorelli}~\eqref{item:cor-torelli-4}, the nonspin case, by removing $H^1(X,\partial X;\Z/2)$ from the diagram and replacing $\Aut_\partial$ with $\Aut_\partial^{\fix}$.
\end{proof}

\section{Realising the invariants \texorpdfstring{$\Theta$}{Theta} and \texorpdfstring{$S(\kappa)$}{S(kappa)} smoothly}\label{subsec:which3manifolds}

We consider the problem of realising the obstruction $\Theta$ of Definition~\ref{def:spinobstruction} in the smooth category, and without stabilising. The main result we prove is that this can be done in the presence of boundary components of Heegaard genus at most one. We then consider the problem of constructing  $F\in\Diffeo^+(X,\partial X)$ such that $F_*=\Id_{H_2(X)}$ but with $\Delta_F\neq 0$, in other words, building diffeomorphisms $F$ such that $\Delta_F=S(\kappa)$ for some $\kappa\in\wedge^2{H_1(\partial X)^{\ast}}$.
 Neither of these smooth realisation questions are fully resolved.

On the other hand, the results of this section can also be viewed as computing the action of $\pi_1\Homeo^+(\partial X)$ on $\pi_0\Homeo^+(X,\partial X)$ in the exact sequence~\eqref{eq:exactsequencehtpy} from Section~\ref{subsection:relaxing-bdy-condition}.

\subsection{Modifying $\Theta$ with generalised Dehn twists}

Our strategy is to modify a given $\Theta$ by using a diffeomorphism supported in a collar on a boundary component of the $4$-manifold, or more generally on an embedded $Y\times[0,1]$, where $Y$ is a closed $3$-manifold. For this purpose we briefly consider the behaviour of $\Theta$ when two $4$-manifolds are glued along one or more boundary components.

Let $(X,\mathfrak{s})$ be a compact, oriented, spin, topological $4$-manifold. Suppose that $X$ decomposes as $X=X_1\cup_Y X_2$, where $Y$ is a (possibly disconnected) $3$-manifold. Write $\mathfrak{s}_i$ for the restriction of $\mathfrak{s}$ to $X_i$. We now consider the long exact cohomology sequence with $\Z/2$-coefficients associated to the map of pairs $(X,\partial X)\to (X,X_2\sqcup (\partial X_1\sm Y))$.
\[
\begin{tikzcd}
H^0(X_2,\partial X_2\sm Y)\ar[r]
& H^1(X,X_2\sqcup (\partial X_1\sm Y))\ar[r]\ar[d, "\operatorname{exc}", "\cong"']
& H^1(X,\partial X)\ar[r, "\sigma"]
& H^1(X_2,\partial X_2\sm Y)
\\
& H^1(X_1,\partial X_1)\ar[ru, "\psi"']
&
&
\end{tikzcd}
\]
The downwards vertical map is the excision isomorphism and the map $\psi$ is defined to make the diagram commute.

\begin{lemma}\label{lem:Thetaadditivity}
With notation as above, suppose that $F\colon X\xrightarrow{\cong} X$ is a homeomorphism that fixes $\partial X$ pointwise and restricts to the identity map on $X_2$. Then $\psi(\Theta(F|_{X_1}, \mathfrak{s}_1))=\Theta(F,\mathfrak{s})$.
\end{lemma}

\begin{proof} Choose a vector space splitting $H^1(X,\partial X;\Z/2)\cong\psi(H^1(X_1,\partial X_1;\Z/2))\oplus V$. Write $\Theta_V$ for the projection of $\Theta(F,\mathfrak{s})$ to $V$. Our objective is to show that $\Theta_V$ is trivial. By exactness of the sequence above, $\sigma$ is injective when restricted to $V$, so it is sufficient to show that $\sigma(\Theta(F,\mathfrak{s}))=\sigma(\Theta_V)$ is trivial. Because the map $\sigma$ is induced by inclusion, the class $\sigma(\Theta_V)$ represents the difference in the rel.~boundary spin structure $(\mathfrak{s}_2,\mathfrak{s}_2|_{\partial X_2\sm Y})$ and $F^*(\mathfrak{s}_2,\mathfrak{s}_2|_{\partial X_2\sm Y})$. But by assumption $F|_{X_2}$ is the identity map and this class $\sigma(\Theta_V)$  vanishes.
\end{proof}

\begin{definition}\label{def:GDT} Let $Y$ be a closed, compact, oriented $3$-manifold. Fix a spin structure on $Y$ and write $\mathfrak{s}$ for the resulting product spin structure on $Y\times I$. Let $\phi_t\in \pi_1(\Diffeo^+(Y))$ be a loop based at the identity and constant near the basepoint. Define a self-diffeomorphism
\begin{align*}
  \Phi \colon Y \times I & \xrightarrow{\cong} Y \times I \\
  (x,t) &\mapsto (\phi_t(x),t).
\end{align*}
If $\Theta(\Phi,\mathfrak{s})=1\in\Z/2=H^1(Y\times I,\partial;\Z/2)$ then we say $\Phi$ is a \emph{generalised Dehn twist} with respect to the spin structure~$\mathfrak{s}$ on~$Y$.
\end{definition}

\begin{proposition}\label{prop:switcheroo}
Let $(X,\mathfrak{s})$ be a compact, simply connected, oriented, spin, smooth $4$-manifold with nonempty connected boundary components $\partial X=Y_1\cup\dots\cup Y_r$. Suppose for $i=1,\dots,r-1$ that $Y_i$ admits a generalised Dehn twist with respect to the restriction of $\mathfrak{s}$ to $Y_i$.  Let $F\in\Diffeo^+(X,\partial X)$. Then, given $x\in H^1(X,\partial X;\Z/2)$, there exists an orientation preserving diffeomorphism $F'\colon X\to X$, that fixes the boundary pointwise, agrees with $F$ outside a collar of the boundary, and is such that $\Theta(F')=x$.
\end{proposition}

\begin{proof}
In this proof, $\Z/2$ coefficients are understood throughout. When $r=1$, the proposition is vacuously true because then $H^1(X,\partial X)=0$.

Now assume $r>1$. Choose properly embedded arcs $\gamma_1,\dots,\gamma_{r-1}$ such that for each $i$, $\gamma_i$ connects $Y_r$ to $Y_i$. Note that these arcs determine a preferred basis for $H_1(X,\partial X)$. Consider a fixed $i\neq r$. We claim there exists a diffeomorphism $G\colon X\to X$ that agrees with $F$ outside of the boundary collar of $Y_i$, and such that $\Theta(G)$ agrees with $\Theta(F)$ except on the $i^{th}$ basis curve, where they differ. To prove the claim, first attach an external collar to form $X':=X\cup (Y_i\times[0,1])$. Let $\Phi\colon Y_i\times[0,1]\to Y_i\times[0,1]$ be a generalised Dehn twist with respect to $\mathfrak{s}|_{Y_i}$. We extend $\mathfrak{s}$ to a spin structure $\mathfrak{s}'$ on $X'$ by using the product spin structure on the collar. Define a self-diffeomorphism $\Phi'\colon X'\to X'$ by extending $\Phi$ using the identity on $X$.  By Lemma~\ref{lem:Thetaadditivity}, we have that $\Theta(\Phi')$ is the image of $\Theta(\Phi, \mathfrak{s}'|_{Y_i\times[0,1]})$ under the map  $\psi\colon H^1(Y_i\times[0,1],Y_i\times\{0,1\})\to H^1(X',\partial X')$. As $r>1$, we have $\partial X\sm Y\neq\emptyset$, so that $H^0(X,\partial X\sm Y)=0$ and $\psi$ is injective. Hence $\psi(\Theta(\Phi, \mathfrak{s}'|_{Y_i\times[0,1]}))=\Theta(\Phi')$ is nontrivial.
Now choose the basis for $H_1(Y_i\times[0,1],Y_i\times\{0,1\})=\Z/2$ given by $y\times [0,1]$ where $y\in \gamma_i$ is the endpoint on $Y_i$. A basis for $H_1(X',\partial X')$ is given by $\gamma_1,\dots,\gamma_{i-1},\gamma_i\cup(y\times [0,1]), \gamma_{i+1},\dots,\gamma_{r-1}$. As the class $\Theta(\Phi, \mathfrak{s}'|_{Y_i\times[0,1]})$ is nontrivial, it is necessarily the $\Z/2$-dual $[y\times[0,1]]^*$. Thus
\[
\Theta(\Phi')=\psi(\Theta(\Phi, \mathfrak{s}'|_{Y_i\times[0,1]}))=\psi([y\times[0,1]]^*)=[\gamma_i\cup(y\times [0,1])]^*.
\]
Finally, as $\Theta$ is additive under composition of maps (see Lemma~\ref{lem:spinbusiness}), the map $G:=\Phi'\circ F$ has the desired properties. (Implicitly, we are using that adding an external collar produces a diffeomorphic manifold.)

With this claim proved, we now have a procedure for altering any individual entry of $\Theta(F)$, by changing the map $F$ in a collar of the boundary. For each $i=1,\dots,r-1$, perform this type of modification on $F$ as necessary, so that the resulting map $F'$ has $\Theta(F')=x$.
\end{proof}

\subsection{Which 3-manifolds admit generalised Dehn twists?}\label{sec:which}

In this subsection we consider the question of which closed, oriented $3$-manifolds admit generalised Dehn twists, and with respect to which of their spin structures. We begin our analysis of Dehn-twistability with two special cases that are easier to understand: $Y= S^3$ and $Y=S^1 \times \Sigma_g$. Then we consider lens spaces, before finally saying a brief word about Seifert fibred spaces in general.

We introduce some notation for this subsection. Given a closed, compact, oriented $3$-manifold $Y$, and given $\phi_t\in \pi_1(\Diffeo^+(Y))$ a loop based at the identity, we denote let $\Phi$ be as in Definition~\ref{def:GDT}, a self-diffeomorphism of $Y\times[0,1]$ that fixes the boundary pointwise.
Given a properly embedded arc $\gamma\subseteq Y\times[0,1]$, with endpoints on the different boundary components, write
\[
\delta:=\gamma\times\{1\}\cup \partial\gamma\times[0,1]\cup \gamma\times\{0\}\subseteq \partial((Y\times [0,1])\times[0,1]).
\]
Given a spin structure on $Y$, let $\mathfrak{s}$ be the product spin structure on $Y\times[0,1]$ and let $\mathfrak{t}$ be the spin structure on $D:=\partial((Y\times [0,1])\times[0,1])$ given by $\mathfrak{t}(x,t)=\mathfrak{s}\circ \Phi$ when $t=1$ and $\mathfrak{t}=\mathfrak{s}$ otherwise. As detailed in the proof of Lemma~\ref{lem:collection}, this determines a framing for the normal $3$-plane bundle $\nu(\delta\subseteq D)$.

In the next three examples, $\Theta(\Phi,\mathfrak{s}) \in H^1(Y \times [0,1],Y \times \{0,1\};\Z/2) \cong \Z/2$. To decide whether this element is nontrivial, it suffices to evaluate it on a curve $\gamma$ with endpoints on $Y \times \{0\}$ and $Y \times \{1\}$, representing a generator of $H_1(Y \times [0,1],Y \times \{0,1\};\Z/2)$. We will do this using Lemma~\ref{lem:collection}, more precisely the equivalence  \eqref{item:collection1}$\Leftrightarrow$\eqref{item:collection5}. That is, we will decide whether the framing on $\nu(\delta\subseteq D)$ corresponding to $\mathfrak{t}$ extends over $\nu_{\gamma \times [0,1]}$.

\begin{example}\label{ex:S3}
Elaborating on Example~\ref{example:dehn-twist-intro}, we prove that the (ordinary) Dehn twist is a generalised Dehn twist with respect to the unique spin structure on $S^3$. Fix a point $p\in S^2\subseteq \R^3$ and let $\phi'_t\in \SO(3)$, for $t\in[0,1]$, be a loop based at the identity and rotating $\R^3$ one full turn around an axis through the origin and $p$.
Under the inclusion of the first factor $\R^3\subseteq \R^3\times\R=\R^4$, and restricting to $S^3 \subseteq \R^4$, there is then induced a generating loop $\phi_t\in\pi_1(\Diffeo^+(S^3))\cong\Z/2$. This determines the self-diffeomorphism $\Phi\colon S^3\times[0,1]\to S^3\times[0,1]$ as in Definition~\ref{def:GDT}.

As $p$ is fixed by $\phi_t$ for all $t$, the tubular neighbourhood of $\gamma:=\{p\}\times[0,1]\subseteq S^3\times[0,1]$ is a product $\nu(\gamma)=\nu(\{p\}\subseteq S^3)\times[0,1]$. The loop $\phi_t$ preserves the tubular neighbourhood $\nu(\{p\}\subseteq S^3)$ set-wise; it fixes one normal coordinate of $\nu(\{p\}\subseteq S^3)$ for all $t$, and rotates the normal 2-plane perpendicular to this fixed coordinate by one full turn. The spin structure $\mathfrak{s}$ and $\Phi$ determine a framing $\mathfrak{t}$ on $\nu(\delta\subseteq D)$. As $p$ was fixed, over $\gamma\times\{0\}$ this is a product framing $\nu(p\subseteq S^3)\times [0,1]$ . However the framing over $\delta$ is not null-bordant because over $\gamma\times\{1\}$, the map $\Phi$ adds a full twist compared to the bundle over $\gamma\times\{0\}$, whereas the framing on the normal bundle of $\gamma \times [0,1]$ restricts to the same framing on both $\gamma\times\{0\}$ and $\gamma\times\{1\}$. It then follows from Lemma~\ref{lem:collection} that $\Theta(\Phi,\mathfrak{s})=1$.
\end{example}

\begin{example}\label{ex:S1xS2}
Consider the loop $\phi_t\in\pi_1(\Diffeo^+(S^1\times S^2))$ that fixes the $S^1$-coordinate, and rotates the $S^2$ one full turn around an axis in $\R^3$. We claim that the corresponding $\Phi$ (as in Definition~\ref{def:GDT}) is a generalised Dehn twist with respect to both of the spin structures on $S^1\times S^2$. To see this, let $p$ be a fixed point of the $S^2$ rotation. Similarly to Example~\ref{ex:S3}, the normal bundle of $\gamma:=\{p\}\times[0,1]\subseteq (S^1\times S^2)\times[0,1]$ is preserved set-wise by $\Phi$, and the normal 2-plane corresponding to the $S^2$ factor is turned by one full rotation, whereas the normal direction corresponding to the $S^1$ is fixed throughout. The spin structure $\mathfrak{s}$ and the diffeomorphism $\Phi$ determine a framing $\mathfrak{t}$ on $\nu(\delta\subseteq D)$. As $p$ was fixed, this is a product framing $\nu(p\subseteq S^1\times S^2)\times [0,1]$ over $\gamma\times\{0\} \subseteq S^1 \times S^2 \times [0,1] \times \{0\}$. However the framing over $\delta$ is not null-bordant because over $\gamma\times\{1\} \subseteq S^1 \times S^2 \times [0,1] \times \{1\}$, the map $\Phi$ adds a full twist to the framing compared to the framing  over $\gamma\times\{0\}$ coming from $\mathfrak{s}$. By Lemma~\ref{lem:collection}, $\Theta(\Phi,\mathfrak{s})=1$.
\end{example}

\begin{example}\label{ex:differentstructures}
In contrast to Example~\ref{ex:S1xS2}, it may generally be the case that a given loop of diffeomorphisms determines a generalised Dehn twist with respect to one spin structure but not with respect to another. For example, let $\phi_t\in\pi_1(\Diffeo^+(S^1\times S^2))$ be the loop that fixes the $S^2$ factor and is a full rotation on the $S^1$ factor. Write $\mathfrak{s}_0$ and $\mathfrak{s}_{1}$ for the spin structures on $(S^1\times S^2)\times[0,1]$ that arise from extending respectively the bounding and Lie spin structures on the $S^1$ factor to spin structures on $S^1 \times S^2$, and then using product spin structures with respect to the $[0,1]$ coordinate. Here, in order to specify the bounding and Lie spin structures, we use the given product structure of $S^1 \times S^2$ to identify $TS^1 \oplus\varepsilon^2$, which has a preferred Lie framing, with $T(S^1 \times S^2)|_{S^1 \times \{p\}}$.  The bounding spin structure differs from this by one full rotation.  We can assume that the loop $S^1 \times \{p\}$ equals the 1-skeleton in a cell decomposition of $S^1 \times S^2$.    Fix a point $(q,p)\in S^1\times S^2$. We may assume the framing of $T_{(q,p)}(S^1\times S^2)\cong T_qS^1\oplus T_pS^2$ induced by both spin structures is of the form $\langle e_1,e_2,e_3\rangle$, where $e_1\in T_qS^1$ and $e_2,e_3\in T_pS^2$. Since $\phi_t$ rotates the $S^1$ factor and fixes the $S^2$ factor, the induced framing on $\phi_t(q,p)$ preserves this decomposition and fixes the vectors $e_2, e_3$. Thus, the induced framing on $S^1 \times \{p\}$ is identified with the non-bounding (Lie) spin structure on $TS^1 \oplus \varepsilon^2$, no matter the initial framing.  That is, $\mathfrak{s}_0 \circ \Phi = \mathfrak{s}_1 \circ \Phi = \mathfrak{s}_1$.

We use the arc $\gamma(t)=(\phi_t(q,p),t)$ to compute the invariants $\Theta(\Phi,\mathfrak{s}_i)$ for $i=0,1$.  For the spin structure $\mathfrak{s}_0$, we have $\mathfrak{s}_0$  on $\gamma \times \{0\}$ and $\mathfrak{s}_1$ on $\gamma \times \{1\}$, so we see that $\Theta(\Phi,\mathfrak{s}_0)=1$. On the other hand for the spin structure $\mathfrak{s_1}$ we have $\mathfrak{s}_1$  on both $\gamma \times \{0\}$ and  $\gamma \times \{1\}$, so  $\Theta(\Phi,\mathfrak{s}_1)=0$.

We can similarly analyse $Y=S^1\times S^1\times S^1$. There are $|H^1(Y;\Z/2)|=8$ spin structures on $Y$ characterised by whether or not their restrictions to the three basis circles are bounding. Note that $Y$ is a Lie group, and exponentiating multiplication in each of the coordinate directions gives rise to a Lie framing along each standard basis element of $H_1(Y;\Z/2)$; the Lie framing is the non-bounding framing.
If $\mathfrak{s}$ is bounding for a basis circle $C$, then let $\Phi$ be induced by the loop of diffeomorphisms that rotates $C$ and fixes the other two. As in the example above, this gives $\Theta(\Phi,\mathfrak{s})=1$. If $\mathfrak{t}$ is the spin structure that is the Lie framing on all three basis circles, all of the obvious circle actions on $Y$ result in $\Theta(\Phi,\mathfrak{t})=0$.

Now let $Y=S^1\times \Sigma_g$, where $g>1$ and let $\Phi$ be induced,  as in Definition~\ref{def:GDT}, by the loop that fixes the $\Sigma_g$ factor and rotates the $S^1$ factor by a full rotation.
As with $S^1 \times S^2$, we may use the product structure of $S^1 \times \Sigma_g$ to identify $TS^1 \oplus\varepsilon^2$ with $T(S^1 \times \Sigma_g)|_{S^1 \times \{p\}}$ along a loop $S^1 \times \{p\}$, and thence we may speak of the Lie framing and the bounding framing of the stable tangent bundle restricted to this loop.
Exactly half of the $|H^1(Y;\Z/2)|=2^{2g+1}$ spin structures will be such that their restriction to $S^1\times\{p\}$ is the bounding framing, and if $\mathfrak{s}_0$ is such a spin structure then $\Theta(\Phi,\mathfrak{s}_0)=1$ by the same argument as above. The other spin structures restrict to the Lie framing on $S^1\times\{p\}$ and if $\mathfrak{s}_1$ is such a spin structure then $\Theta(\Phi,\mathfrak{s}_1)=0$.
\end{example}

Next, we generalise the method for computing $\Theta$ that was used in the previous example.
  Let $Y$ be a closed, oriented 3-manifold and let $\phi_t \in \pi_1(\Diffeo^+(Y))$ be a loop based at the identity. This determines a self-diffeomorphism $\Phi \colon Y \times [0,1] \to Y \times [0,1]$.  Suppose we are given a spin structure~$\mathfrak{s}$ on~$Y$, and fix a framing for the tangent bundle $TY$ which is compatible with the given spin structure.  We can compute $\Theta(\Phi,\mathfrak{s})$ using any path $\gamma$ between the two boundary components. We shall compute $\Theta$ by choosing a point $p\in Y$ and using the particular choice of path $\gamma$ given by $t \mapsto (\phi_t(p),t)$ in $Y \times [0,1]$.
Write~$\langle e_1,e_2,e_3 \rangle$ for the standard basis of $\R^3$, identified with $T_p Y$ using the framing of $TY$.
For each $t\in[0,1]$, define an element of $\GL^+(3,\R)$ by comparing $\langle D\phi_t(p)(e_i)\mid i=1,2,3\rangle \in T_{\phi_t(p)}Y$ with the basis of $T_{\phi_t(p)}Y$ given by the framing of $TY$. This determines a based map $S^1 \to \GL^+(3,\R)$, and whence an element $x$ of $\pi_1(\GL^+(3,\R)) \cong \Z/2$.
Comparing with Lemma~\ref{lem:collection}, the stable framing of $\nu_{\delta}$ (recall that $\delta$ is defined in terms of $\gamma$) is null-bordant if and only if  $x=0 \in \Z/2$ is trivial. Then by Lemma~\ref{lem:collection},  $\Theta(\Phi,\mathfrak{s}) = x$.

The main result of the subsection is the following proposition, which leverages the fact that lens spaces admit multiple Seifert fibrations to guarantee that for every lens space there is a Seifert fibration giving rise to nontrivial $\Theta$.

\begin{proposition}\label{prop:which-3-manifolds}
Every closed, orientable 3-manifold $Y$ of Heegaard genus at most one admits a generalised Dehn twist for every spin structure on  $Y \times [0,1]$.
\end{proposition}

\begin{proof}
  Since Examples~\ref{ex:S3} and \ref{ex:S1xS2} already dealt with the cases of $S^3$ and $S^1 \times S^2$ respectively, it remains to consider lens spaces $L(p,q)$.
  The aim is to fix a lens space $L(p,q)$ and a spin structure on the lens space, and then to find a Seifert fibration such that one full rotation around each regular $S^1$ fibre induces a generalised Dehn twist.

  Suppose without loss of generality that $p>0$, $0 < q < p$ and $(p,q)=1$.  By B\'{e}zout's Lemma, there exist integers $r$ and $s$ such that
  $qs-pr=-1$.
  In order to fix notation, we describe a genus one Heegaard splitting for $L(p,q)$ explicitly.
  Let $S_1$, $S_2$ be copies of the solid torus $S^1 \times D^2$ and write $T_i := \partial S_i$ for $i=1,2$.  Write $\mu_i \subseteq T_i$ for $\{1\} \times S^1$ and $\lambda_i \subseteq T_i$ for $S^1 \times \{1\}$.  Fixing an orientation of $S^1$, we will conflate this notation for $\mu_i$ and $\lambda_i$ with the associated homology classes in $H_1(T_i)$. Also note that $\langle \mu_i,\lambda_i\rangle$ is a basis for $H_1(T_i) \cong \Z^2$.  Let $\psi \colon T_1 \xrightarrow{\cong} T_2$  be a diffeomorphism, determined up to isotopy by prescribing that, with respect to the bases just defined, $\psi_* \colon H_1(T_1) \to H_1(T_2)$ is represented by
  $\bsm
    q & r \\ p & s
  \esm$.
  Then $L(p,q) = S_1 \cup_{\psi} S_2$.
  In the Mayer-Vietoris sequence for this decomposition, noting that $H_1(S_i) \cong \Z$, we have the map
  $\bsm
    0 & 1 \\ p & s
  \esm \colon \Z \oplus \Z \cong H_1(T_1) \to H_1(S_1) \oplus H_1(S_2) \cong \Z \oplus \Z.$
It follows that $H_1(L(p,q)) \cong \Z/p$, generated by the core of the solid torus $S_2$.  The core of $S_1$ is identified with $s$ times the core of $S_2$. In what follows we will use the core of $S_2$ as the preferred generator of $\Z/p$ (although as $s$ is coprime to $p$, the core of $S_1$ is also a generator).

Fix a spin structure on $L(p,q)$.
Choose a Seifert fibration of $S_2$ with multiplicity one and such that it restricts to a fibration of $T_2$ by embedded curves that represent $(0,1)$ in $H_1(T_2)$, with respect to the basis $\langle\mu_2,\lambda_2 \rangle$. We obtain a loop $\phi_t\in \pi_1(\Diffeo^+(S_2))$ by rotating each regular fibre through one full rotation. Pick a point $p\in S_2$ on the core, and compute $\Theta$ by following a framing of $T_pL(p,q)$ around the loop $\phi_t(p)$ in the way described earlier in this section. If $\Theta$ is trivial, we change the choice of Seifert fibration on $S_2$ to one that has multiplicity 1 and restricts to a fibration of $T_2$ by embedded curves that represent $(1,1)$ in $H_1(T_2)$. The additional meridional turn as the frame moves around the core will force $\Theta$ to be nontrivial.

(Whether we had to use the $(0,1)$ or the $(1,1)$ Seifert fibration depended on the initial given spin structure. In our applications, the spin structure on $L(p,q)$ will be the one induced by the unique  spin structure on the simply connected 4-manifold $X$, so we cannot control it. We also remark that for $p$ odd there is a unique spin structure on $L(p,q)$ because $H^1(L(p,q);\Z/2)=0$. We can determine this spin structure explicitly, but since we do not need this information we omit the details.)

It remains to check that our choice of Seifert fibration on $S_2$ extends over $S_1$, therefore inducing an $S^1$ action on all of $L(p,q)$ as required.
We have that $\psi^{-1} \colon T_2 \to T_1$ induces
$\bsm   q & r \\ p & s \esm^{-1} = \bsm  -s & r \\ p & -q \esm \colon \Z \oplus \Z \cong H_1(T_2) \to H_1(T_1) \cong \Z \oplus \Z.$
A $(0,1)$-curve in $T_2$ is identified with an $(r,-q)$-curve in $T_1$, while a $(1,1)$-curve in $T_2$ is identified with an $(r-s,p-q)$-curve in $T_1$.  In both cases, the second coordinate is nonzero, by our assumption that $0 < q < p$. The induced Seifert fibrations of $T_1$ therefore extend to standard Seifert fibrations of $S_1$, with singular fibres at the core of $S_1$ of multiplicities $q$ and $p-q$ respectively.  Thus in both cases we obtain a Seifert fibration of $L(p,q)$ with exactly one singular fibre, such that, with respect to the given spin structure, rotating around the Seifert fibres gives rise to a generalised Dehn twist, as desired.
\end{proof}

We are now in a position to prove Corollary~\ref{corollaryG} from the introduction, which we now recall for the reader's convenience.

\begin{corollaryF}
Let $X$ be a compact, simply connected, oriented, topological $4$-manifold.
Suppose that every connected component of $\partial X$ has Heegaard genus at most 1, and at most one of the connected components is $S^1\times S^2$.
Then there is an exact sequence of groups
\[
0\to \Aut^{\fix}_\partial(H_2(X),\lambda_X)\to \pi_0\Homeo^+(X)\to\pi_0\Homeo^+(\partial X).
\]
\end{corollaryF}

\begin{proof}
Our assumptions imply $H_1(\partial X;\Q)$ has rank at most 1. By Corollary~\ref{cor:variationreduction} this implies $\pi_0\Homeo^+(X,\partial X)$ is isomorphic to $\Aut^{\fix}_\partial(H_2(X),\lambda_X)$ when $X$ is not spin, and to $\Aut^{\fix}_\partial(H_2(X),\lambda_X)\times H^1(X,\partial X;\Z/2)$ when $X$ is spin.
When $X$ is not spin, the corollary follows, by considering the exact sequence \eqref{eq:exactsequencehtpy} and the fact that the image of $\pi_1 \Homeo^+(\partial X)$ in this sequence necessarily produces the trivial automorphism of $H_2(X)$.

Now suppose $X$ is spin. Write $\mathfrak{s}$ for the spin structure. As each boundary component has Heegaard genus at most 1, we may apply Proposition~\ref{prop:which-3-manifolds} to see that each boundary component has a loop of diffeomorphisms producing a generalised Dehn twist with respect to the restriction of $\mathfrak{s}$. Now Proposition~\ref{prop:switcheroo} shows that any given $\Theta$ can be realised under the action of $\pi_1$ in the sequence \eqref{eq:exactsequencehtpy}. Hence the cokernel of the map
$\pi_1\Homeo^+(\partial X)\to\pi_0\Homeo^+(X,\partial X)$
%$\pi_0\Homeo^+(X,\partial X) \to \pi_0\Homeo^+(X)$
 is isomorphic to $\Aut^{\fix}_\partial(H_2(X),\lambda_X)$. The claimed result now follows from the exactness of the sequence~\eqref{eq:exactsequencehtpy}.
\end{proof}

We also considered the question of which Seifert fibred spaces, in general,  admit a generalised Dehn twist. As we now show, every orientable Seifert fibred space $Y$ admits a homotopically nontrivial loop of self-diffeomorphisms, so there is the potential for a generalised Dehn twist to exist. We already discussed that $\pi_1(\Diffeo^+(S^3))\cong\Z/2$, so we focus on the case $Y\neq S^3$.

\begin{proposition}\label{prop:niftycalc}
Let $Y\neq S^3$ be a closed, orientable Seifert fibred $3$-manifold. Then the rotation around the Seifert fibres determines a nontrivial loop $\phi_t\in\pi_1(\Diffeo^+(Y))$.
\end{proposition}

\begin{proof}
For each $p\in Y$ there is a map $\pi_1(\Diffeo^+(Y))\to \pi_1(Y,p)$ that sends $\phi_t$ to the loop $\phi_t(p)\subseteq Y$. It is sufficient to show that for some $p\in Y$ the map is nontrivial. Assume the converse, for a contradiction. This implies that all fibres of the Seifert fibration are homotopically trivial. A presentation for $\pi_1(Y)$ is computed in~\cite[Theorem~6.1]{MR741334}. In the language of that proof, $\alpha_j$ is the multiplicity of the $j^\text{th}$ singular fibre, and there are integers $\alpha_j', \beta_j, \beta_j'$ coming from a specific description of the neighbourhoods of the singular fibres (see~\cite[\textsection1]{MR741334}), such that $\alpha_j\beta_j'-\alpha_j'\beta_j=1$. Following the language of that theorem, and its proof, we write~$h$ for the class of a regular fibre and $q_j$ for the class in $\pi_1(Y,p)$ that comes of a loop around the $j^{\text{th}}$ orbifold point on the base orbifold, sent into $Y$ via a section away from the orbifold points. This makes ${q_j}^{\alpha_j'}h^{\beta'_j}$ the class of a singular fibre. Our assumption that $h$ and ${q_j}^{\alpha_j'}h^{\beta'_j}$ are trivial, implies ${q_j}^{\alpha_j'}=1$. Already in the presentation we have ${q_j}^{\alpha_j}h^{\beta_j}=1$, so that also ${q_j}^{\alpha_j}=1$. For all $j$, we thus have
\[
q_j=q_j^{\alpha_j\beta_j'-\alpha_j'\beta_j}=(q_j^{\alpha_j})^{\beta'_j}(q_j^{\alpha'_j})^{-\beta_j}=1.
\]
Observe this means there are no nontrivial orbifold points in the base orbifold, in other words the base orbifold is a manifold. Setting $h=1$ and $q_j=1$, for all $j$, in the presentation of~\cite[Theorem~6.1]{MR741334}, we have that $\pi_1(Y)\cong\pi_1(\Sigma)$, where $\Sigma$ is the base surface of the Seifert fibration.

The only closed oriented $3$-manifolds $Y$ with surface fundamental group are $Y\cong S^3$ and $Y\cong\RP^3$.
To see this, note first that $Y$ is irreducible as a surface group is not a non-trivial free product nor isomorphic to $\Z$. Suppose $\pi_1(Y)$ is infinite. Then $\pi_2(Y)=0$, by the Sphere Theorem, and hence $Y\simeq K(\pi_1(\Sigma),1)$. But $H_3(Y;\Z) =\Z$ whereas $H_3(\pi_1(\Sigma);\Z) =0$ for every closed surface $\Sigma$ with infinite fundamental group, which is a contradicition. Thus we are left with the case that $\pi_1(Y)$ is finite. The only finite surface groups are the trivial group and $\Z/2$, implying respectively either $Y\cong S^3$ or $Y\cong\RP^3$.

By assumption, it is not the case that $Y\cong S^3$, so suppose $Y \cong \RP^3$. In this case, the projection of the Seifert fibration $f \colon \RP^3 \to \RP^2$ induces an isomorphism $\pi_1(\RP^3) \to \pi_1(\RP^2)$, and hence also an isomorphism $f^* \colon H^1(\RP^2;\Z/2) \to H^1(\RP^3;\Z/2)$.  But for $x \in H^1(\RP^2;\Z/2)$ a generator we have $0 = f^*(0) = f^*(x^3) = f(x)^3 \neq 0 \in H^3(\RP^3;\Z/2)$, a contradiction.
\end{proof}

%Regarding the $\Theta$ invariant for a general Seifert fibred space, with $\Phi$ the self-diffeomorphism induced by rotating the fibres, we were able to show the following.
%
%\begin{claim}
%Let $Y$ be an oriented Seifert fibred space. Write the Euler number of the Seifert fibred structure as $a/b$, where $b$ is the least common multiple of the multiplicities of the exceptional fibres. If $a$ and $b$ are odd, then $\Theta(\Phi,\mathfrak{s})=0$ with respect to any spin structure $\mathfrak{s}$ on $Y$.
%\end{claim}
%
%We judge our proof of this result to be too lengthy to include. %, considering that the result is negative, and
%The result does not actually preclude any given value of $\Theta$ from being realised by a diffeomorphism (it merely shows this particular diffeomorphism will not be useful).
%
%However, it does show that if $X$ is spin and every boundary component of $X$ is homeomorphic to such a $Y$, then the action of $\pi_1\Homeo^+(\partial X)$ on $\pi_0\Homeo^+(X,\partial X) \cong H^1(X,\partial X;\Z/2)\times\mathcal{V}(H_2(X),\lambda_X)$ does not affect the $H^1(X,\partial X;\Z/2)$ component.

It remains an interesting outstanding problem to fully determine which Seifert fibred 3-manifolds admit generalised Dehn twists. This would constitute part of the computation of the action of $\pi_1\Homeo^+(\partial X)$ on $\pi_0\Homeo^+(X,\partial X)$.
We have made partial progress on this problem, but judge our results too incomplete to include.

\subsection{Realising the kernel of $\Xi$ smoothly}
Recall the map $S \colon  \wedge^2{H_1(\partial X)^{\ast}} \to \mathcal{V}(H_2(X),\lambda_X)$ from Theorem~\ref{thm:Saekialgebraic}. We investigate the possibility of realising elements $S(\kappa)\in\mathcal{V}(H_2(X),\lambda_X)$ by using smooth self-maps of boundary collars, extended by the identity to the rest of the $4$-manifold.  We prove the following partial result on this question.
Suppose that $Y \subseteq \partial X$ is a Seifert fibred connected component of the boundary of a simply connected, compact, smooth 4-manifold $X$, and suppose that the Seifert fibred structure on $Y$ has oriented base surface of genus $g \geq 1$.  Let $\{\alpha_i,\beta_i\}$ be a symplectic basis of curves on the base surface. By \cite[Corollary~6.2]{MR741334} these curves are all homologically essential and infinite order in $H_1(Y)$. Define
  \[\omega := \sum_{i=1}^g \alpha_i^* \wedge \beta_i^* \in \wedge^2{H_1(Y)^{\ast}} \subseteq \wedge^2{H_1(\partial X)^{\ast}}.\]
As above, rotating the fibres determines a loop of diffeomorphisms of $Y$, which gives rise to a diffeomorphism $F \colon X \to X$ supported in a collar of $Y$.

\begin{proposition}\label{prop:grabthetori}
The Poincar\'{e} variation of the diffeomorphism $F$ satisfies \[\Delta_F = S(\omega) \in \mathcal{V}(H_2(X),\lambda_X).\]
\end{proposition}

\begin{proof}
Represent the curves $\alpha_i,\beta_i$, for $i=1,\dots,g$, by simple closed curves on the base surface disjoint from exceptional fibres.  Take the product with the fibres to obtain tori $T_{i}^\alpha, T^\beta_i$, again for $i=1,\dots,g$.  The torus $T_i^\alpha$ intersects the curve $\beta_i$ transversely in one point. Since $\beta_i$ is infinite order in $H_1(Y)$, by Poincar\'{e} duality $T_i^\alpha$ is nontrivial in $H_2(Y)$, and represents $\beta_i^* \in H_1(Y)^*$ under the isomorphisms $H_2(Y)\cong H^1(Y)\cong H_1(Y)^*$. The same holds with $\alpha$ and $\beta$ reversed, so $T_i^\beta$ is nontrivial in $H_2(Y)$ and represents $\alpha_i^*$.

Let $\Phi\colon Y\times[0,1]\to Y\times[0,1]$ be the diffeomorphism induced by the loop of diffeomorphisms of $Y$ that rotates the fibre. Identifying $Y\times [0,1]$ with a collar of this boundary component in $X$, extend $\Phi$ to a diffeomorphism $F\colon X\to X$ by the identity map. We have that $F_*=\Id_{H_2(X)}$, so $\Delta_F=S(\kappa)$ for some $\kappa\in\wedge^2{H_1(\partial X)^{\ast}}$. Since $F$ is supported in a collar of $Y$ it follows that $\kappa$ lies in the summand $\wedge^2{H_1(Y)^{\ast}}$ of $\wedge^2{H_1(\partial X)^{\ast}}$.

We now compute $\Delta_F$. Let $\gamma$ denote a regular fibre.   By \cite[Corollary~6.2]{MR741334}, $\gamma$ is infinite order in $H_1(Y)$ if and only if the Euler number of the Seifert fibration is zero.  Since $H_2(X,\partial X)\to H_1(\partial X)$ is surjective, for $i=1,\dots,g$ we may choose homologically essential, properly embedded surfaces $\Sigma_i\subseteq X$, for $i=1,\dots,2g$ with boundaries $\partial \Sigma_{2i-1} = \alpha_i$ and $\partial \Sigma_{2i} = \beta_i$ for $i=1,\dots,g$. If the Euler number of the Seifert fibred structure is zero then additionally let $\Sigma_{2g+1}$ be a homologically essential, properly embedded surface in $X$ with $\partial \Sigma_{2g+1} = \gamma$.

For each $i$, the surface $F(\Sigma_i)$ is the effect of replacing a boundary collar of $\Sigma_i$ with the track of the isotopy of $\partial \Sigma_i$  under the loop of diffeomorphisms defining $\Phi$. This collar isotopy rotates $h$, but fixes it setwise, and so $[F(\Sigma_{2g+1}) - \Sigma_{2g+1}] =0 \in H_2(X)$. The tracks of $\alpha_i$ and $\beta_i$ under this isotopy are the tori $T^\alpha_i$ and $T^\beta_i$ respectively. It follows that $[F(\Sigma_{2i-1}) - \Sigma_{2i-1}] = [T_i^{\alpha}]$ and $[F(\Sigma_{2i}) - \Sigma_{2i}] = [T_i^{\beta}]$. This gives $\Delta_F$.

It remains to deduce that $\kappa = \omega$.
Use that $T_i^\alpha$ is dual to $\beta_i^*$ and $T_i^\beta$ is dual to $\alpha_i^*$ to see that for $\kappa = \omega$ the composition $\wt{\kappa}$ in
\eqref{eqn:kappa-tilde} sends $\Sigma_{2i-1} \mapsto T^{\alpha}_i$ and $\Sigma_{2i} \mapsto T^\beta_i$ for $i=1,\dots,g$, and sends $\Sigma_{2g+1} \mapsto 0$. We thus compute that $\Delta_F = S(\kappa) = S(\omega)$ as desired.  Using injectivity of $S$, in fact, $\kappa = \omega$.
\end{proof}

\begin{remark}
Combining Proposition~\ref{prop:grabthetori} with Example~\ref{ex:differentstructures} we observe there exist examples of collar twisting maps that simultaneously realise nontrivial elements $\kappa\in\wedge^2H_1(\partial X)^*$ and $\Theta\in H^1(X,\partial X)$. As discussed in the introduction, it is an interesting problem to determine in general which combinations of $(\kappa,x) \in \wedge^2 H_1(\partial X)^* \times H^1(X,\partial X;\Z/2)$ can be realised smoothly, and also in the context of Section~\ref{subsection:relaxing-bdy-condition}, to determine how  $\pi_1 \Homeo^+(\partial X)$ acts on $\wedge^2 H_1(\partial X)^* \times H^1(X,\partial X;\Z/2)$.
\end{remark}

\appendix
\section{Dehn twists for non-spin smooth manifolds}
\label{sec:Auckly}

Let $M$ be a closed, connected, oriented, $\CAT$ $n$-manifold that decomposes into~$n$-manifolds as $M=N\cup_{\partial N}P$. Then there are fibre sequences~\cite{Lashof-embedding-spaces}, \cite[Corollary~3.1]{Hong-Kalliongis-McCullough-Rubinstein}:
\renewcommand{\theequation}{\thesection.\arabic{equation}}
\begin{equation}\label{eq:fibresequences}
\begin{tikzcd}[row sep = small]
\Diffeo^+(N,\partial N)\ar[r] &\Diffeo^+(M)\ar[r, "r"]&\Emb^{+}(P,M) &\text{when $\CAT=\Diff$;}\\
\Homeo^+(N,\partial N)\ar[r] &\Homeo^+(M)\ar[r,"r"]&\Emb^+_{\TOP}(P,M) &\text{when $\CAT=\TOP$.}
\end{tikzcd}
\end{equation}
Here, the space $\Emb^+_{\TOP}(P,M)$ (resp.~$\Emb^+(P,M)$) denotes the space of locally flat (resp.~smooth) orientation-preserving embeddings~$P\hookrightarrow M$. In each case, the map $r$ denotes the map that restricts a morphism to the copy of $P$ in the decomposition  $M= N\cup_{\partial N}P$.

We will use the second fibre sequence now, to justify a statement we used in the introduction.

\begin{proposition}
Let $X$ be a connected, oriented manifold. Then there is a fibre sequence
\[
\Homeo^+(X,\partial X)\to \Homeo^+(X)\xrightarrow{r'} \Homeo^+(\partial X),
\]
where $r'$ denotes the restriction map.
\end{proposition}

\begin{proof}
Certainly the fibre of the identity map is $\Homeo^+(X,\partial X)$.
We must prove that $r'$ has the homotopy lifting property. Let $F\colon Y\to \Homeo^+(X)$ be a map from a space $Y$ and let $H\colon Y\times I \to \Homeo^+(\partial X)$ be a homotopy starting at $r'\circ F$. Write $i\colon \partial X\to -X\cup_{\partial X}X$ for the inclusion map of the boundary, which is a separating codimension one submanifold.

Apply~\eqref{eq:fibresequences} when $M=-X\cup_{\partial X}X$ and $P = \partial X$ to see that $r\colon \Homeo^+(-X\cup_{\partial X}X)\to\Emb_{\TOP}^+(\partial X,-X\cup_{\partial X}X)$ is a Serre fibration, where $r(f) = f \circ i$.   Doubling a homeomorphism yields a map $D \colon \Homeo^+(X) \to \Homeo^+(-X\cup_{\partial X}X)$. Consider the following commuting diagram:
\[
\begin{tikzcd}%[row sep = small]
Y \ar[r,"F"] \ar[d,"{y \mapsto (y,0)}"'] & \Homeo^+(X) \ar[d,"{r'}"] \ar[r,"D"] & \Homeo^+(-X\cup_{\partial X}X) \ar[d,"r"] \\
Y \times I \ar[r,"H"] & \Homeo^+(\partial X) \ar[r,"i_*"] & \Emb^+_{\TOP}(\partial X,-X\cup_{\partial X}X).
\end{tikzcd}
\]
Let $H' := i_* \circ H \colon Y \times I \to \Emb^+_{\TOP}(\partial X,-X\cup_{\partial X}X)$ be the composition along the bottom row. Since $r$ satisfies the homotopy lifting property, we obtain a lift $\wt{H}' \colon Y \times I \to \Homeo^+(-X\cup_{\partial X}X)$ of $H'$ such that $r \circ \wt{H}' = H'$ and $\wt{H}'|_{Y \times \{0\}} = D\circ F$.

Now note that for all $(y,t) \in Y \times I$, we have that $r \circ \wt{H}'(y,t) = i \circ H(y,t) \in \Emb^+_{\TOP}(\partial X,-X\cup_{\partial X}X)$ sends $\partial X$ to $\partial X \subseteq -X\cup_{\partial X}X$. Since $\partial X$ is separating in $-X\cup_{\partial X}X$ and since $r \circ \wt{H}'(y,0) = D \circ F(y) \in \Homeo^+(-X\cup_{\partial X}X)$ preserves each copy of $X$, we have that $\wt{H}'(y,t)$ restricts to the desired lift $\wt{H}'' \colon Y \times I \to \Homeo^+(X)$.  Thus $r'$ is a Serre fibration as desired.
\end{proof}

The following theorem is already known to several experts and we decided to include it here so that it does not become folklore. We are grateful to Dave Auckly for informing us of the proof and allowing us to include it.
%, which  him, Peter Kronheimer, Tom Mrowka, and Danny Ruberman.

\begin{theorem}\label{thm:Auckly}
Let $n\geq 4$. Let $X$ be a connected, oriented, smooth $n$-manifold with boundary $\partial X \cong S^{n-1}\sqcup Y$ and let $\tau$ be a Dehn twist on a sphere parallel to the $S^{n-1}$ boundary component. If there exists a class $\sigma\in\pi_2(X)$ with $w_2(TX)[\sigma]\neq 0$, then $\tau$ is smoothly isotopic to the identity relative to $S^{n-1}$.
\end{theorem}

\begin{proof}
Let $\widehat{X} = X\cup_{S^{n-1}}-D^n$ and let $\Diffeo^+(\widehat{X},D^n)$ be the space of diffeomorphisms that restrict to the identity on $D^n$.
Consider the fibration $\Diffeo^+(\widehat{X},D^n)\to \Diffeo^+(\widehat{X}) \to \Emb^+(D^n,\widehat{X})$ of \eqref{eq:fibresequences} and the connecting homomorphism in the  associated long exact sequence in homotopy groups
\[\pi_1\Emb^+(D^n,\widehat{X})\xrightarrow{\partial} \pi_0\Diffeo^+(\widehat{X},D^n).\]
Write $\widehat{\tau}\in\Diffeo^+(\widehat{X})$ for the extension of $\tau$ to $\widehat{X}$ by the identity on $D^n$. The restriction map induces a homeomorphism $\Diffeo^+(\widehat{X},D^n)\cong\Diffeo^+(X,S^{n-1})$, so the twist $\tau\in \Diffeo^+(X,S^{n-1})$ is isotopic to the identity $\Id_{X}$ rel.~$S^{n-1}$ if and only if $[\widehat{\tau}]$ is trivial. We now aim to show this is the case.

Let $\varepsilon$ be the loop of self-diffeomorphisms of $D^n$ that is a full rotation in the first two coordinates and fixes the remaining coordinates.
%This determines a generator for  $\pi_1\Diffeo^+(D^n)=\Z/2$ since $n \geq 3$.
Write $i\colon D^n\to \widehat{X}$ for the inclusion map. Define a loop \[\gamma':=i\circ\varepsilon \in \pi_1\Emb^+(D^n,\widehat{X}).\]
Associated to the fibration \eqref{eq:fibresequences},  is an action of loops in $\Emb^+(D^n,\widehat{X})$ on $\Diffeo^+(\widehat{X},D^n)$; the connecting homomorphism $\partial$ is defined by looking at the effect of this action on the basepoint $\Id_{\widehat{X}}\in \Diffeo^+(\widehat{X},D^n)$. The effect of the action of $\gamma'$ on $\Id_{\widehat{X}}$ is a diffeomorphism in $\Diffeo^+(\widehat{X},D^n)$ isotopic to the Dehn twist $\widehat{\tau}$. In other words, $[\gamma']$ maps to~$[\widehat{\tau}]$ under the connecting map. We will show that $[\gamma']=0 \in \pi_1\Emb^+(D^n,\widehat{X})$, from which it will follow that  $[\widehat{\tau}]=0$, and hence $[\tau]\in \pi_0 \Diffeo^+(X,S^{n-1})$.

Write $\operatorname{Fr}^+(T\widehat{X})$ for the $\SO(n)$-principal bundle of oriented orthonormal frames on $T\widehat{X}$. Consider the map $\Emb^+(D^n,\widehat{X})\to \operatorname{Fr}^+(T\widehat{X})$, given by sending $f\in \Emb^+(D^n,\widehat{X})$ to the image of the standard frame on $T_0D^n$ under the derivative $Df \colon T_0D^n\to T_{f(0)}\widehat{X}$. This map determines a homotopy equivalence $\Emb^+(D^n,\widehat{X})\simeq \operatorname{Fr}^+(T\widehat{X})$.  Under this, $\gamma'$ is sent to a loop $\gamma$ in $\operatorname{Fr}^+(T\widehat{X})$ that is a full rotation in the first two coordinates of the frame above $0\in D^n\subseteq \widehat{X}$ and fixes the other coordinates. Consider the long exact sequence
\[
\cdots\to\pi_2(\widehat{X})\xrightarrow{\partial} \pi_1(\SO(n))\to\pi_1(\operatorname{Fr}^+(T\widehat{X}))\to\pi_1(\widehat{X})\to\cdots
\]
coming from the principal fibration.
The loop $[\gamma]$ is clearly in the image of $\pi_1(\SO(n))=\Z/2$ and is therefore trivial if there exists a class $\widehat{\sigma}\in\pi_2(\widehat{X})$ such that $\partial \widehat{\sigma}\neq 0$. We will show that this is the case.  Write $h\colon \pi_2(\widehat{X})\to H_2(\widehat{X})$ for the Hurewicz map and
$\ev\colon H^2(\widehat{X};\pi_1(\SO(n)))\to \Hom(H_2(\widehat{X}),\pi_1(\SO(n)))$
for the evaluation map.
Comparing the definition of the connecting homomorphism $\partial$ with the interpretation of $w_2$ as the obstruction to reducing the structure group of $\operatorname{Fr}^+(T\widehat{X})$ to $\Spin(n)$, we have $\ev(w_2(T\widehat{X}))\circ h=\partial$. The inclusion $i\colon X\to \widehat{X}$ induces an isomorphism on $\pi_2$, and we define $\widehat{\sigma}:=(i_*)^{-1}(\sigma)$. Then
\[
\partial\widehat{\sigma}=(\ev(w_2(T\widehat{X}))\circ h)((i_*)^{-1}(\sigma))=i^*(\ev(w_2(T\widehat{X}))\circ h)(\sigma)=w_2(TX)[\sigma]\neq 0.
\]
Therefore $\widehat{\sigma}$ is the desired element of $\pi_2(\widehat{X})$, so  $[\gamma]=0 \in \pi_1\operatorname{Fr}^+(T\widehat{X})$, and it follows that $[\gamma'] =0 \in \pi_1 \Emb^+(D^n,\widehat{X})$.
Since $\partial [\gamma'] = [\widehat{\tau}] \in \pi_0\Diffeo^+(\widehat{X},D^n)$, which in turn maps to $\tau$ under the isomorphism $\pi_0\Diffeo^+(\widehat{X},D^n)\cong\pi_0\Diffeo^+(X,S^{n-1})$. This completes the proof.
\end{proof}

\begin{remark}
We have stated Theorem~\ref{thm:Auckly} with the assumption $n\geq 4$ because for $n=1, 2, 3$, the statement is vacuous. When $n=1, 2$, this is because $\pi_2(X)=0$. When $n=3$, consider a spherical class $\sigma\colon S^2\to X$ and observe that $w_2(TX)[\sigma]=w_2(\sigma^*TX) \in\Z/2$.
 We may assume $\sigma$ is represented by an immersion, so that $\sigma^*TX\cong TS^2\oplus \nu({\sigma(S^2)\subseteq X})$.
We compute that $w_2(\sigma^*TX) = w_2(TS^2 \oplus \nu \sigma) = w_2(TS^2) + w_1(TS^2)w_1(\nu\sigma) + w_2(\nu \sigma)$. Since $\nu \sigma$ is a line bundle, the last term vanishes, while the first two terms vanish because $TS^2$ is stably trivial.  Hence $w_2(TX)[\sigma]=0$ for all spherical classes when~$n=3$.
\end{remark}

\begin{corollary}
Let $n\geq 4$. Let $X$ be a simply connected, oriented, smooth manifold with boundary $\partial X \cong S^{n-1}\sqcup M$ and let $\tau$ be a Dehn twist on a sphere parallel to the $S^{n-1}$ boundary component. If $X$ is non-spin, then $\tau$ is smoothly isotopic to the identity map, relative to~$S^{n-1}$.
\end{corollary}

\begin{proof}
For a simply connected manifold, the Hurewicz map determines an isomorphism $\pi_2(X)\otimes\Z/2\cong H_2(X;\Z/2)$. For a non-spin manifold, $w_2(TX)$ is nontrivial. Hence there must exist a homology class $x\in H_2(X;\Z/2)$ such that $w_2(TX)[\sigma]\neq 0$. The result now follows immediately from Theorem~\ref{thm:Auckly}.
\end{proof}

\bibliographystyle{alpha}
\bibliography{pseudoisotopy}
\end{document}